\documentclass[preprint,authoryear,12pt]{imsart}

\RequirePackage{natbib,amssymb,verbatim}
\RequirePackage{amsthm,amsmath,natbib,mathrsfs}
\RequirePackage{hypernat,mathabx}
\usepackage{xr}
\usepackage{hyperref}
 
\usepackage{color}
\usepackage{etoolbox}
\usepackage{graphicx}
\usepackage[title]{appendix}
\usepackage{enumitem}

\makeatletter
\def\namedlabel#1#2{\begingroup
    #2%
    \def\@currentlabel{#2}%
    \phantomsection\label{#1}\endgroup
}
\makeatother
\makeatletter

\patchcmd{\NAT@test}{\else \NAT@nm}{\else \NAT@nmfmt{\NAT@nm}}{}{}

\DeclareRobustCommand\citepos
  {\begingroup
   \let\NAT@nmfmt\NAT@posfmt
   \NAT@swafalse\let\NAT@ctype\z@\NAT@partrue
   \@ifstar{\NAT@fulltrue\NAT@citetp}{\NAT@fullfalse\NAT@citetp}}

\let\NAT@orig@nmfmt\NAT@nmfmt
\def\NAT@posfmt#1{\NAT@orig@nmfmt{#1's}}

\makeatother

\startlocaldefs
\newtheorem{thm}{Theorem}
\newtheorem{lemma}{Lemma}
\newtheorem{prop}{Proposition}
\newcommand\numberthis{\addtocounter{equation}{1}\tag{\theequation}}
\endlocaldefs

\def\1{\mathbf{1}}

\def\argmax{\operatorname{argmax} \displaylimits}
\def\b{\beta}

\def\bb{\boldsymbol{\beta}}
\def\be{\boldsymbol{\epsilon}}
\def\bm{\boldsymbol{\mu}}
\def\bth{\boldsymbol{\theta}}

\def\bz{\boldsymbol{\zeta}}

\def\C{\mathcal{C}}
\def\Cov{\mathrm{Cov}}

\def\diag{\mathrm{diag}}
\def\e{\epsilon}
\def\E{\mathbb{E}}
\def\g{\gamma}
\def\HS{\mathrm{HS}}
\def\I{\mathcal{I}}
\def\i{\underline{i}}
\def\J{J_0}
\def\l{\lambda}
\def\L{\Lambda}
\def\P{\mathbb{P}}
\def\q{\mathbf{q}}
\def\r{\mathbf{r}}
\def\R{\mathbb{R}}
\def\s{\sigma}

\def\tr{\mathrm{tr}}
\def\T{\top}

\def\u{\mathbf{u}}
\def\U{\mathcal{U}}
\def\vfrak{\mathfrak{v}}
\def\Var{\mathrm{Var}}
\def\w{\mathbf{w}}
\def\W{\mathbf{W}}
\def\x{\mathbf{x}}
\def\y{\mathbf{y}}
\def\z{\mathbf{z}}

\newcommand{\commentout}[1]{}

\begin{document}
\allowdisplaybreaks
\bibliographystyle{ims}
\begin{frontmatter}
\title{Flexible results for quadratic forms with applications to
  variance components estimation \protect}

\begin{aug}
\author{\fnms{Lee H.} \snm{Dicker}\thanksref{t1}
\ead[label=e1]{ldicker@stat.rutgers.edu}}
\and \author{\fnms{Murat A.} \snm{Erdogdu}\ead[label=e2]{erdogdu@stanford.edu}}

\thankstext{t1}{Supported by NSF Grants DMS-1208785 and DMS-1454817}
\runauthor{L.H. Dicker and M.A Erdogdu}

\affiliation{Rutgers University\thanksmark{m1} and Stanford University\thanksmark{m2}}

\address{Department of Statistics and Biostatistics \\ Rutgers University \\ 501 Hill Center \\ Piscataway, NJ 08854 \\
\printead{e1}}
\address{Statistics Department \\ Stanford
  University \\ Sequoia Hall \\ Stanford, CA 94305 \\
\printead{e2}}
\end{aug}

\runtitle{Quadratic forms and variance components estimation}



\begin{abstract}
We derive convenient uniform concentration bounds and finite sample
multivariate normal
approximation results for quadratic forms, then describe some
applications involving variance components estimation in linear
random-effects models.  Random-effects models and variance components estimation are classical
topics in statistics, with a corresponding well-established asymptotic
theory.  However, 
our finite sample results for quadratic forms provide
additional flexibility for easily analyzing random-effects models in
non-standard settings, which are becoming more important in modern
applications (e.g. genomics).  For instance, in addition to deriving novel non-asymptotic
bounds for variance components estimators in classical linear random-effects models, we provide a concentration bound for
variance components estimators in linear models with correlated random-effects.
Our general concentration bound is a uniform version of the
Hanson-Wright inequality.  The main normal
approximation result in the paper is derived using
\citepos{reinert2009multivariate} embedding technique  and multivariate
Stein's method with exchangeable pairs.  
\end{abstract}

\end{frontmatter}

\section{Introduction} \label{sec:intro}

Suppose that $\bz = (\zeta_1,\ldots,\zeta_d)^{\T} \in \R^d$ is a
random vector with independent components satisfying $\E(\zeta_j) = 0$, $j =
1,\ldots,d$.  Additionally, let $Q$ be a $d \times d$ positive
semidefinite matrix with real (non-random) entries.  Quadratic forms $\bz^{\T}Q\bz$ have been studied for decades in statistics \citep{sevastyanov1961class,whittle1964convergence,hanson1971bound,hall1984central,gotze1999asymptotic, gotze2002asymptotic,chatterjee2008new}.  This
paper is largely motivated by recent applications involving random-effects
models, which rely heavily on properties of quadratic
forms
\citep[e.g.][]{jiang2014high,campos2015genomic,dicker2015efficient}.
In the first part of the paper, we give two new finite sample bounds
for quadratic forms --- a uniform concentration inequality
(Theorem \ref{thm:QF_conc})
and a normal approximation result (Theorem \ref{thm:QF_napprox}) --- which may be useful in a variety of statistical
applications.  The second part of the paper focuses on applications of
Theorems \ref{thm:QF_conc}--\ref{thm:QF_napprox}
related to variance components estimation in linear random-effects
models, including non-standard models with correlated random effects
(cf. Proposition \ref{prop:RE_dconc}).  

Theorems \ref{thm:QF_conc} and \ref{thm:QF_napprox} are the main
theoretical results in the paper.  We rate the novelty of our normal
approximation result Theorem
\ref{thm:QF_napprox}, which is a multivariate normal approximation
results proved via Stein's method for exchangeable pairs, higher than that of Theorem \ref{thm:QF_conc}.
However, the main emphasis of both results is convenience for use in
applications.  

Our concentration bound, Theorem \ref{thm:QF_conc}, is a uniform version of the
Hanson-Wright inequality for quadratic forms.  The method of proof for Theorem
\ref{thm:QF_conc} is relatively standard -- combining a chaining
argument from empirical process theory \citep[e.g. Chapter 3
of][]{vandegeer2000empirical}  with
the pointwise-bound of the original Hanson-Wright inequality -- and it
should be possible to generalize our result to larger classes of quadratic forms,
similar to \citep{adamczak2014note}.  However, we note that while Theorem
\ref{thm:QF_conc} is restricted to relatively simple (Lipschitz) classes of
quadratic forms, it is not a corollary of the uniform bounds in \citep{adamczak2014note}, which require a stronger condition on the distribution of
$\bz$ (see the comments in Section \ref{sec:QF_conc} following the
statement of Theorem \ref{thm:QF_conc}).  

Theorem \ref{thm:QF_napprox} is a normal approximation
result for vectors of quadratic forms.  Most of the existing normal approximation
results for quadratic forms are asymptotic results
\citep{whittle1964convergence, hall1984central,jiang1996REML}, require
the random variables $\zeta_i$ to be iid   
\citep{hall1984central,gotze1999asymptotic,gotze2002asymptotic, chatterjee2008new}, or
have other limitations \citep{sevastyanov1961class}. Theorem \ref{thm:QF_napprox} gives a non-asymptotic normal
approximation bound, which applies
to $\bz$ with independent (but not necessarily identically
distributed) sub-Gaussian components. Furthermore, in contrast with
most existing results on quadratic forms, which are predominantly univariate, Theorem
\ref{thm:QF_napprox} is a multivariate result, which applies to
vectors of quadratic forms $(\bz^{\T}Q_1\bz,\ldots,\bz^{\T}Q_K\bz)$,
for positive semidefinite matrices $Q_1,\ldots,Q_K$ (the applications
to random-effects models considered in Section
\ref{sec:RE} require $K =2$).  The proof of Theorem
\ref{thm:QF_napprox} relies on Stein's method of exchangeable pairs
and follows the embedding approach of \cite{reinert2009multivariate}.
Theorem \ref{thm:QF_napprox} and its proof shares similarities
with Proposition 3.1 of \cite{chatterjee2008new}.  However,
Proposition 3.1 of \cite{chatterjee2008new} 
applies only to a single quadratic form $\bz^{\T}Q\bz$ in iid Rademacher random variables
$\zeta_i$ satisfying $\P(\zeta_i = 1) = \P(\zeta_i = -1) = 1/2$.  

Linear random-effects models are studied in Section \ref{sec:RE}. Asymptotic results for quadratic forms serve as the
theoretical underpinning for many applications involving
random-effects models \citep{hartley1967maximum,
  jiang1996REML, jiang1998asymptotic}.  However, new applications of
random-effects models in genomics have pushed the boundaries of
existing theoretical results
\citep{yang2010common,golan2011accurate,speed2012improved,zaitlen2012heritability,jiang2014high,yang2014advantages,campos2015genomic}.
In Section \ref{sec:RE}, we present new non-asymptotic bounds for
variance components estimators in linear random-effects models.  To
our knowledge, these are the first finite sample results on the
statistical properties of variance components estimators.
Many
now-classical asymptotic results for random-effects models
\citep[e.g.][]{jiang1996REML}
follow as corollaries of our finite sample results in Section \ref{sec:RE}.
More significantly, non-asymptotic bounds, like those in this paper, provide increased
flexibility for use in applications.  
In particular, our results can be easily
applied in non-standard settings, where it is less clear how to adapt
the existing asymptotic theory; see, for example Proposition
\ref{prop:RE_dconc}, which applies to
 random-effects models with correlated random-effects, and
\citep{dicker2015efficient} for an application involving fixed-effects
models.  

The rest of the paper proceeds as follows.  Some basic notation is
introduced in Section \ref{sec:notation}.  The main results are stated in Section \ref{sec:QF}.
Linear random-effects models are studied in Section \ref{sec:RE}.  The proofs of Theorems \ref{thm:QF_conc}--\ref{thm:QF_napprox}  and Propositions \ref{prop:RE_conc} and
\ref{prop:RE_napprox} are contained in the Appendices; other results
are proved in the Supplementary Material.  

\section{Notation}\label{sec:notation}

If $\u =(u_1,\ldots,u_p)^{\T}\in \R^p$, then $\Vert \u\Vert =
(u_1^2+\cdots + u_p^2)^{1/2}$ is its Euclidean norm.  For a $d \times m$ matrix $A = (a_{ij})$ with real entries, let $\Vert A \Vert = \sup_{\Vert
  \x\Vert = 1} \Vert A\x\Vert$ and $\Vert A
\Vert_{\HS} = \left\{\sum_{i = 1}^d \sum_{j=1}^m
  a_{ij}^2\right\}^{1/2}$ be the operator norm and the
Hilbert-Schmidt (Frobenius) norm of $A$, respectively.  If $f: \R^m \to \R$ is a function with $k$-th order derivatives, define
\[
\vert f\vert_j = \sup_{\substack{\x \in \R^m, \\  1\leq i_1,\ldots,i_j \leq m}}
\left\vert\frac{\partial^j}{\partial x_{i_1} \cdots \partial x_{i_j}}
  f(\x)\right\vert, \ \ j= 1,\ldots,k,
\]
and let $\vert f\vert_0 = \sup_{\x \in \R^m} \vert f(\x)\vert$.  
Additionally, define $\mathcal{C}_b^k(\R^m) = \{f: \R^m \to \R; \ \vert f \vert_j <
\infty, \ j = 0,1,\ldots, k\}$ to be the class of real-valued
functions on $\R^m$ with bounded derivatives up to order $k$.  
Finally, following \citep{vershynin2010introduction},
let $\Vert \zeta \Vert_{\psi_2} =
\sup_{r \geq1} r^{-1/2}\E(\vert \zeta \vert^r)^{1/r}$ be the
sub-Gaussian norm of the real-valued random variable $\zeta$.  

\section{Results for quadratic forms}\label{sec:QF}

\subsection{Uniform concentration bound}\label{sec:QF_conc}

The Hanson-Wright inequality is a classical probabilistic bound for
quadratic forms, which has been the subject of renewed attention
recently in applications related to random matrix theory
\citep[e.g.][]{hsu2012tail,rudelson2013hanson,adamczak2014note}.  Theorem
\ref{thm:QF_conc} is a uniform version of the Hanson-Wright
inequality, which applies to families of quadratic forms $\bz^{\T}Q(\u)\bz$, where 
$Q(\u)$ is a matrix function of $\u \in \R^K$.  As
illustrated in Section \ref{sec:RE}, Theorem \ref{thm:QF_conc} has
applications in the analysis of random-effects models; more broadly,
it has applications in 
$M$-estimation and maximum likelihood problems with non-iid data.  

\begin{thm}\label{thm:QF_conc}
Let $0 < R < \infty$ and let $t_1(\u),\ldots,t_m(\u)$ be real-valued Lipschitz
functions on $[0,R]^K \subseteq \R^K$, satisfying
\begin{equation}\label{lipschitz}
\max_{i = 1,\ldots,m} \vert t_i(\u) - t_i(\u')\vert \leq L\Vert \u
- \u'\Vert, \ \ \u,\u' \in [0,R]^K,
\end{equation}
for some constant $0 < L  < \infty$.  Let $T(\u) =
\diag\{t_1(\u),\ldots,t_m(\u)\}$, let $V$ be an $d \times m$
matrix, and define $Q(\u) = V T(\u) V^{\T}$. 
Additionally, let $\bz = (\zeta_1,\ldots,\zeta_d)^{\T} \in \R^d$, where
$\zeta_1,\ldots,\zeta_d$ are independent mean 0 sub-Gaussian random variables satisfying
\begin{equation}\label{subgauss}
\max_{i = 1,\ldots,d} \Vert \zeta_i \Vert_{\psi_2}
\leq \g
\end{equation}
for some constant $\g \in (0,\infty)$.
Then there
exists an absolute constant $C \in (0,\infty)$ such that 
\begin{align*}
\P\Bigg[ &\sup_{\u \in [0,R]^K} \vert \bz^{\T}Q(\u)\bz -
  \E\{\bz^{\T}Q(\u)\bz\}\vert > r\Bigg] \\ & \leq
  C\exp\left[-\frac{1}{C}\min\left\{\frac{r^2}{\g^4\Vert V^{\T}V\Vert^2(\Vert T(0)
  \Vert_{\HS}^2 +K L^2R^2m)},\frac{r}{\g^2\Vert V^{\T}V\Vert(\Vert T(0)\Vert + K^{1/2}LR)}\right\}\right],
\end{align*}
whenever $r^2 \geq C \g^4\Vert V^{\T}V\Vert^2K^3 L^2R^2m$.
\end{thm}

Theorem \ref{thm:QF_conc} is proved in Appendix \ref{app:QF_conc}.  In
a typical application, the dimension $K$ will be small (in Section
\ref{sec:RE}, we use Theorem \ref{thm:QF_conc} with $K=1$) and $m,d$
may be large. A
uniform Hanson-Wright inequality, with a similar upper bound, is also given in \citep{adamczak2014note}.  However, Adamczak's result
applies to random vectors satisfying a relatively strong concentration
property and does not cover sub-Gaussian random vectors satisfying
only \eqref{subgauss}; see Remark 4 following Theorem 2.3 in
\citep{adamczak2014note}.  

\subsection{Normal approximation}\label{sec:QF_napprox}

The main result of this section is Theorem \ref{thm:QF_napprox}, a multivariate normal approximation
result for vectors of quadratic forms
$(\bz^{\T}Q_1\bz,\ldots,\bz^{\T}Q_K\bz)^{\T} \in \R^K$.  Theorem \ref{thm:QF_napprox} may be viewed as a generalization of Proposition 3.1 in 
\citep{chatterjee2008new}, which applies to a single quadratic for $\bz^{\T}Q\bz$ in Rademacher random variables $\zeta_j$
satisfying $\P(\zeta_j = \pm 1) = 1/2$ (though our bound in Theorem
\ref{thm:QF_napprox} is not as tight as Chatterjee's; see the
discussion after the statement of the theorem).  A proof of Theorem
\ref{thm:QF_napprox} may be found in Appendix
\ref{app:QF_napprox}.  The proof is based on Stein's method of
exchangeable pairs and the embedding technique from \citep{reinert2009multivariate}.  

\begin{thm}\label{thm:QF_napprox}
Let $\zeta_1,\ldots,\zeta_d$ be independent sub-Gaussian random
variables with mean 0 and variance 1, and assume that they satisfy
\eqref{subgauss}.  Let $\bz = (\zeta_1,\ldots,\zeta_d)^{\T} \in
\R^d$.   Additionally, for $k = 1,\ldots,K$, let $Q_k= (q_{ij}^{(k)})$
be an $d \times d$
positive semidefinite matrix and let $\widecheck{Q}_k =
\diag(q_{11}^{(k)},\ldots,q_{dd}^{(k)})$.  Define $w_k = \bz^{\T}Q_k\bz - \tr(Q_k)$, 
$\widecheck{w}_k = \bz^{\T}\widecheck{Q}_k\bz - \tr(Q_k)$, and
\[
\w = \left[\begin{array}{c} w_1 \\ \widecheck{w}_1 \\ \vdots \\ w_K \\
             \widecheck{w}_K \end{array}\right] = \left[\begin{array}{c}
                                              \bz^{\T}Q_1\bz -
                                              \tr(Q_1) \\
                                              \bz^{\T}\widecheck{Q}_1\bz
                                              - \tr(Q_1) \\ \vdots \\
                                              \bz^{\T}Q_K\bz -
                                              \tr(Q_K) \\
                                              \bz^{\T}\widecheck{Q}_K\bz
                                              - \tr(Q_K)\end{array}\right]
                                              \in
\R^{2K}
\]
Finally, let $\z \sim N(0,I_{2K})$ and $V = \Cov(\w)$.  
There is an absolute constant $0 < C< \infty$ such that
\begin{align}\label{bd_napprox}
\big|\E \{f(\w)\} &- \E\{f(V^{1/2}\z)\}\big| \\ \nonumber & \leq
  C (\g + 1)^8\left\{K^{3/2}d^{1/2}\vert
  f\vert_2\left(\max_{k = 1,\ldots,K}\Vert Q_k\Vert\right)^2 + K^3d\vert f
  \vert_3\left(\max_{k = 1,\ldots,K}\Vert Q_k\Vert\right)^3\right\},
\end{align}
for all three-times differentiable functions $f:\R^{2K} \to \R$.    
\end{thm}

The upper-bound \eqref{bd_napprox} does not appear to be optimal; cf.
Section 5 of \citep{jiang1996REML} and Section 3 of
\citep{chatterjee2008new}, where conditions for convergence depend on
the ratios $\Vert Q_k - \widecheck{Q}_k\Vert^2/\Var(\bz^{\T}Q_k\bz)$ and $\tr(Q_k^4)/$ $\Var(\bz^{\T}Q_k\bz)^2$, respectively. However, it is likely
that \eqref{bd_napprox} can be improved by carefully examining the
proof in the Appendix, if one is willing to accept a more complex (and potentially less
user-friendly) bound. 
Moreover, we argue presently that the bound \eqref{bd_napprox} is
already effective in a range of
practical settings.  Assume that in addition to the conditions of Theorem
\ref{thm:QF_napprox}, the $\zeta_i$ are iid with excess kurtosis 
$\g_2 = \E(\zeta_i) - 3 \geq -2$.  Also, let $\s_k^2 = \Var(\bz^{\T}Q_k\bz)$. By Lemma \ref{lemma:vqf} from the Supplementary Material, 
\[
\s_k^2 = 2\tr(Q_k^2) + \g_2\tr(\widecheck{Q}^2_k) \geq (2 +
\g_2)\tr(Q_k^2).  
\] 
Hence, the upper-bound in Theorem
\ref{thm:QF_napprox} implies that $\bz^{\T}Q_k\bz/\s_k$ is
asymptotically $N(0,1)$, if 
\[
\frac{d^{1/2}\Vert Q_k\Vert^2}{\s_k^2} + \frac{d\Vert
  Q_k\Vert^3}{\s_k^3} \leq \frac{d^{1/2}\Vert Q_k\Vert^2}{(2+\g_2)\tr(Q_k^2)} + \left\{\frac{d^{2/3}\Vert
  Q_k\Vert^2}{(2+\g_2)\tr(Q_k^2)}\right\}^{3/2}\to 0.
\]
We conclude that if (i) $\liminf \g_2 > -2$ and (ii) $\Vert
Q_k\Vert^2/\tr(Q_k^2) = o(d^{-2/3})$, then $\bz^{\T}Q_k\bz/\s_k$ is
asymptotically $N(0,1)$.  Regarding (i), note that $\g_2 > -2$ for all distributions
except the Rademacher distribution; furthermore, (ii) holds if, for instance, all of the
eigenvalues of $Q_k$ are contained in a compact subset of
$(0,\infty)$.

\section{Linear random-effects models} \label{sec:RE}

In this section, we apply the results from Section \ref{sec:QF} to the
variance components estimation problem in a linear random-effects
model.  We assume that
\begin{equation}\label{LM}
\y = X\bb + \be, 
\end{equation}
where $\y = (y_1,\ldots,y_n)^{\T} \in \R^n$ is an observed $n$-dimensional outcome
vector, $X = (x_{ij})$ is an observed $n \times p$ predictor matrix
with $\x_{ij} \in \R^p$, $\bb = (\b_1,\ldots,\b_p)^{\T}\in \R^p$ is an unknown $p$-dimensional vector, and $\be = (\e_1,\ldots,\e_n)^{\T} \in\R^n$ is an unobserved
error vector.  We further assume that $\b_1,\ldots,\b_p,\e_1,\ldots, \e_n$ are independent random
variables with 
\begin{equation}\label{RE}
\begin{array}{l}
\E(\e_i) = 0 \mbox{ and }\Var(\e_i) = \s_0^2, \ \ i = 1,\ldots,n,
  \\
\E(\b_j) = 0 \mbox{ and } \Var(\b_j) = \dfrac{\s_0^2\eta_0^2}{p}, \ \ j
  = 1,\ldots,p.
\end{array}
\end{equation}
Here, we  assume that the $\b_j$ are all independent. In Section
\ref{sec:RE_dconc}, we investigate a more general model with
dependent random-effects and give a corresponding concentration bound.
Throughout, we also assume that $X$ is independent
of $\be$ and $\bb$.  
Overall,
\eqref{LM}--\eqref{RE} is a linear random-effects model with variance components parameters
$\bth_0 = (\s_0^2,\eta_0^2)$.  Observe that we have parametrized the model so that
$\eta_0^2$ is a measure of the signal-to-noise ratio; this
parametrization is standard \citep[e.g.][]{hartley1967maximum}.

Let $\bth = (\s^2,\eta^2)$ and define the Gaussian data
log-likelihood, 
\[
\ell(\bth) = - \frac{1}{2}\log (\s^2)-\frac{1}{2n}\log \det\left(\frac{\eta^2}{p}XX^{\T} +
   I\right) - \frac{1}{2\s^2n}\y^{\T}\left(\frac{\eta^2}{p}XX^{\T} +
   I\right)^{-1}\y .
\]
Note that $\ell(\bth)$ is the log-likelihood for $\bth$, if $\bb \sim
N\{0,(\eta_0^2\s_0^2/p)I\}$ and $\be \sim N(0,\s_0^2I)$ are Gaussian.
In this section, we study properties of the
maximum likelihood estimator (MLE),
\begin{equation}\label{MLE}
\hat{\bth} = (\hat{\s}^2,\hat{\eta}^2) = \argmax_{\s^2,\eta^2 \geq 0} \ell(\bth),
\end{equation}
in settings where $\be$ and $\bb$ are not necessarily Gaussian. 
[N.B. if $\ell(\bth)$ in \eqref{MLE} has multiple maximizers, then use any
pre-determined rule to select $\hat{\bth}$.]

The estimator $\hat{\bth}$ has already been widely studied in the
literature, even in settings where $\be$ and $\bb$ are not Gaussian
\citep[e.g.][]{richardson1994asymptotic, jiang1996REML}.
In practice, $\hat{\bth}$ and other closely related estimators, such
as REML estimators, are probably the most commonly used variance
components estimators for linear random-effects models
\citep{harville1977maximum,searle1992variance,demidenko2013mixed}.
\citepos{jiang1996REML} work is especially relevant for the results in
this section.  Jiang studied models with independent random-effects
and derived general consistency and asymptotic normality results for
$\hat{\bth}$ that are valid in some of the settings considered here.
However, asymptotic results tend to have more limited flexibility for
use in certain applications. This
has become more notable recently, with the widespread use of
random-effects models in genomic and
other 
applications, as discussed in Section \ref{sec:intro}.  

In Sections
\ref{sec:RE_conc}--\ref{sec:RE_napprox} below, we present finite
sample concentration and
normal approximation bounds for $\hat{\bth}$, which follow from
Theorem \ref{thm:QF_conc} and \ref{thm:QF_napprox}, respectively.
These bounds have not been optimized and some of the quantities in
the bounds can be extremely large for given values of $\bth_0$ and $p^{-1}XX^{\T}$ [e.g. $\kappa(\eta_0^2,\s_0^2,\L)^{-1}$ and
$\nu(\eta_0^2,\s_0^2,\L)$, defined in \eqref{kappa} and \eqref{nu}
below].  However, as described in the text below, Propositions \ref{prop:RE_conc} and \ref{prop:RE_napprox} still yield the ``correct''
asymptotic conclusions, similar to \citep{jiang1996REML}, which ensure consistency and asymptotic
normality of $\hat{\bth}$, if $p/n \to \rho \in (0,\infty)$ and the model
parameters are bounded.  Though it may be of interest to further
optimize Propositions
\ref{prop:RE_conc}--\ref{prop:RE_napprox} (and it is almost certainly
possible), our main emphasis is that the non-asymptotic approach taken
here provides additional flexibility for deriving  and understanding results
in less standard settings.   For instance, while Propositions
\ref{prop:RE_conc} and \ref{prop:RE_napprox} parallel existing results in \citep{jiang1996REML},  Proposition
\ref{prop:RE_dconc} is a concentration bound for linear
models with correlated random-effects and appears to be more novel [an
application of Proposition \ref{prop:RE_dconc} may be found in
\citep{dicker2015efficient}].

\subsection{Additional notation}

It is convenient to introduce some notation relating to the spectrum
of $X$.  Let $\l_1 \geq \cdots \l_n \geq 0$ be the eigenvalues of
$p^{-1}XX^{\T}$ and suppose that $p^{-1}XX^{\T} = U\L U^{\T}$ is the
eigen-decomposition of $p^{-1}XX^{\T}$, where
$\L = \diag(\l_1,\ldots,\l_n)$ and
$U$ is an $n \times n$ orthogonal matrix.  Let $n_0 = \max\{i; \ \l_i
> 0\}$ and define the empirical variance of the eigenvalues of $p^{-1}XX^{\T}$, 
\begin{equation}\label{vfrak}
\vfrak(\L) = \frac{1}{n}\sum_{i = 1}^n \l_i^2 - \left(\frac{1}{n}\sum_{i
    = 1}^n \l_i\right)^2 =
\frac{1}{n}\tr\left\{\left(\frac{1}{p}XX^{\T}\right)^2\right\} - \left\{\frac{1}{n}\tr\left(\frac{1}{p}XX^{\T}\right)\right\}^2.
\end{equation} 
\subsection{Concentration bound}\label{sec:RE_conc}

To derive a concentration bound for $\hat{\bth}$ (Proposition
\ref{prop:RE_conc} below), we follow standard steps in the
analysis of variance components estimators
\citep{hartley1967maximum}.  In
particular, we introduce the profile likelihood and other
related objects, which essentially reduce the bivariate optimization
problem \eqref{MLE} to a univariate problem.  Basic calculus implies that if $\eta^2 \geq 0$, then
\[
\max_{\s^2,\eta^2 \geq 0} \ell(\s^2,\eta^2) =
\max_{\eta^2 \geq 0} \ell_*(\eta^2),
\]
where $\ell_*(\eta^2) = \ell\{\s_*^2(\eta^2),\eta^2\}$ is called the profile
likelihood and
\[
\s_*^2(\eta^2) = \frac{1}{n}\y^{\T}\left(\frac{\eta^2}{p}XX^{\T} + I\right)^{-1}\y.
\]
It follows that $\hat{\bth} =
(\s_*^2(\hat{\eta}^2),\hat{\eta}^2)$, where 
\begin{equation}\label{profile}
\hat{\eta}^2 = \argmax_{\eta^2 \geq 0} \ell_*(\eta^2) .
\end{equation}
The proof
of Proposition \ref{prop:RE_conc} hinges on comparing the profile
likelihood $\ell_*(\eta^2)$ to
its population version,
\[
\ell_0(\eta^2)  = -\frac{1}{2}\log\{\s_0^2(\eta^2) \} -
  \frac{1}{2n}\log\det\left(\frac{\eta^2}{p}XX^{\T} + I\right) -
  \frac{1}{2}\log(\s_0^2) - \frac{1}{2},
\]
where we have replaced $\s_*^2(\eta^2)$ in $\ell_*^2(\eta^2)$ with its
expectation, 
\[
\s_0^2(\eta^2) = \E\{\s_*^2(\eta^2)\vert X\}=\frac{\s_0^2}{n}\tr\left\{\left(\frac{\eta_0^2}{p}XX^{\T}
                 + I\right)\left(\frac{\eta^2}{p}XX^{\T} +
                 I\right)^{-1}\right\} = \frac{\s_0^2}{n}\sum_{i =
               1}^n \frac{\eta_0^2\l_i+1}{\eta^2\l_i + 1}.
\]
Observe that $\s_0^2(\eta_0^2) = \s_0^2$.  

Overall, our strategy for proving Proposition \ref{prop:RE_conc} mirrors the classical parametric theory for consistency of
maximum likelihood and
$M$-estimators \citep[e.g. Chapter 5 of][]{vandervaart2000asymptotic},
except that we employ Theorem \ref{thm:QF_conc} at several key steps.
As in the standard analysis, two important facts underlying Proposition \ref{prop:RE_conc}   are (i) $\eta^2_0$ is the unique maximizer
of $\ell_0(\eta^2)$ and (ii) $\ell_*(\eta^2) \approx \ell_0(\eta^2)$,
when $n,p$ are large.  Theorem \ref{thm:QF_conc} is used to make
the approximation $\ell_*(\eta^2) \approx \ell_0(\eta^2)$ more
precise.  It should not be surprising that quadratic forms play an
important role in the analysis, given the dependence of
$\ell_*(\eta^2) = \ell\{\s_*^2(\eta^2),\eta^2\}$ on the quadratic form
$\s_*^2(\eta^2)$.  We emphasize that to prove
Proposition \ref{prop:RE_conc}, we use Theorem \ref{thm:QF_conc} with
$K=1$ and $\u = \eta^2$; the general version of Theorem \ref{thm:QF_conc} with
matrix functions defined on $\R^K$ may be useful for studying
random-effects model with $K$-groups of random-effects, e.g. the
general linear random-effects model considered in \citep{jiang1996REML}.  Proposition \ref{prop:RE_conc} is proved in
Appendix \ref{app:RE_conc}.

\begin{prop}\label{prop:RE_conc}
Assume that the linear random-effects model \eqref{LM}--\eqref{RE} holds and
that $\b_1,\ldots,\b_p$, $\e_1,\ldots,\e_n$
are independent sub-Gaussian random variables satisfying
\begin{equation}\label{subgbd}
\max\left\{\Vert p^{1/2}\b_j\Vert_{\psi_2}, \ \Vert \e_i\Vert_{\psi_2};
  \ i = 1,\ldots,n, \ j = 1,\ldots, p\right\} \leq \g
\end{equation}
for some $0 < \g < \infty$.  
Finally, define
\begin{equation}\label{kappa}
\kappa(\s_0^2,\eta_0^2,\L) =
\frac{\s_0^4\eta_0^8\vfrak(\L)^2}{(\s_0^2+1)^5(\eta_0^2+1)^{12}
  (\l_1+1)^{18}(\l_{n_0}^{-1}+1)^8\{\vfrak(\L)+1\}^2}. 
\end{equation}  
\begin{itemize}
\item[(a)] Suppose that $n_0 = n$.  There is an absolute constant $0 <
  C < \infty$ such that
\[
\P\left\{\left.\Vert \hat{\bth} - \bth_0\Vert > r\right\vert X\right\}
\leq C\exp\left\{-\frac{n}{C}\cdot
  \frac{\kappa(\s_0^2,\eta_0^2,\L)}{\g^2(\g+1)^2}\cdot \frac{r^2}{(r+1)^2}\right\}
\]
for every $r \geq 0$.  
\item[(b)] Suppose that $n_0 < n$.  There is an absolute constant $0 <
  C < \infty$ such that
\[
\P\left\{\left.\Vert \hat{\bth} - \bth_0\Vert > r\right\vert X\right\}
\leq C\exp\left\{-\frac{n}{C} \cdot \frac{\kappa(\s_0^2,\eta_0^2,\L)}{\g^2(\g+1)^2}\cdot
  \left(1 - \frac{n_0}{n}\right)^4\left(\frac{n_0}{n}\right)^2\cdot\frac{r^2}{(r+1)^2}\right\}
\]
for every $r \geq 0$.  
\end{itemize}
\end{prop}

For given values of $\s_0^2$, $\eta_0^2$ and $\L$, the quantity
$\kappa(\s_0^2,\eta_0^2,\L)$ in Proposition \ref{prop:RE_conc} may be
extremely small. We have not attempted to optimize
$\kappa(\s_0^2,\eta_0^2,\L)$, and the bounds
in the proposition can almost certainly 
be improved at the expense of some additional calculations and a more
complex bound.  However, despite the magnitude of
$\kappa(\s_0^2,\eta_0^2,\L)$, the proposition yields very sensible asymptotic
conclusions.  Indeed, the key property of $\kappa(\s_0^2,\eta_0^2,\L)$ is that if $\U \subseteq (0,\infty)$ is compact, then
\begin{equation}\label{inf}
0 < \inf
\left\{\kappa(\s_0^2,\eta_0^2,\L); \
  \s_0^2,\eta_0^2,\l_1,\ldots,\l_n,\vfrak(\L) \in \U\right\}.
\end{equation}
An immediate consequence is that if $\s_0^2,\eta_0^2,\l_1,\ldots,\l_n,\vfrak(\L)$ are contained in a compact subset
of $(0,\infty)$, then Proposition \ref{prop:RE_conc} implies that
$\hat{\bth}$ converges to $\bth_0$ at rate $n^{1/2}$ [at least 
when $n = n_0$; if $n_0 <n$, then part (b) of the proposition requires
the additional condition that $n_0/n$ stays away from 1 --- this is
discussed further below].  

The bounds in Proposition \ref{prop:RE_conc} are tighter
[i.e. $\kappa(\s_0^2,\eta_0^2,\L)$ is larger] when the
eigenvalue variance $\vfrak(\L)$ is large.  This is
 related to identifiability: 
 $\s_0^2$ and $\eta_0^2$ are not identifiable when $\vfrak(\L) = 0$,
 and it is easier to distinguish between them when $\vfrak(\L)$ is large. 

The cases where $n_0 = n$ and $n_0 < n$ are considered separately in Proposition \ref{prop:RE_conc} because the large-$\eta^2$ asympotic behavior of
$\s_0^2(\eta^2) = \E\{\s_*^2(\eta^2)\vert X\}$ differs in these two settings.  In
particular,  if $n_0 = n$,
then $\s_0^2(\eta^2) \asymp \eta^2$ as $\eta^2 \to
\infty$; on the other hand, if $n_0 < n$, then $\s_0^2(\eta^2) \asymp
(1 - n_0/n)$ as $\eta^2 \to \infty$.   

Note that Proposition \ref{prop:RE_conc} (a) actually makes no
explicit reference to $p$, or to the relative
convergence rates of $p$ and $n$.  However, there are implicit conditions on $p$.
For instance, since $n_0 = n$ in part (a), we must have $n \leq p$.  Additionally, in order 
ensure that $\l_1,\ldots,\l_n$ are contained in a compact subset 
$\U \subseteq (0,\infty)$, so that \eqref{inf} holds, it may be
natural to enforce other conditions on $p$, e.g. $p/n \to \rho \in
(1,\infty)$.

Part (b) of Proposition \ref{prop:RE_conc} applies to
settings where $p <n$.  Note that the upper bound in part (b) contains
an additional term $(1 - n_0/n)^4(n_0/n)^2$, as compared to
Proposition \ref{prop:RE_conc} (a).   Thus, assuming that
$\s_0^2,\eta_0^2,\l_1,\ldots,\l_n,\vfrak(\L)$ are contained in a
compact subset of $(0,\infty)$, 
  we conclude that $\hat{\bth}$ converges to $\bth_0$ at rate
  $n^{1/2}$, if
\begin{equation}
\label{liminf} \liminf \left(1 -
\frac{n_0}{n}\right)\frac{n_0}{n} > 0.
\end{equation}  
Observe that \eqref{liminf} implies $p \to \infty$.  Hence, we need $p
\to \infty$ in order to ensure that $\hat{\bth}$ is consistent. This is reasonable because information about
$\eta^2_0 = p\E(\b_j^2)/\s_0^2$ is accumulated through
$\b_1,\ldots,\b_p$.  The condition \eqref{liminf} also implies that if
$X$ is full rank, then we
must have $p/n \to \rho < 1$ in order to ensure consistency.  This
condition seems less natural and can likely be relaxed with a more
careful analysis; similar challenges arise frequently in random matrix
theory when $p/n \to 1$
\citep[e.g.][]{bai2003convergence}.

\subsection{A more general concentration bound}\label{sec:RE_dconc}

In this section, we investigate the performance of $\hat{\bth}$ in
models where the random-effects might be dependent.  Suppose that $\tilde{\bb} = (\tilde{\b}_1,\ldots,\tilde{\b}_p)^{\T}
\in \R^p$ is a random vector that is independent of $\be,X$ and let
\begin{equation}\label{lmtilde}
\tilde{\y} = X\tilde{\bb} + \be.
\end{equation}    
We do not assume that $\tilde{\bb}$ has independent components or that
each of the components has the same variance.  
We define the variance
components estimator based on the data $(\tilde{\y},X)$,
\begin{equation}\label{vartilde}
\tilde{\bth} = (\tilde{\s}^2,\tilde{\eta}^2) = \argmax_{\s^2,\eta^2
  \geq 0} \tilde{\ell}(\bth), 
\end{equation}
where 
\[
\tilde{\ell}(\bth) = -\frac{1}{2}\log(\s^2) - \frac{1}{2n}\log
\det\left(\frac{\eta^2}{p}XX^{\T} + I\right) -
\frac{1}{2\s^2n}\tilde{\y}^{\T}\left(\frac{\eta^2}{p}XX^{\T} + I \right)^{-1}\tilde{\y}.
\]
The next proposition is a concentration bound for $\tilde{\bth}$,
which implies that the estimator may still perform reliably, if there is a good independent
coupling for $\tilde{\bb}$.  

\begin{prop}\label{prop:RE_dconc} 
Suppose that $\tilde{\y},\tilde{\bth}$ satisfy
\eqref{lmtilde}--\eqref{vartilde}. Suppose further that $\bb = (\b_1,\ldots,\b_p)^{\T} \in \R^p$
is a random vector with independent components, which is independent of $\be,X$ (but may be
correlated with
$\tilde{\bb}$), such that the independents random-effects model \eqref{LM}--\eqref{RE} and
\eqref{subgbd} hold.  Let $\kappa(\s_0^2,\eta_0^2,\L)$ be as in
\eqref{kappa}.
\begin{itemize}
\item[(a)] Suppose that $n_0 = n$.  There is an absolute constant $0 <
  C < \infty$ such that
\begin{align}\nonumber
\P\left\{\left.\Vert \tilde{\bth} - \bth_0\Vert > r\right\vert X\right\}
& \leq C\exp\left\{-\frac{n}{C}\cdot
  \frac{\kappa(\s_0^2,\eta_0^2,\L)}{\g^2(\g+1)^2} \cdot \frac{r^2}{(r+1)^2}\right\} \\ \label{conca}
& \quad + 4\P\left\{\left.\Vert  \tilde{\bb} - \bb\Vert >
  \frac{1}{C}\cdot \frac{\kappa(\s_0^2,\eta_0^2,\L)}{(\g +1)^4}
  \cdot\frac{n}{p+n} \cdot \frac{r}{r+1}\right\vert
  X\right\}. 
\end{align}
for every $r \geq 0$.  
\item[(b)] Suppose that $n_0 < n$.  There is an absolute constant $0 <
  C < \infty$ such that
\begin{align}\nonumber
\P\left\{\left.\Vert \tilde{\bth} - \bth_0\Vert > r\right\vert X\right\}&
\leq C\exp\left\{-\frac{n}{C} \cdot \frac{\kappa(\s_0^2,\eta_0^2,\L)}{\g^2(\g+1)^2}\cdot
  \left(1 - \frac{n_0}{n}\right)^4\left(\frac{n_0}{n}\right)^2\cdot\frac{r^2}{(r+1)^2}\right\}
  \\ \label{concb}
& \quad + 4\P\left\{\left.\Vert  \tilde{\bb} - \bb\Vert >
  \frac{1}{C}\cdot \frac{\kappa(\s_0^2,\eta_0^2,\L)}{(\g+1)^4}\cdot \frac{n}{p+n} \cdot \frac{r}{r+1}\right\vert
  X\right\}. 
\end{align}
for every $r \geq 0$.  
\end{itemize}
\end{prop}

The proof of Proposition \ref{prop:RE_dconc} is similar to that
of Proposition \ref{prop:RE_conc} and may be found in Section
\ref{supp:RE_dconc} of the Supplementary Material.  Observe that the first term in each upper bound
\eqref{conca}--\eqref{concb} is the exact same as in Proposition
\ref{prop:RE_conc}.  The second term in each of the bounds is new; this
term is small, if $\Vert \tilde{\bb} - \bb\Vert$ is typically small.
In other words, in a random-effects
model where \eqref{RE} does not hold, the Gaussian maximum likelihood
estimator $\tilde{\bth}$ may be a reliable estimator for the variance
components parameter $\bth_0 = (\s_0^2,\eta_0^2)$ from a corresponding random-effects
model \eqref{LM}--\eqref{RE}, if $\tilde{\bb} \approx \bb$.
Proposition \ref{prop:RE_dconc} is useful for applications involving misspecified
random-effects models.  For example, it can be used to recover some of 
\citepos{jiang2014high} results for sparse random-effects models in
genome-wide assocation studies (though Jiang et al. take a very
different approach), and for variance estimation problems in
high-dimensional linear models with fixed (non-random) $\bb$ \citep{dicker2015efficient}.
In both of these applications, the predictors $x_{ij}$ are assumed to
be random; the strategy is to leverage symmetry in the predictor distribution to
reduce the problem to one where $\tilde{\bb}$ is exchangeable and has
a tight independent coupling, so that Proposition \ref{prop:RE_dconc}
can be applied. 

\subsection{Normal approximation}\label{sec:RE_napprox}

In this section, we shift our attention back to the independent random-effects
model \eqref{LM}--\eqref{RE} and give a normal approximation result
for $\hat{\bth}$ (Proposition \ref{prop:RE_napprox} below).  One consequence of
Proposition \ref{prop:RE_napprox} is that under conditions similar to
those described after Proposition \ref{prop:RE_conc},
$n^{1/2}(\hat{\bth} - \bth_0)$ is
asymptotically normal, when $p/n \to \rho \in [0,\infty)$.  As with
consistency (discussed in Section \ref{sec:RE_conc}), asymptotic normality of $\hat{\bth}$
has been studied previously in similar settings
\citep{jiang1996REML}.  
However, the
main significance of Theorem \ref{prop:RE_napprox} is its flexible
finite-sample nature, which makes it an easy-to-use tool for 
applications.  

To derive Theorem \ref{prop:RE_napprox}, 
we again
follow the standard
strategy for parametric $M$-estimators.  First, we introduce the score
function
\begin{align*}
S(\bth) & = \frac{\partial}{\partial\bth} \ell(\bth) =
          \left[\begin{array}{c} S_1(\bth) \\ S_2(\bth)\end{array}\right]\\ & =
\left[\begin{array}{c}  \frac{1}{2\s^4n}\y^{\T}\left(\frac{\eta^2}{p}XX^{\T} +
        I\right)^{-1}\y -\frac{1}{2\s^2} 
       \\ \frac{1}{2\s^2n}\y^{\T}\left(\frac{1}{p}XX^{\T}\right) \left(\frac{\eta^2}{p}XX^{\T} +
        I\right)^{-2}\y-\frac{1}{2n}\tr\left\{\left(\frac{1}{p}XX^{\T}\right) \left(\frac{\eta^2}{p}XX^{\T} +
        I\right)^{-1}\right\} \end{array}\right].
\end{align*}
Then $S(\hat{\bth}) = 0$, provided $\hat{\eta}^2 > 0$.  The main idea
of the proof is to Taylor
expand the score function about $\bth_0$ so that
\begin{equation}\label{taylor_score}
0 = S(\hat{\bth}) = S(\bth_0) +
J(\bth_0)(\hat{\bth} - \bth_0) + \r,
\end{equation}
where $J(\bth) = \frac{\partial}{\partial\bth} S(\bth)$
and $\r$ is a remainder term.  Theorem \ref{prop:RE_napprox} follows by solving
for $\hat{\bth}- \bth_0$ above, then using three key intermediate
results: (i) $S(\bth_0)$ is approximately normal, (ii)
$J(\bth_0) \approx J_0(\bth_0)$, where
\begin{equation}\label{J}
\J(\bth) = \E\{J(\bth)\} = \E\left\{\frac{\partial}{\partial\bth} S(\bth)\right\},
\end{equation} and (iii) the remainder term $\r$ is small.
 Approximate
normality of $S(\bth_0)$ follows from Theorem \ref{thm:QF_napprox} in
this paper.  The approximation
$J(\bth_0) \approx\J(\bth_0)$ and the fact
that $\r$ is small follow from  concentration properties of quadratic forms.

\begin{prop}\label{prop:RE_napprox} 
Assume that the linear random-effects model \eqref{LM}--\eqref{RE}
holds and that $\b_1,\ldots,\b_p$, $\e_1,\ldots,\e_n$ are independent
random variables satisfying \eqref{subgbd}.  Define
\begin{equation}\label{nu}
\nu(\s_0^2,\eta_0^2,\L) = \frac{(\s_0^2+1)^9(\eta_0^2+1)^{16}(\l_1+1)^{24}}{\s_0^6\eta_0^2}\cdot\frac{\{\vfrak(\L)+1\}^3}{\vfrak(\L)^3},
\end{equation}
let $f \in \C_b^3(\R^2)$,
and let $\z_2 \sim N(0,I)$ be a two-dimensional standard normal
random vector.
Finally, let $\J(\bth_0)$ be as in \eqref{J}, define
$\I(\bth_0)  = \Var\{S(\bth_0)\vert X\}$, and define
\begin{equation}\label{avar}
\Psi  = \J(\bth_0)^{-1}\I(\bth_0)\J(\bth_0)^{-1}.  
\end{equation}
There is an
absolute constant $0 < C < \infty$ such that
\begin{align}\nonumber
\bigg\vert \E\left[\left.f\{\sqrt{n}(\hat{\bth} - \bth_0)\}\right\vert
  X\right]  & - \E\{f(\Psi^{1/2}\z_2)\vert X\}
\bigg\vert \\ \nonumber &\leq
                C(\g+1)^8\nu(\s_0^2,\eta_0^2\L)\left\{\prod_{k =
                1}^3(1 + \vert
                f\vert_k)\right\}\cdot \frac{p+n}{n}\cdot \frac{\log(n)^2}{n^{1/2}} \\ \label{nbd}
& \qquad + 2\vert f\vert_0\P\left\{\left.\Vert
  \hat{\bth} - \bth_0\Vert >
  \frac{\s_0^2\log(n)}{2\sqrt{n}}\right\vert X\right\}.
\end{align}
\end{prop}

A detailed proof of Theorem \ref{prop:RE_napprox} may be found in Appendix
\ref{app:RE_napprox}.  The quantity $\nu(\s_0^2,\eta_0^2,\L)$ in
\eqref{nbd} is potentially extremely large, and plays a role similar
to $\kappa(\s_0^2,\eta_0^2,\L)$ in Propositions
\ref{prop:RE_conc}--\ref{prop:RE_dconc}.  As with the previous
propositions, despite the potential magnitude of
$\nu(\s_0^2,\eta_0^2,\L)$, the asymptotic implications of Proposition
\ref{prop:RE_napprox} are very reasonable.  Indeed, assume that the conditions of the proposition
hold.  If, additionally, $\s_0^2,\eta_0^2,\l_1,\ldots,\l_n,\vfrak(\L)$
are contained in a compact subset of $(0,\infty)$ and $p/n \to \rho
\in [0,\infty)$,
then it is clear that the first term on the right-hand side of
\eqref{nbd} converges to 0.  Moreover, Theorem \ref{prop:RE_conc} implies
that the second term on the right-hand side of \eqref{nbd} converges
to 0, as long as we have the additional condition \eqref{liminf} when
$n_0 < n$.   Thus, under the specified conditions,
\begin{equation}\label{nasymp}
\bigg\vert \E\left[\left.f\{\sqrt{n}(\hat{\bth} - \bth_0)\}\right\vert
  X\right]   - \E\{f(\Psi^{1/2}\z)\vert X\}
\bigg\vert \to 0
\end{equation}
for all $f \in \mathcal{C}_b^3(\R^2)$.  This is an
asymptotic normality result for $\sqrt{n}(\hat{\bth} - \bth_0)$.  One
apparent
limitation of \eqref{nasymp} is that it only applies for $f \in
\mathcal{C}_b^3(\R^2)$.    However, standard arguments
\citep[e.g. Section 3 of][]{reinert2009multivariate} imply that
\eqref{nasymp} is valid for broader classes of non-smooth functions $f$,
including indicator functions for measurable convex subsets of $\R^2$; thus, we may
conclude that $\sqrt{n}\Psi^{-1/2}(\hat{\bth} - \bth_0) \leadsto N(0,I)$ in
distribution, where $\Psi$ is defined in \eqref{avar}.   We note
additionally that if $\bb$ and
$\be$ are Gaussian, then $\Psi = \I(\bth_0)^{-1} = \I_N(\bth_0)^{-1}$,
where $\I_N(\bth_0) = (\iota_{ij}(\bth_0))_{i,j=1,2}$ is the Gaussian
Fisher information matrix for $\bth_0$ and
\[
\iota_{kl}(\bth_0) = \frac{1}{2\s_0^{2(4 - k - l)}n}\tr\left\{\left(\frac{1}{p}XX^{\T}\right)^{k+l-2}\left(\frac{\eta_0^2}{p}XX^{\T}
                                          + I\right)^{2-k-l}\right\},
                                      \ \ k,l=1,2.
\]
Moreover, standard likelihood theory
\citep[e.g. Chapter 6 of][]{lehmann1998theory} implies that $\hat{\bth}$ is
asymptotically efficient in the Gaussian random-effects model.

\section{Discussion}

We have presented new uniform concentration and normal approximation
bounds for quadratic forms, and described some applications to variance
components estimation in linear random-effects models.  We expect that
the general results for quadratic forms, found in Section
\ref{sec:QF}, will be useful in a range of other
applications, such as variance components estimation in
non-standard random- and fixed-effects linear models, which arise in
genomics and other applications \citep{jiang2014high, dicker2015efficient};
hypothesis testing for variance components parameters in
high-dimensional models; and other
hypothesis testing problems, where the test statistics involve
quadratic forms in many random variables.  As discussed in Sections
\ref{sec:QF_napprox} and \ref{sec:RE}, many of the bounds in the paper
can be improved, at the expense of introducing some additional
complexity into the results. Furthermore, all of our results require sub-Gaussian
random variables.  It may be of interest to sharpen the results in the paper
and extend them to allow for heavier-tailed random variables with
sufficiently many moments.

\begin{appendices}
\section{Proof of Theorem \ref{thm:QF_conc}} \label{app:QF_conc}

The proof begins with a chaining construction.  Fix a positive integer
$M$ and define a regular grid on $[0,R]^K$
with $(2^M+1)^K$ points,
  $\U_M^K = \U_M \times \cdots \times \U_M$, where $\U_M =
  \{i2^{-M}R\}_{i = 0}^{2^M} \subseteq [0,R] \subseteq\R$.  For each $\u = (u_1,\ldots,u_K)^{\T} \in [0,R]^K$ and $j =
  1,\ldots,M$ define $\dot{\u}_j = (\dot{u}_{1j},\ldots,\dot{u}_{Kj})^{\T}$
  where $\dot{u}_{ij}$ is the smallest point in $\U_M$ that is at least as large as
  $u_i$; additionally, define $\dot{u}_0 = 0$.  

Next, consider the decomposition
\[
\bz^{\T}Q(\u)\bz - \E\{\bz^{\T}Q(\u)\bz\} 
  = \Delta_1(\u) + \Delta_2(\u) +\Delta_3(\u),  
\]
where
\begin{align*}
\Delta_1(\u) & = \bz^{\T}\{Q(\u) - Q(\dot{\u}_M)\}\bz -
               \E\left[\bz^{\T}\{Q(\u) - Q(\dot{\u}_M)\}\bz\right], \\ 
\Delta_2(\u) & = \bz^{\T}Q(\dot{\u}_0)\bz - \E\{\bz^{\T}Q(\dot{\u}_0)\bz\},
  \\
\Delta_3(\u) & = \sum_{j = 1}^M \bz^{\T}\{Q(\dot{\u}_j) -
              Q(\dot{\u}_{j-1})\}\bz -\sum_{j= 1}^M
               \E\left[\bz^{\T}\{Q(\dot{\u}_j) - Q(\dot{\u}_{j-1})\}\bz\right].
\end{align*}
Let $ r_1,r_2,r_3 >  0$ satisfy $r_1 + r_2 + r_3 = r$.  Then
\begin{equation}\label{p_bound}
\P\left[\sup_{\u \in [0,R]^K} \vert \bz^{\T}Q(\u)\bz -
  \E\{\bz^{\T}Q(\u)\bz\}\vert > r\right] \leq \sum_{i= 1}^3\P\left\{\sup_{\u \in [0,R]^K}
                                                \vert \Delta_i(\u) \vert
                                                > r_l \right\}.
\end{equation}
To prove the theorem, we bound each term on the right-hand side of
\eqref{p_bound}.

To bound the term in \eqref{p_bound} involving $\Delta_1(\u)$, observe that
\begin{align*}
\vert \Delta_1(\u) \vert & \leq \Vert \bz \Vert^2 \Vert Q(\u) -
Q(\dot{\u}_M)\Vert + \left\vert\tr\left[\E(\bz\bz^{\T})\{Q(\u) -
Q(\dot{\u}_M)\}\right]\right\vert \\
& \leq L \Vert \u - \dot{\u}_M\Vert \Vert V^{\T}V\Vert \left\{\Vert \bz \Vert^2 + \sum_{i
  = 1}^d \E(\zeta_i^2)\right\} \\
& \leq K^{1/2}L R 2^{-M} \Vert V^{\T}V\Vert\left(\Vert \bz \Vert^2 + 4d\g^2\right),
\end{align*}
where the second inequality follows from Von Neumann's trace
inequality \citep{mirsky1975trace} and the last inequality follows
from \eqref{subgauss}.  It follows that
\begin{align}\nonumber
\P\left\{\sup_{\u \in [0,R]^K}
                                                \vert \Delta_1(\u) \vert
                                                > r_1 \right\} & \leq
                                                                  \P\left\{\frac{K^{1/2}LR}{2^Mr_1} \Vert V^{\T}V\Vert
                                                                  \left(\Vert\bz
                                                                  \Vert^2
                                                                  +
                                                                  4d\g^2\right)
                                                                  > 1
                                                                  \right\}
  \\ \nonumber
& \leq \frac{K^{1/2}LR}{2^Mr_1} \Vert V^{\T}V\Vert
                                                                  \E\left(\Vert\bz
                                                                  \Vert^2
                                                                  +
                                                                  4d\g^2\right)
  \\ \label{p1_bound}
& \leq \frac{8K^{1/2}LRd\g^2}{2^Mr_1} \Vert V^{\T}V\Vert.
\end{align}

To bound the term in \eqref{p_bound} that depends on $\Delta_2(\u)$, we use the Hanson-Wright
inequality \citep[Theorem 1.1 of][]{rudelson2013hanson}, which implies
that there is an absolute constant $c > 0$, such that
\begin{align}\nonumber
\P\left\{\sup_{\u \in [0,R]^K} \vert \Delta_2(\u)\vert >
  r_2\right\} & = \P\left[ \vert\bz^{\T} Q(0) \bz - \E\{\bz^{\T}Q(0)\bz\}\vert >
  r_2\right] \\ \label{p2_bound}
& \leq 2\exp\left[-c\min\left\{\frac{r_2^2}{\g^4\Vert
  Q(0)\Vert_{\HS}^2},\frac{r_2}{\g^2\Vert Q(0)\Vert}\right\}\right]
  \\ \nonumber
& \leq 2\exp\left[-c\min\left\{\frac{r_2^2}{\g^4\Vert V^{\T}V\Vert^2\Vert T(0)
  \Vert_{\HS}^2 },\frac{r_2}{\g^2\Vert V^{\T}V\Vert\Vert T(0)\Vert}\right\}\right]
\end{align}
(specifically, the first inequality above follows from the
Hanson-Wright inequality; the second inequality follows from basic
bounds on matrix norms).

Finally, we bound the term in involving $\Delta_3(\u)$ in \eqref{p_bound}.  Let
$s_1,\ldots,s_K \geq 0$ satisfy $s_1 + \cdots + s_K = 1$.  Then
\begin{align*}
\P\Bigg\{&\sup_{\u \in [0,R]^K} \vert \Delta_3(\u)\vert >
  r_3\Bigg\} \\ & \leq \sum_{j = 1}^M \P\left\{\sup_{\u \in [0,R]^K} \left\vert\bz^{\T}\{Q(\dot{\u}_j) -
                 Q(\dot{\u}_{j-1})\}\bz - \E[\bz^{\T}\{Q(\dot{\u}_j) -
                 Q(\dot{\u}_{j-1})\}\bz]\right\vert > s_jr_3\right\}.
\end{align*}
By construction, for each $j = 1,\ldots, M$ and $i = 1,\ldots,K$, there is a $k =
1,\ldots,2^J$ such that $\dot{u}_{ij} =k2^{-j}R$.  Thus, for each $j =
1,\ldots,M$ and $i = 1,\ldots,K$, there are $2^{jK}$
possible pairs $(\dot{\u}_j,\dot{\u}_{j-1})$ and it follows that
\begin{align*}
\P\Bigg\{&\sup_{\u \in [0,R]^K}\left\vert\bz^{\T}\{Q(\dot{\u}_j) -
                 Q(\dot{\u}_{j-1})\}\bz - \E[\bz^{\T}\{Q(\dot{\u}_j) -
                 Q(\dot{\u}_{j-1})\}\bz]\right\vert >
  s_jr_3\Bigg\} \\
& \leq 2^{jK}\max_{(\dot{\u}_j,\dot{\u}_{j-1})} \P\Bigg\{ \left\vert\bz^{\T}\{Q(\dot{\u}_j) -
                 Q(\dot{\u}_{j-1})\}\bz - \E[\bz^{\T}\{Q(\dot{\u}_j) -
                 Q(\dot{\u}_{j-1})\}\bz]\right\vert >
  s_kr_3\Bigg\} \\
& \leq 2^{jK+1}\max_{(\dot{\u}_j,\dot{\u}_{j-1})} 
  \exp\left[-c\min\left\{\frac{s_j^2r_3^2}{\g^4\Vert Q(\dot{\u}_{j})
  - Q(\dot{\u}_{j-1})\Vert_{\HS}^2},\frac{s_jr_3}{\g^2\Vert Q(\dot{\u}_{j})
  - Q(\dot{\u}_{j-1})\Vert}\right\}\right] \\
& \leq
  2^{jK+1}\exp\left[-c\min\left\{\frac{4^js_j^2r_3^2}{\g^4\Vert
  V^{\T}V\Vert^2KL^2R^2 m},\frac{2^js_jr_3}{\g^2\Vert V^{\T}V\Vert K^{1/2}LR}\right\}\right],
\end{align*}
where we have used the Hanson-Wright inequality again in the third
line above.  We conclude that 
\begin{align*}
\P\Bigg\{\sup_{\u \in [0,R]^K}  &\vert \Delta_3(\u)\vert >
  r_3\Bigg\} \\ & \leq \sum_{j = 1}^M 2^{jK+1}\exp\left[-c\min\left\{\frac{4^js_j^2r_3^2}{\g^4\Vert
  V^{\T}V\Vert^2KL^2R^2 m},\frac{2^js_jr_3}{\g^2\Vert V^{\T}V\Vert K^{1/2}LR}\right\}\right].
\end{align*}
Now take $s_j = (1/3)\cdot(3/4)^j$, for $j = 1,\ldots,M-1$, and $s_M = 1 - (s_1 +
\cdots + s_{M-1}) > (1/3)\cdot(3/4)^M$.  Then 
\begin{align*}
\P\bigg\{&\sup_{\u \in [0,R]^K}  \vert \Delta_3(\u)\vert >
  r_3\bigg\}  \\ & \leq \sum_{j = 1}^M
                 2^{jK+1}\exp\left[-c\min\left\{\frac{(9/4)^jr_3^2}{9\g^4\Vert
                  V^{\T}V\Vert^2KL^2R^2m},\frac{(3/2)^jr_3}{3\g^2\Vert
                  V^{\T}V\Vert K^{1/2}LR}\right\}\right]
  \\
& = \sum_{k = 1}^K \exp\left[(jK+1)\log(2)
  -c\min\left\{\frac{(9/4)^jr_3^2}{9\g^4\Vert
  V^{\T}V\Vert^2KL^2R^2m},\frac{(3/2)^jr_3}{3\g^2\Vert V^{\T}V\Vert K^{1/2}LR}\right\}\right].
\end{align*}
If 
\begin{equation}\label{lb_epsilon}
\frac{r_3^2}{\g^4\Vert V^{\T}V\Vert^2KL^2R^2m} \geq \frac{225K^2}{\min\{c,c^2\}}, 
\end{equation}
then 
\begin{align*}
(jK+1)\log(2)&-  c\min\left\{\frac{(9/4)^jr_3^2}{9\g^4\Vert
                V^{\T}V\Vert^2KL^2R^2m},\frac{(3/2)^jr_3}{3\g^2\Vert
                V^{\T}V\Vert K^{1/2} LR}\right\} \\ &
  \leq -j\log(2) - c\min\left\{\frac{r_3^2}{9\g^4\Vert
                                              V^{\T}V\Vert^2KL^2R^2m},\frac{r_3}{3\g^2\Vert
                                              V^{\T}V\Vert K^{1/2}LR}\right\}.
\end{align*}
Hence, if \eqref{lb_epsilon} holds,
\begin{align} \label{p3_bound}
\P\Bigg\{\sup_{\u \in [0,R]^K}  & \vert \Delta_3(\u)\vert >
  r_3\Bigg\} \\ \nonumber
&  \leq \exp\left[-c \min\left\{\frac{r_3^2}{9\g^4\Vert
                                 V^{\T}V\Vert^2KL^2R^2m},\frac{r_3}{3\g^2\Vert
                                 V^{\T}V\Vert K^{1/2}LR}\right\}\right].
\end{align}

To finish the proof, we combine \eqref{p_bound}--\eqref{p2_bound} and \eqref{p3_bound}, and let
$K \to \infty$, to obtain
\begin{align*}
\P\Bigg[\sup_{\u \in [0,R]^K}  & \vert \bz^{\T}Q(\u)\bz -
  \E\{\bz^{\T}Q(\u)\bz\}\vert > r\Bigg] \\
& \leq 2\exp\left[-c\min\left\{\frac{r_2^2}{\g^4\Vert V^{\T}V\Vert^2\Vert T(0)
  \Vert_{\HS}^2},\frac{r_2}{\g^2 \Vert V^{\T}V\Vert\Vert T(0)\Vert
  }\right\}\right]\\ & \qquad + \exp\left[-c
                       \min\left\{\frac{r_3^2}{9\g^4\Vert
                       V^{\T}V\Vert^2KL^2R^2m},\frac{r_3}{3\g^2\Vert
                       V^{\T}V\Vert K^{1/2}LR}\right\}\right],
\end{align*}
whenever \eqref{lb_epsilon} holds.
The theorem follows by taking, say, $r_1=r_2 = r_3 = r/3$.  

\section{Proof of Theorem
  \ref{thm:QF_napprox}} \label{app:QF_napprox}

We follow the proof of Theorem 2.1 in \cite{reinert2009multivariate}, and use 
Stein's method with exchangeable pairs.   Let $f: \R^{2K} \to \R$ be a
three-times differentiable function.  By Lemma 2.6 in
\citep{reinert2009multivariate}, there is a 3-times differentiable
function $g: \R^{2K} \to \R$ satisfying the Stein identity
\begin{equation}\label{stein0}
\E \{f(\w)\}-
\E\{f(V^{1/2}\w)\} = \E\left\{\nabla^{\T} V \nabla
g(\w) - \w^{\T} \nabla g(\w)\right\}
\end{equation}
and
\[
\left|\frac{\partial^k
    g(\x)}{\prod_{j = 1}^k \partial x_{i_j}}\right| \leq \frac{1}{k}\left|\frac{\partial^k
    f(\x)}{\prod_{j = 1}^k \partial x_{i_j}}\right|
\]
for all $\x =(x_1,\ldots,x_{2K})^{\T} \in \R^{2K}$, $k= 1,2,3$, and $i_j
\in \{1,\ldots,k\}$.  To
prove the theorem, we bound 
\begin{equation}\label{stein}
S = \E\left\{\nabla^{\T} V \nabla
g(\w) - \w^{\T} \nabla g(\w)\right\}.
\end{equation}

Next, we use exchangeability.  
Let
$\bz' = (\zeta_1',\ldots,\zeta_{d}')^{\T}$ be an independent copy
of $\zeta$, and let $\i\in \{1,\ldots,d\}$ be an independent and
uniformly distributed random index.  Define the vector $\w' \in \R^{2K}$ exactly
as we defined $\w$, except that $\zeta_{\i}$ is replaced with
$\zeta_{\i}'$ throughout.  More precisely, let $\mathbf{e}_i \in \R^d$ be the
$i$-th standard basis vector in $\R^d$ and define
\begin{align*}
w_k' & = \{\bz + (\zeta_{\i}' - \zeta_{\i})\mathbf{e}_{\i}\}^{\T}Q_k \{\bz +
       (\zeta_{\i} ' - \zeta_{\i})\mathbf{e}_{\i}\}  = w_k + 2(\zeta_{\i} ' -
       \zeta_{\i})\mathbf{e}_{\i}^{\T}Q_k\bz + \mathbf{e}_{\i}^{\T}Q_k\mathbf{e}_{\i} (\zeta_{\i} ' - \zeta_{\i})^2, \\
\widecheck{w}_k' & = \{\bz + (\zeta_{\i} ' - \zeta_{\i})\mathbf{e}_{\i}\}^{\T}\widecheck{Q}_k
       \{\bz + (\zeta_{\i} ' - \zeta_{\i})\mathbf{e}_{\i}\} - \tr(Q_k)= \widecheck{w}_k + \mathbf{e}_{\i}^{\T}Q_k\mathbf{e}_{\i}\{(\zeta_{\i} ')^2 - \zeta_{\i}^2\},
\end{align*}
for $k = 1,...,K$.  Then $
\w' = (w_1',\widecheck{w}_1',...,w_K',\widecheck{w}_K')^{\T} \in \R^{2K}$.

Let's compute $\E(w_k' - w_k|\bz)$ and $\E(\widecheck{w}_k' -
\widecheck{w}_k|\bz)$. Since
\begin{align}\label{d1}
w_k' - w_k &  =   2(\zeta_{\i} ' - \zeta_{\i})\mathbf{e}_{\i}^{\T}Q_k\bz +
  \mathbf{e}_{\i}^{\T}Q_k\mathbf{e}_{\i} (\zeta_{\i} ' - \zeta_{\i})^2\\ 
\widecheck{w}_k' - \widecheck{w}_k 
& =  \mathbf{e}_{\i}^{\T}Q_k\mathbf{e}_{\i}\{(\zeta_{\i} ')^2 - \zeta_{\i}^2\},
\end{align}
it follows that
\begin{align*}
\E(w_k' - w_k|\bz) & = \E\left\{\left.2(\zeta_{\i} ' -
       \zeta_{\i})\sum_{j = 1}^{d} q_{\i j}^{(k)}\zeta_j + q_{\i\i}^{(k)}(\zeta_{\i} ' -
                     \zeta_{\i})^2\right|\bz\right\} \\
& = -\frac{2}{d}\sum_{i,j = 1}^{d} q_{ij}^{(k)} \zeta_i\zeta_j + \frac{1}{d}
  \sum_{i = 1}^{N} q_{ii}^{(k)}\zeta_i^2 + \frac{1}{N}\tr(Q_k) \\
& = -\frac{2}{d}w_k + \frac{1}{d}\widecheck{w}_k
\end{align*} 
and
\[
\E(\widecheck{w}_k' - \widecheck{w}_k|\bz)  =
                                     \E\left[\left. q_{\i\i}^{(k)}\{(\zeta_{\i} ')^2
                                     - \zeta_{\i}^2\}\right|\bz\right]  = - \frac{1}{d}\widecheck{w}_k.
\]
Thus, 
\begin{equation}\label{ediff}
\E(\w' - \w|\bz)   = -\L_K\w,
\end{equation}
where
\begin{align*}
\L_1 & = \left[\begin{array}{rr} \frac{2}{d} & -\frac{1}{d} \\ 0 &
                                                                   \frac{1}{d}\end{array}\right]
                                                                   \in
                                                                   \R^{2
                                                                   \times 2}
                                                                  , \
                                                                   \ 
\L_K  = \left[\begin{array}{cccc}\L_1 & 0 & \cdots & 0 \\ 0 & \L_1 &
                                                                    \cdots
                                                  & 0 \\ \vdots &
                                                                  \vdots
                                         & \ddots & \vdots \\ 0 & 0 &
                                                                      \cdots
                                                  & \L_1 \end{array}\right]
                                                                \in \R^{2K
  \times 2K}.
\end{align*}

Next, we will work our way back to the Stein identity
\eqref{stein} and take advantage of the identity we just
derived \eqref{ediff}.  Define
\[
G(\x',\x) = \frac{1}{2}(\x' - \x)^{\T}\L_K^{-\T}\{\nabla g(\x') +
\nabla g(\x)\}, \ \ \x,\x' \in \R^{2K}.
\]
By exchangeability, $\E\{G(\w',\w)\} = 0$.  Thus,
\begin{align}\nonumber
0 & = \frac{1}{2}\E\left[(\w' -\w)^{\T}\L_K^{-\T}\{\nabla g(\w') +
\nabla g(\w)\}\right] \\ \nonumber
& = \E\left\{(\w' -\w)^{\T}\L_K^{-\T}\nabla
  g(\w)\right\} + \frac{1}{2}\E\left[(\w' - \w)^{\T}\L_K^{-\T}\{\nabla g(\w') -
\nabla g(\w)\}\right] \\ \label{0decomp}
& = -\E\left\{\w^{\T}\nabla g(\w)\right\}+ \frac{1}{2}\E\left[(\w' - \w)^{\T}\L_K^{-\T}\{\nabla g(\w') -
\nabla g(\w)\}\right].
\end{align}
where we used (\ref{ediff}) in the last step.  Now we Taylor expand and
use some other basic manipulations to
get a direct connection between \eqref{stein} and \eqref{0decomp}.
Indeed, by Taylor's theorem,
\begin{align*}
(\w' - \w)^{\T}\L_K^{-\T}\{\nabla g(\w') & -
\nabla g(\w)\}  \\ & = (\w' -\w)^{\T}\L_K^{-\T}\nabla^2
g(\w) (\w' - \w)+  (\w' -\w)^{\T}\L_K^{-\T} \mathbf{r}^{(2)} \\
& = \tr\left[(\w' - \w)(\w' - \W)^{\T}\L_K^{-\T}\nabla^2g(\w)\right] +  (\w' -\w)^{\T}\L_K^{-\T} \mathbf{r}^{(2)},
\end{align*}
where $\mathbf{r}^{(2)} = (r_1^{(2)},...,r_{2K}^{(2)})^{\T}$, 
\[
r_k^{(2)} =  (\w' -\w)^{\T}R_k^{(2)}(\w' -\w),
\]
and each $R_k^{(2)} = (R_{ijk}^{(2)})$ is a $2K \times 2K$ matrix with
$\vert R_{ijk}^{(2)}\vert\leq (1/2)\vert f\vert_3$.  Thus, by \eqref{0decomp},
\begin{align}\label{stein1}
\E\left\{\w^{\T}\nabla g(\w)\right\} & = \frac{1}{2}\E\tr\left[(\w' - \w)(\w' - \w)^{\T}\L_K^{-\T}\nabla^2g(\w)\right] \\ \nonumber
& \qquad + \frac{1}{2}\E\left\{ (\w' - \w)^{\T}\L_K^{-\T} \mathbf{r}^{(2)}\right\}.
\end{align}
Since
\begin{equation}\label{exvar}
\E\left\{(\w' - \w)(\w' -
  \w)^{\T}\right\} = 2 \E\left\{\w(\w -
   \w')^{\T}\right\} 
 = 2\E\left(\w\w^{\T}\L_K^{\T}\right) = 2V\L_K^{\T},
\end{equation}
it follows that 
\begin{align}\nonumber
\E\left\{\nabla^{\T} V \nabla g(\w)\right\} &=
\frac{1}{2}\E\left[\nabla^{\T} \E\left\{(\w' - \w)(\w' -
  \w)^{\T}\right\} \L_K^{-\T} \nabla g(\w)\right] \\ \label{stein2}
& = \frac{1}{2}\E\tr\left[\E\left\{(\w' - \w)(\w' -
  \w)^{\T}\right\}\L_K^{-\T}\nabla^2g(\w)\right].
\end{align}
Combining \eqref{stein} and \eqref{stein1},\eqref{stein2} yields
\begin{align} \nonumber
S &  = \E\left\{\nabla^{\T} V \nabla
g(\w) - \w^{\T} \nabla
                      g(\w)\right\} \\ \nonumber
& = \frac{1}{2}\E\tr\left[\E\left\{(\w' - \w)(\w' -
  \w)^{\T}\right\}\L_K^{-\T}\nabla^2g(\w)\right] \\ \nonumber
& \qquad-  \frac{1}{2}\E\tr\left\{(\w' - \w)(\w' - \w)^{\T}\L_K^{-\T}\nabla^2g(\w)\right\}-
 \frac{1}{2}\E\left\{ (\w' - \w)^{\T}\L_K^{-\T} \mathbf{r}^{(2)}\right\}\\
& = \frac{1}{2}S_1 + \frac{1}{2}S_2, \label{stein_next}
\end{align}
where $S_1  = -(1/2)\E\tr\left\{T\L_K^{-\T}\nabla^2g(\w)\right\}$,  
$S_2  = -(1/2)\E\left\{ (\w' - \w)^{\T}\L_K^{-\T}
  \mathbf{r}^{(2)}\right\}$, 
and $T = \E\{(\w' - \w)(\w' - \w)^{\T}|\bz\} - \E\{(\w' - \w)(\w' - \w)^{\T}\}$.
Thus, in order to bound $S$ it suffices to bound $S_1,S_2$.  

First, we work with $S_1$.  Notice that
\begin{align*}
\left|\E\tr\left\{T\L_K^{-\T}\nabla^2g(\w)\right\}\right| & \leq \frac{ |f|_2 }{2}\Vert
\L_K^{-\T}\Vert_{\HS}\E(\Vert T\Vert_{\HS}) \\ &  =
\frac{|f|_2K^{1/2}}{2}\tr\{(\L_1^{\T}\L_1)^{-1}\}^{1/2}\E(\Vert
                                       T\Vert_{\HS}) \\ & 
\leq \frac{3}{5} K^{1/2}d|f|_2\E(\Vert T\Vert_{\HS}),
\end{align*}
where we have used the fact that $\tr\{(\L_1^{\T}\L_1)^{-1} = (3/2)d^2$.
Thus, 
\begin{equation}\label{s1b1}
|S_1| \leq \frac{3}{10} K^{1/2}d|f|_2\E(\Vert T\Vert_{\HS}).
\end{equation}
It requires a bit more work to bound $\E(\Vert T \Vert_{\HS})$ in \eqref{s1b1}.  

The matrix $T$ can be written as 
\[
T = \left[\begin{array}{cccc}
                                                 T_{11} & T_{12} &
                                                                   \cdots
                                                 & T_{1K} \\ T_{21} &
                                                                      T_{22}
                                                                 &
                                                                   \cdots
                                                 & T_{2K} \\ \vdots &
                                                                      \vdots
                                                                 &
                                                                   \ddots
                                                 & \vdots \\ T_{K1} &
                                                                      T_{K2}
                                                                 &
                                                                   \cdots
                                                 &
                                                   T_{KK} \end{array}\right],
                                               \mbox{ where } T_{kl} = \left[\begin{array}{cc} t_{11}^{kl} & t_{12}^{kl} \\
                 t_{21}^{kl} & t_{22}^{kl}\end{array}\right],
\] 
and
\begin{align}\label{tentry11}
t_{11}^{kl} & = \E\left\{(w_k' - w_k)(w_l' - w_l)|\bz\right\} -
              \E\left\{(w_k' - w_k)(w_l' - w_l)\right\}, \\ \label{tentry12}
t_{12}^{kl} & = \E\{(w_k' - w_k)(\widecheck{w}_l' - \widecheck{w}_l)|\bz\} - \E\{(w_k' -
              w_k)(\widecheck{w}_l' - \widecheck{w}_l)\}, \\ \label{tentry21}
t_{21}^{kl} & = \E\{(\widecheck{w}_k' - \widecheck{w}_k)(w_l' - w_l)|\bz\} - \E\{(\widecheck{w}_k' -
              \widecheck{w}_k)(w_l' - w_l)\}, \\ \label{tentry22}
t_{22}^{kl} & = \E\{(\widecheck{w}_k' - \widecheck{w}_k)(\widecheck{w}_l' - \widecheck{w}_l)|\bz\} - \E\{(\widecheck{w}_k' -
              \widecheck{w}_k)(\widecheck{w}_l' - \widecheck{w}_l)\}. 
\end{align}
We conclude that 
\[
\E(\Vert T\Vert_F) \leq \left[\sum_{k,l = 1}^K \sum_{i,j=1}^2
  \E\{(t_{ij}^{kl})^2\}\right]^{1/2} 
\]
and, furthermore, if we can control each of the terms $\E\{(t_{ij}^{kl})^2\}$, then
a bound on $\E(\Vert T\Vert_F)$ will follow.  Fortunately, Lemma
\ref{lemma:Tbd} from the Supplementary Material gives bounds for on
these moments.  Indeed, let $c(\g) =
4096(\g+1)^8$ and $q_{\max} = \max_{k = 1,\ldots,K} \Vert Q\Vert$.  It follows from Lemma \ref{lemma:Tbd} that
\begin{align*}
\E(\Vert T\Vert_F) & \leq \frac{Kq_{\max}^2}{d^{1/2}} 
                   \left[8\{108c(\g)^2 + 763c(\g) + 930\} +
                   4\{24c(\g)^2 + 69c(\g) + 1\} + c(\g) + 4\right]^{1/2} \\
& \leq \frac{Kq_{\max}^2}{d^{1/2}} \left\{960 c(\g)^2 +
  6381 c(\g) + 7448\right\}^{1/2} \\
& \leq \frac{Kq_{\max}^2}{d^{1/2}} \{65c(\g) + 104\}.
\end{align*}
Combining this bound on $\E(\Vert T\Vert_F)$ with \eqref{s1b1} yields
\begin{equation}\label{s1b2}
|S_1| \leq   4\{5c(\g) + 8\} K^{3/2}d^{1/2}\vert f\vert_2q_{\max}^2.
\end{equation}

Next, we bound $S_2$.   First consider the basic inequalities
\begin{align*}
\vert S_2 \vert & \leq\frac{1}{2} \Vert \L_K^{-\T}\Vert \E\left(\Vert
                  \w' - \w \Vert \Vert \mathbf{r}^{(2)}\Vert\right)
  \\
& \leq \frac{d}{4}\left\Vert \left[\begin{array}{cc} 1 & 0 \\ 1 &
                                     2 \end{array}\right]\right\Vert
                                                                        \E\left\{\Vert
                  \w' - \w \Vert^3 \left(\sum_{k = 1}^{2K} \Vert
                                                                  R_k^{(2)}\Vert^2\right)^{1/2} \right\}\\
& \leq \frac{5\cdot 2^{1/2}}{8} K^{3/2}d\vert f\vert_3 \E\left(\Vert\w' - \w\Vert^3\right).
\numberthis\label{s2b1}
\end{align*}
Now focus on bounding $\E(\Vert \w' - \w \Vert^3)$.  Each inequality
in the following chain is elementary: 
\begin{align*}
\E(\Vert \w' - \w \Vert^3) & = \E\left[\left\{\sum_{k = 1}^K (w_k' - w_k)^2
                             + \sum_{k = 1}^K (\widecheck{w}_k' -
                             \widecheck{w}_k)^2\right\}^{3/2}\right] \\
& = \E\Bigg[\left\{\sum_{k = 1}^K (2(\zeta_{\i}' -
                       \zeta_{\i})\mathbf{e}_{\i}^{\T}Q_k\bz +
                       \mathbf{e}_{\i}^{\T}Q_k\mathbf{e}_I(\zeta_{\i}' -
                       \zeta_{\i})^2)^2
                             \right. \\
& \qquad \qquad \left.+ \sum_{k = 1}^K (\mathbf{e}_{\i}^{\T}Q_k\mathbf{e}_I)^2\{(\zeta_{\i}')^2
                       - \zeta_{\i}^2\}^2\right\}^{3/2}\Bigg] \\
& \leq 2^{3/2}\E\left[\left\{8\sum_{k = 1}^K \{(\zeta_{\i}')^2 +
                       \zeta_{\i}^2\}(\mathbf{e}_I^{\T}Q_k\bz)^2 +
  9\sum_{k = 1}^K \Vert Q_k\Vert^2\{(\zeta_{\i}')^4
                       + \zeta_{\i}^4\}\right\}^{3/2}\right] \\
& = \frac{2^{3/2}}{d}\sum_{i = 1}^d \E\left[\left\{8\sum_{k = 1}^K \{(\zeta_i')^2 +
                       \zeta_i^2\}(\mathbf{e}_i^{\T}Q_k\bz)^2 +
  9\sum_{k = 1}^K \Vert Q_k\Vert^2\{(\zeta_i')^4
                       + \zeta_i^4\}\right\}^{3/2}\right] \\
& \leq \frac{150K^{1/2}}{d}\sum_{i = 1}^d \sum_{k = 1}^K
  \E\{(|\zeta_i'|^3 + |\zeta_i|^3)|\mathbf{e}_i^{\T}Q_k\bz|^3\} +
 \frac{300K^{1/2}}{d}\sum_{i = 1}^d \sum_{k = 1}^K \Vert Q_k
  \Vert^3\E(\zeta_i^6) \\
& \leq \frac{300K^{1/2}c(\g)^{1/2}}{d}\sum_{i = 1}^d \sum_{k = 1}^K \sqrt{\E\{(\mathbf{e}_i^{\T}Q_k\bz)^6\}} + 300K^{3/2}c(\g)q_{\max}^3.
\end{align*}
It remains to bound $\E\{(\mathbf{e}_i^{\T}Q_k\bz)^6\}$.  This is
accomplished by a version of Khintchine's inequality, given in
Corollary 5.12 of \citep{vershynin2010introduction}.  It implies that
there is an absolute constant $C_1 > 0$ such that
\[
\E\{(\mathbf{e}_i^{\T}Q_k\bz)^6\} \leq C_1^2
(\g+1)^6\left\{\sum_{j = 1}^d (q_{ij}^{(k)})^2\right\}^{3}.
\]
Thus,
\begin{align*}
\E(\Vert \w' - \w \Vert^3) & \leq \frac{300CK^{1/2}c(\g)^{1/2}(\g+1)^3}{d} \sum_{k
                             = 1}^K \sum_{i = 1}^d  
\left\{\sum_{j = 1}^d (q_{ij}^{(k)})^2\right\}^{3/2} + 300K^{3/2}c(\g)q_{\max}^3 \\
& \leq 300K^{3/2}\{C_1c(\g)^{1/2}(\g+1)^3 + c(\g)\}q_{\max}^3. 
\end{align*}
Combining this with \eqref{s2b1} yields
\begin{equation}\label{s2b2}
|S_2| \leq  266K^{3}q_{\max}^3\{C_1c(\g)^{1/2}(\g+1)^3 + c(\g)\}\vert f\vert_3 d.
\end{equation}
Finally, combining \eqref{stein0}--\eqref{stein}, \eqref{stein_next}, \eqref{s1b2}, and \eqref{s2b2},
we obtain
\begin{align*}
\left\vert\E\{f(\w)\} - \E\{f(V^{1/2}\z)\}\right\vert & \leq 2\{5c(\g)+8\}K^{3/2}d^{1/2}\vert
f\vert_2 q_{\max}^2\\
& \qquad + 133K^{3}\{C_1c(\g)^{1/2}(\g+1)^3 + c(\g)\}\vert f\vert_3 d q_{\max}^3\\
& \leq C (\g+1)^8\left(K^{3/2}d^{1/2}\vert
  f\vert_2q_{\max}^2 + K^3d\vert f
  \vert_3q_{\max}^3\right)
\end{align*}
for some absolute constant $C > 0$, which proves the theorem.

\section{Proof of Proposition \ref{prop:RE_conc}} \label{app:RE_conc}

The proof of Proposition \ref{prop:RE_conc} is based on several
lemmas, which are stated precisely and proved in the Supplementary
Material.  Several of these lemmas (Lemmas
\ref{lemma:s2_unif}, \ref{lemma:sld}, and \ref{lemma:llik_dev})
are basicallly corollaries of our uniform concentration bound for quadratic forms, Theorem
\ref{thm:QF_conc}.  

To prove the proposition, first let $r \geq 0$. 
Since $\hat{\s}^2 = \s_*^2(\hat{\eta}^2)$ and $\s_0^2 =
\s_0^2(\eta_0^2)$, it follows that
\begin{align*}
\left\{\Vert \hat{\bth} -
\bth_0\Vert > r\right\} & = \left\{(\hat{\eta}^2 - \eta_0^2)^2 + (\hat{\s}^2 -
                    \s_0^2)^2 > r^2\right\} \\
& \subseteq \left\{\vert\hat{\eta}^2 - \eta_0^2\vert >
  \frac{r}{\sqrt{2}}\right\} \cup \left\{\vert\hat{\s}^2 - \s_0^2\vert >
  \frac{r}{\sqrt{2}}\right\} \\
& = \left\{\vert\hat{\eta}^2 - \eta_0^2\vert >
  \frac{r}{\sqrt{2}}\right\} \cup \left\{\left\vert \s_*^2(\hat{\eta}^2)- \s_0^2(\eta_0^2)\right\vert >
  \frac{r}{\sqrt{2}}\right\} \\
& \subseteq \left\{\vert\hat{\eta}^2 - \eta_0^2\vert >
  \frac{r}{\sqrt{2}}\right\} \cup \left\{\sup_{0 \leq \eta^2 < \infty}
  \left\vert \s_*^2(\eta^2) - \s_0^2(\eta^2)\right\vert>
  \frac{r}{2\sqrt{2}}\right\} \\
& \qquad \cup \left\{\left\vert\s_0^2(\hat{\eta}^2)- \s_0^2(\eta_0^2)\right\vert >   \frac{r}{2\sqrt{2}}\right\}.
\end{align*}
Additionally, since
\[
\left\vert \s_0^2(\hat{\eta}^2)-
  \s_0^2(\eta_0^2)\right\vert   \leq \frac{\s_0^2}{n} \sum_{i = 1}^n
                                  \left\vert
                                  \frac{\eta_0^2\l_i+1}{\hat{\eta}^2\l_i
                                                + 1} -1\right\vert = \frac{\s_0^2}{n} \sum_{i = 1}^n
                                  \frac{\l_i \vert \hat{\eta}^2-\eta_0^2\vert}{\hat{\eta}^2\l_i +
  1} \leq \s_0^2\l_1\vert \hat{\eta}^2 - \eta_0^2\vert,
\]
we conclude that
\begin{align}\nonumber
\left\{\Vert \hat{\bth} -
\bth_0\Vert > r\right\} & \subseteq  \left\{\vert\hat{\eta}^2 - \eta_0^2\vert >
  \frac{r}{2\sqrt{2}(\s_0^2 + 1)(\l_1+1)}\right\} \\ \label{inclusion1}& \qquad \cup \left\{\sup_{0 \leq \eta^2 < \infty}
  \left\vert \s_*^2(\eta^2) - \s_0^2(\eta^2)\right\vert >
  \frac{r}{2\sqrt{2}}\right\}.
\end{align}
We bound the probability of the two events on
the right-hand side in \eqref{inclusion1}.  

Bounding the probability of the second event in \eqref{inclusion1} is
easy, thanks to Theorem \ref{thm:QF_conc} and Lemma \ref{lemma:s2_unif}.  By Lemma \ref{lemma:s2_unif}, there is a constant $0 < C_0 < \infty$ such that
\begin{equation}\label{Hbd1}
\P\left\{\left.\sup_{0 \leq \eta^2 < \infty}
  \left\vert \s_*^2(\eta^2) - \s_0^2(\eta^2)\right\vert >
  \frac{r}{2\sqrt{2}}\right\vert X\right\} \leq C_0
\exp\left\{\frac{-n\omega(\L)^2}{C_0\g^2(\g^2+1)}\cdot \frac{r^2}{r+1}\right\},
\end{equation}
where $\omega(\L)$ is defined in \eqref{omega}.

Bounding the probability of the first event on the right in \eqref{inclusion1}
takes more work.  In fact, we further decompose the event as follows: 
\begin{equation}\label{E1}
\left\{\vert\hat{\eta}^2 - \eta_0^2\vert >
  \frac{r}{2\sqrt{2}(\s_0^2 + 1)(\l_1+1)}\right\} \subseteq A^- \cup A^+,
\end{equation}
where
\begin{align*}
A^- & = \left\{ \frac{r}{2\sqrt{2}(\s_0^2 + 1)(\l_1+1)}<
      \vert\hat{\eta}^2 - \eta_0^2\vert \leq 
  \frac{\eta_0^2}{2\sqrt{2}(\s_0^2 + 1)(\l_1+1)}\right\} , \\
A^+ & = \left\{\vert\hat{\eta}^2 - \eta_0^2\vert >
  \frac{\eta_0^2}{2\sqrt{2}(\s_0^2 + 1)(\l_1+1)}\right\}. 
\end{align*}
To bound $\P(A^-)$, we use properties of the profile score function 
\begin{align}\nonumber
H_*(\eta^2) & = 2\s_*^2(\eta^2)\left\{\frac{d}{d\eta^2} \ell_*(\eta^2)
              \right\}\\  \label{score}
& = 
\frac{1}{n}\y^{\T}\left(\frac{1}{p}XX^{\T}\right)\left(\frac{\eta^2}{p}XX^{\T}
  + I\right)^{-2}\y  \\ \nonumber
& \qquad -
  \s_*^2(\eta^2)\frac{1}{n}\tr\left\{\left(\frac{1}{p}XX^{\T}\right)\left(\frac{\eta^2}{p}XX^{\T}
  + I\right)^{-1}\right\}.
\end{align}
In particular, let $A(r_1,r_2) = \{\mbox{There exists } \eta^2 > 0 \mbox{ such that }
r_1 <\vert
\eta^2 - \eta_0^2\vert \leq r_2 \mbox{ and } H_*(\eta^2) = 0\}$ and observe that
\[
A^- \subseteq A\left\{\frac{r}{2\sqrt{2}(\s_0^2+1)(\l_1+1)},\frac{\eta_0^2}{2\sqrt{2}(\s_0^2+1)(\l_1+1)}\right\}.
\]
Furthermore,
\[
A(r_1,r_2) \subseteq \left\{\sup_{0 \leq \eta^2 < \infty} \vert H_*(\eta^2) - H_0(\eta^2) \vert \geq \inf_{r_1 <\vert \eta^2 - \eta_0^2\vert
    \leq r_2} \vert H_0(\eta^2)\vert \right\},
\]
where $H_0(\eta^2) = \E\{H_*(\eta^2)\vert X\}$.   By Lemma
\ref{lemma:H_bd} in the Supplementary Material, if 
\[
\frac{r}{2\sqrt{2}(\s_0^2+1)(\l_1+1)}  < \vert \eta^2
  - \eta_0^2\vert \leq
  \frac{\eta_0^2}{2\sqrt{2}(\s_0^2+1)(\l_1+1)},
\]
then 
\[
\vert H_0(\eta^2) \vert 
> \frac{\s_0^2 r}{32\sqrt{2}(\s_0^2+1)(\l_1+1)^5(\eta_0^2 +1)^4}\vfrak(\L).
\]
Thus,
\[
   A^-                                                        \subseteq\left\{\sup_{0
                                                           \leq \eta^2
                                                           <\infty}
                                                           \vert
                                                           H_*(\eta^2)
                                                           -
                                                           H_0(\eta^2)\vert
                                                           > \frac{\s_0^2 r}{32\sqrt{2}(\s_0^2+1)(\l_1+1)^5(\eta_0^2 +1)^4}\vfrak(\L)\right\}.
\]
Now we can use Lemma \ref{lemma:sld}, which is an application of
Theorem \ref{thm:QF_conc}, to bound the probability of the right-hand
side above.  We conclude that there is a
constant $0 < C_1^- < \infty$ such that
\begin{equation} \label{Anegbd1}
\P(A^-\vert X)  \leq C_1^-
  \exp\bigg[-\frac{n\s_0^4}{C_1^-\g^2(\g^2+1)(\s_0^2+1)^2
               (\eta_0^2+1)^8(\l_1+1)^{16}}  
  \cdot\frac{\vfrak(\L)^2}{\vfrak(\L)+1} \cdot \frac{r^2}{r+1}\bigg].
\end{equation}

To bound $\P(A^+\vert X)$, we consider cases where $n = n_0$ and $n < n_0$
separately.  First assume that $n = n_0$.  Lemma \ref{lemma:llik_cty}
(a) from the Supplementary Material 
implies that
\begin{align*}
A^+ & \subseteq \left\{\ell_0(\eta_0^2) - \ell_0(\hat{\eta}^2) >
    \frac{\eta_0^4\chi(\eta_0^2,\L) \vfrak(\L)}{8(\s_0^2+1)^2(\l_1+1)^2(\eta_0^2 + 1)^2} \right\} \\
& \subseteq \left\{\ell_0(\eta_0^2) - \ell_*(\eta_0^2) >
    \frac{\eta_0^4\chi(\eta_0^2,\L) \vfrak(\L)}{16(\s_0^2+1)^2(\l_1+1)^2(\eta_0^2 + 1)^2}\right\} \\
& \qquad \cup
  \left\{\ell_0(\hat{\eta}^2) - \ell_*(\hat{\eta}^2) >
    \frac{\eta_0^4\chi(\eta_0^2,\L) \vfrak(\L)}{16(\s_0^2+1)^2(\l_1+1)^2(\eta_0^2 + 1)^2}\right\} \\
&\subseteq \left\{\sup_{0 \leq \eta^2 < \infty} \vert\ell_0(\eta^2) - \ell_*(\eta^2)\vert >
    \frac{\eta_0^4\chi(\eta_0^2,\L) \vfrak(\L)}{16(\s_0^2+1)^2(\l_1+1)^2(\eta_0^2 + 1)^2}\right\} .
\end{align*}
Next we apply Lemma \ref{lemma:llik_dev}, which depends on Theorem
\ref{thm:QF_conc}.  Lemma \ref{lemma:llik_dev} (a) implies that there
is a constant $0 < C_1^+ < \infty$ such that 
\begin{align}\nonumber
\P (A^+\vert X) & \leq \P\left\{\left.\sup_{0 \leq \eta^2 < \infty} \vert\ell_0(\eta^2) - \ell_*(\eta^2)\vert >
    \frac{\eta_0^4\chi(\eta_0^2,\L) \vfrak(\L)}{16(\s_0^2+1)^2(\l_1+1)^2(\eta_0^2 +
                                                               1)^2}\right\vert
                                                               X\right\}
  \\ \label{eta_bda}
& \leq
  C_1^+\exp\left[-\frac{n}{C_1^+}\cdot\frac{\s_0^4\eta_0^8\omega(\L)^2\chi(\eta_0^2,\L)^2}{\g^2(\g^2+1)(\s_0^2+1)^5(\eta_0^2+1)^4(\l_1+1)^6}\cdot
  \frac{\vfrak(\L)^2}{\{\vfrak(\L) + 1)\}^2} \right].
\end{align}
Part (a) of the proposition follows by combining \eqref{inclusion1}--\eqref{E1} and \eqref{Anegbd1}--\eqref{eta_bda}.

To prove part (b) of the proposition, assume that $n_0 < n$.  Since 
\[
\vfrak(\L) \geq \l_{n_0}^2\left(1 - \frac{n_0}{n}\right)\frac{n_0}{n},
\]
the inequality \eqref{Anegbd1} implies that
\begin{equation}\label{Anegbd2}
\P(A^-\vert X)  \leq C_2^-
  \exp\bigg[-\frac{n\s_0^4(1 - n_0/n)^2(n_0/n)^2}{C_2^-\g^2(\g^2+1)(\s_0^2+1)^2
               (\eta_0^2+1)^8(\l_1+1)^{16}(\l_{n_0}^{-1}+1)^4}
  \cdot \frac{r^2}{r+1}\bigg].
\end{equation}
Additionally, by Lemma
\ref{lemma:llik_cty} (b), 
\[
A^+
 \subseteq \left\{\sup_{0 \leq \eta^2 < \infty} \vert\ell_0(\eta^2) - \ell_*(\eta^2)\vert >
    \frac{\chi(\eta_0^2,\L) (1 - n_0/n)(n_0/n)\eta_0^4}{16(\s_0^2+1)^2(\l_1+1)^2(\eta_0^2 + 1)^2}\right\} .
\]
Hence, Lemma \ref{lemma:llik_dev} (b) implies that there is a constant $0 < C_2^+ <
\infty$ such that 
\begin{align}\label{eta_bdb}
\P (A^+\vert X) & \leq
                                                               \P\left\{\left.\sup_{0
                                                               \leq
                                                               \eta^2
                                                               <
                                                               \infty}
                                                               \vert\ell_0(\eta^2)
                                                               -
                                                               \ell_*(\eta^2)\vert
                                                               > \frac{\chi(\eta_0^2,\L) (1 - n_0/n)(n_0/n)\eta_0^4}{16(\s_0^2+1)^2(\l_1+1)^2(\eta_0^2 + 1)^2}\right\vert
                                                               X\right\}
  \\ \nonumber
& \leq C_2^+\exp\left[-\frac{n}{C_2^+} \cdot \frac{\s_0^4\eta_0^8\chi(\eta_0^2,\L)\omega(\L)^2}{\g^2(\g^2+1)(\s_0^2+1)^3(\eta_0^2+1)^4(\l_1+1)^4} \cdot \left(1 -
  \frac{n_0}{n}\right)^4\left(\frac{n_0}{n}\right)^2\right].
\end{align}
Part (b) follows from
\eqref{inclusion1}--\eqref{E1} and \eqref{Anegbd2}--\eqref{eta_bdb}.  This completes the
proof of Proposition \ref{prop:RE_conc}.

\section{Proof of Proposition
  \ref{prop:RE_napprox}} \label{app:RE_napprox}

Proposition \ref{prop:RE_napprox} is a direct application of Theorem
\ref{thm:QF_napprox}, in conjunction with some basic Taylor
expansions.  However, keeping track of all the quantities to be
bounded does require some effort.  
Let $f \in \C_b^3(\R^2)$.
By \eqref{taylor_score}, on the event that $\hat{\eta}^2 \neq 0$, 
\[
\hat{\bth} - \bth_0  = -\J(\bth_0)^{-1}S(\bth_0) - \J(\bth_0)^{-1}\{J(\bth_0) -
  \J(\bth_0)\}(\hat{\bth} - \bth_0) - \J(\bth_0)^{-1}\r,
\]
where $\r \in \R^2$ is a remainder term.  
Furthermore, by Taylor's theorem,
\[
\r =\frac{1}{2}\left[\begin{array}{c}
(\hat{\bth}
- \bth_0)^{\T}\left\{\frac{\partial^2}{\partial\bth^2}S_1(\bth^*)\right\}(\hat{\bth}
- \bth_0) \\(\hat{\bth}
- \bth_0)^{\T}\left\{\frac{\partial^2}{\partial\bth^2}S_2(\bth^*)\right\}(\hat{\bth}
- \bth_0) \end{array}\right]
\]
and $\bth^* \in \R^2$ is on the line segment connecting
$\hat{\bth}$ and $\bth_0$.  Thus, defining $\w = \sqrt{n}(\hat{\bth} -
\bth_0)$ and applying Taylor's theorem again,
\begin{align*}
f(\w) & = f\left\{-\sqrt{n}\J(\bth_0)^{-1}S(\bth_0)\right\} \\
& \qquad - \sqrt{n}\left[\r_f^{\T} \J(\bth_0)^{-1}\{J(\bth_0) -
  \J(\bth_0)\}(\hat{\bth} - \bth_0) +\r_f^{\T} \J(\bth_0)^{-1}\r\right],
\end{align*}
where $\r_f\in \R^2$ and $\Vert \r_f \Vert \leq \sqrt{2}\vert
f\vert_1$.  Now define the event
\[
A = \left\{\Vert \hat{\bth} - \bth_0\Vert \leq
  \s_0^2\frac{\log(n)}{2\sqrt{n}}\right\}
\] 
and let $\1_A$ denote the indicator of $A$.  Then
\begin{align*}
\E\{f(\w)\vert X\} & = \E\{f(\w)\1_A\vert X\} + \E\{f(\w)
                     \1_{A^c}\vert X\} \\
& =
  \E\left[f\left.\left\{-\sqrt{n}\J(\bth_0)^{-1}S(\bth_0)\right\}\right\vert
  X\right] -  \E\left[f\left.\left\{-\sqrt{n}\J(\bth_0)^{-1}S(\bth_0)\1_{A^c}\right\}\right\vert
  X\right] \\
& \qquad  -\sqrt{n}\E\left[\left.\r_f^{\T}\J(\bth_0)^{-1}\{J(\bth_0) -
  \J(\bth_0)\}(\hat{\bth} - \bth_0)\1_A + \r_f^{\T}\J(\bth_0)^{-1}\r\1_A \right\vert
  X\right] \\
& \qquad + \E\{f(\w)\1_{A^c}\vert X\}
\end{align*}
and it follows that
\begin{equation}\label{normal_decomp}
\left\vert\E\{f(\w)\vert X\} - \E\{f(\Psi^{1/2}\z_2)\vert X\}\right\vert \leq \Delta_1 +
\Delta_2  + \Delta_3 + \Delta_4,
\end{equation}
where 
\begin{align*}
\Delta_1 & = 
 \left\vert\E\left[\left.f\left\{-\sqrt{n}\J(\bth_0)^{-1}S(\bth_0)\right\}\right\vert
  X\right] - \E\{f(\Psi^{1/2}\z_2)\vert X\}\right\vert, \\
\Delta_2 & = \left\vert\E\{f(\w)\1_{A^c}\vert X\}\right\vert + \left\vert\E\left[\left.f\left\{-\sqrt{n}\J(\bth_0)^{-1}S(\bth_0)\right\}\1_{A^c}\right\vert
  X\right]\right\vert, \\
\Delta_3 & = \sqrt{n}\left\vert\E\left[\left.\r_f^{\T}\J(\bth_0)^{-1}\{J(\bth_0) -
  \J(\bth_0)\}(\hat{\bth} - \bth_0)\1_A\right\vert X\right]\right\vert, \\
\Delta_4 &= \sqrt{n}\left\vert\E\left\{\left.\r_f^{\T}
              \J(\bth_0)^{-1}\r\1_A \right\vert
  X\right\}\right\vert.
\end{align*}
To prove the theorem, we bound $\Delta_1,\Delta_2,\Delta_3,\Delta_4$ separately.

To bound $\Delta_1$, we use Theorem \ref{thm:QF_napprox} with $K = 2$,
$\x \mapsto f\{-\J(\bth_0)^{-1}\x\}$ in place of $f$, 
$\bz = (\sqrt{p}\bb^{\T}/\tau_0,\be^{\T}/\s_0)^{\T} \in \R^{n+p}$, and
\begin{align*}
Q_1 & = \frac{1}{2\s_0^4\sqrt{n}}\left[\begin{array}{c}
                                         \frac{\tau_0}{\sqrt{p}}X^{\T}
                                         \\ \s_0I \end{array}\right]\left(\frac{\eta_0^2}{p}XX^{\T}
      + I\right)^{-1}\left[\begin{array}{cc}
                             \frac{\tau_0}{\sqrt{p}}X &
                                                        \s_0I\end{array}\right],
  \\ 
Q_2 & = \frac{1}{2\s_0^2\sqrt{n}}\left[\begin{array}{c}
                                         \frac{\tau_0}{\sqrt{p}}X^{\T}
                                         \\
                                         \s_0I \end{array}\right]\left(\frac{1}{p}XX^{\T}\right)\left(\frac{\eta_0^2}{p}XX^{\T}
  + I\right)^{-2}\left[\begin{array}{cc}
                             \frac{\tau_0}{\sqrt{p}}X &
                                                        \s_0I\end{array}\right].
\end{align*}
Since 
\[
\Vert Q_1 \Vert ,  \ \Vert Q_2 \Vert \leq \frac{(\s_0^2 + 1)(\eta_0^2+1)(\l_1+1)^2}{2\s_0^2\sqrt{n}}, 
\]
Theorem \ref{thm:QF_napprox} implies that there is a constant $0 < C_0
< \infty$
such that
\begin{align}\nonumber
\Delta_1 & \leq \frac{C_0(\g+1)^8(\s_0^2 +
           1)^3(\eta_0^2+1)^3(\l_1+1)^6 }{\s_0^6} \cdot \frac{p+n}{n^{3/2}}\\ \label{D1bd0} &
                                                                    \qquad \cdot \left\{\Vert
  \J(\bth_0)\Vert^{-3} + 1\right\}(\vert f\vert_2 + 1)(\vert f\vert_3
  + 1).
\end{align}
Next, we bound $\Vert
  \J(\bth_0)\Vert^{-1}$.  First  observe that 
\[
\J(\bth) = \E\{J(\bth)\vert
X\} = \left[\begin{array}{cc} \frac{1}{2\s^4} -
                      \frac{\s_0^2}{\s^6n}\sum_{i = 1}^n
                      \frac{\eta_0^2\l_i + 1}{\eta^2\l_i + 1} &
                                                             -\frac{\s_0^2}{2\s^4n}\sum_{i
                                                             = 1}^n
                                                             \frac{\l_i(\eta_0^2\l_i+1)}{(\eta^2\l_i
                                                             + 1)^2}
                      \\ -\frac{\s_0^2}{2\s^4n}\sum_{i
                                                             = 1}^n
                                                             \frac{\l_i(\eta_0^2\l_i+1)}{(\eta^2\l_i
                                                             + 1)^2} &
                                                                      \frac{1}{2n}\sum_{i
                                                                      =
                                                                      1}^n
                                                                      \frac{\l_i^2}{(\eta^2\l_i
                                                                      +
                                                                      1)^2}
                                                                      -
                                                                      \frac{\s_0^2}{\s^2
                                                                      n}\sum_{i
                                                                      =
                                                                      1}^n \frac{\l_i^2(\eta_0^2\l_i+1)}{(\eta^2\l_i+1)^3}
\end{array}\right].
\]  
It follows that
\begin{align*}
\det\{\J(\bth_0)\} & = \frac{1}{4\s_0^4}\left\{\frac{1}{n}\sum_{i =
                      1}^n \frac{\l_i^2}{(\eta_0^2\l_i + 1)^2} -
                      \left(\frac{1}{n}\sum_{i = 1}^n
                      \frac{\l_i}{\eta_0^2\l_i+1}\right)^2\right\} \\
& = \frac{1}{8\s_0^4n^2}\sum_{i,j=1}^n
\frac{ (\l_i - \l_j)^2}{(\eta_0^2\l_i+1)^2(\eta_0^2\l_j+1)^2} \\
& \geq \frac{\vfrak(\L)}{4\s_0^4(\eta_0^2 + 1)^4(\l_1+ 1)^4}
\end{align*}
and, since each entry in $\J(\bth_0)$ is bounded in absolute value by $(\s_0^2 + 1)^2(\l_1+1)^2/(2\s_0^4)$,
we conclude that
\begin{equation}\label{Jinvbd}
\Vert \J(\bth_0)\Vert^{-1} \leq \frac{4}{\vfrak(\L)}(\s_0^2+1)^2(\eta_0^2+1)^4(\l_1+1)^6.
\end{equation}
Combining \eqref{D1bd0}--\eqref{Jinvbd}, there is a $0 < C_1 < \infty$ such that
\begin{equation}\label{D1bd}
\Delta_1 \leq C_1 \frac{(\g+1)^8(\s_0^2 + 1)^{9}(\eta_0^2 + 1)^{15}(\l_1+1)^{24}(\vert f\vert_2 + 1)(\vert f\vert_3
  + 1)}{\vfrak(\L)^3\s_0^6}\cdot \frac{p+n}{n^{3/2}}.
\end{equation}

Bounding $\Delta_2$ is straightforward.  Since $\vert f(\w) \vert \leq \vert
f \vert_0$, it follows that 
\begin{equation}\label{D2bd}
\Delta_2 \leq 2\vert f \vert_0 \P(A^c\vert X).
\end{equation}
Now we move on to $\Delta_3$.  In order to obtain the desired
bound, we need to do a little bit of preliminary work.  We begin by bounding $\E\left\{\left.\Vert
  J(\bth_0) - \J(\bth_0)\Vert^2\right\vert X\right\}$.  Let $J_{kl}(\bth)$ denote the $kl$-th
element of $J(\bth)$ and observe that 
\begin{align*}
J(\bth) & = \left[\begin{array}{cc} J_{11}(\bth) & J_{12}(\bth) \\
                    J_{21}(\bth) & J_{22}(\bth) \end{array}\right] \\
& =
                                   \left[\begin{array}{cc}
                                           \frac{1}{2\s^4} -
                                           \frac{1}{\s^6n}\sum_{i =
                                           1}^n
                                           \frac{\widecheck{y}_i^2}{\eta^2\l_i
                                           + 1} &
                                                                    -\frac{1}{2\s^4n}\sum_{i
                                                                    =
                                                                    1}^n
  \frac{\l_i\widecheck{y}_i^2 }{(\eta^2\l_i + 1)^2 }\\
                                           -\frac{1}{2\s^4n}\sum_{i =
                                           1}^n
                                           \frac{\l_i\widecheck{y}_i^2}{(\eta^2\l_i
                                           + 1)^2} &
                                                     \frac{1}{2n}\sum_{i
                                                     = 1}^n
                                                     \frac{\l_i^2}{(\eta^2
                                                     \l_i + 1)^2} -
                                                     \frac{1}{\s^2n}\sum_{i
                                                     = 1}^n
                                                     \frac{\l_i^2\widecheck{y}_i^2}{(\eta^2\l_i
                                                     + 1)^3}\end{array}\right],
\end{align*}
where $\widecheck{\y} = (\widecheck{y}_1,\ldots,\widecheck{y}_n)^{\T} = U^{\T}\y$.
Since the operator norm is bounded by the Hilbert-Schmidt norm,
\begin{equation}\label{vbound}
\E\left\{\left.\Vert J(\bth_0)- \J(\bth_0) \Vert^2\right\vert X\right\}
\leq \E\left\{\left.\Vert J(\bth_0)- \J(\bth_0) \Vert_{\HS}^2\right\vert
  X\right\} = \sum_{k,l=1}^2 \Var\left\{\left. J_{kl}(\bth_0)\right\vert
    X\right\}. 
\end{equation}
The variances on the right-hand side in \eqref{vbound} can be bounded
using Lemma \ref{lemma:vqf} from the Supplementary Material, since each term is the
variance of a quadratic form.  Indeed, 
\[
\Var\{J_{kl}(\bth_0)\vert X\} = \Var(\bz^{\T}Q_{kl}\bz\vert X),
\]
where $\bz = (\zeta_1,\ldots,\zeta_{n+p})^{\T} = (p^{1/2}\bb^{\T}/\tau_0,
\be^{\T}/\s_0)^{\T} \in \R^{n+p}$, $\tau_0^2 = \eta_0^2\s_0^2$, and
\begin{align*}
Q_{11} & = \frac{1}{\s_0^6n}\left[\begin{array}{c} (\tau_0/p^{1/2}) X^{\T}\\ \s_0
                  I\end{array}\right]
  \left(\frac{\eta_0^2}{p}XX^{\T} + I\right)^{-1}\left[\begin{array}{cc} (\tau_0/p^{1/2}) X & \s_0 I\end{array}\right],   \\ 
Q_{12} & = Q_{21} = \frac{1}{2\s_0^4n}\left[\begin{array}{c} (\tau_0/p^{1/2}) X^{\T}\\ \s_0 I\end{array}\right]\left(\frac{\eta_0^2}{p}XX^{\T}\right)\left(\frac{\eta_0^2}{p}XX^{\T}
  + I\right)^{-2}\left[\begin{array}{cc} (\tau_0/p^{1/2}) X & \s_0 I\end{array}\right],\\
Q_{22} & = \frac{1}{\s_0^2n}\left[\begin{array}{c} (\tau_0/p^{1/2}) X^{\T}\\ \s_0 I\end{array}\right]\left(\frac{\eta_0^2}{p}XX^{\T}\right)^2\left(\frac{\eta_0^2}{p}XX^{\T}
  + I\right)^{-3}\left[\begin{array}{cc} (\tau_0/p^{1/2}) X & \s_0 I\end{array}\right].
\end{align*}
By \eqref{subgbd}, $\E(\zeta_j^4) \leq
16\g^4(\eta_0^2+1)^2/(\eta_0^4\s_0^4)$.  Additionally, 
\[
\Vert Q_{11}\Vert,\Vert Q_{12}\Vert,\Vert Q_{22}\Vert \leq \frac{(\s_0^2+1)^2(\eta_0^2+1)(\l_1+1)}{\s_0^4n}.
\]
Thus, by Lemma \ref{lemma:vqf}, 
\[
\Var(\bz^{\T}Q_{kl}\bz\vert X ) \leq \frac{n+p}{n^2}\left\{\frac{16(\s_0^2+1)^6(\eta_0^2+1)^4(\g+1)^4(\l_1+1)^2}{\s_0^{12}\eta_0^4}\right\}, \ \ k,l=1,2.
\]
Combining this with \eqref{vbound}, we have 
\begin{equation}\label{deltaJ}
\E\left\{\left.\Vert J(\bth_0) - \J(\bth_0)\Vert^2\right\vert
  X\right\} \leq\frac{n+p}{n^2}\left\{\frac{64(\s_0^2+1)^6(\eta_0^2+1)^4(\g+1)^4(\l_1+1)^2}{\s_0^{12}\eta_0^4}\right\}.
\end{equation}
Thus, by \eqref{Jinvbd}, \eqref{deltaJ}, and the definition of $A$,
\begin{equation}\label{D3bd}
\Delta_3 \leq \frac{23\vert
  f\vert_1(\g+1)^2 (\s_0^2+1)^5(\eta_0^2+1)^6(\l_1 + 1)^7}{\s_0^4\eta_0^2\vfrak(\L)}\cdot\frac{\sqrt{n+p}}{n}\log(n).
\end{equation}

To  bound $\Delta_4$, we need a preliminary bound involving $\r$.  By some basic manipulations,
\begin{align*}
\E\left(\Vert\r\Vert\1_{A}\vert X\right) & \leq
                                         \frac{1}{2}\E\left[\left.\left\{
                                         \left\Vert
                                         \frac{\partial^2}{\partial\bth^2}
                                         S_1(\bth^*)\right\Vert
                                         + \left\Vert
                                         \frac{\partial^2}{\partial\bth^2}
                                         S_2(\bth^*)\right\Vert\right\}\Vert\hat{\bth}
                                         -
                                         \bth_0\Vert^2\1_{A}\right\vert
                                         X\right] \\
& \leq \frac{1}{\sqrt{2}}\left[\E\left\{\left. \left\Vert
                                         \frac{\partial^2}{\partial\bth^2}
                                         S_1(\bth^*)\right\Vert^2\1_{A}\right\vert
  X\right\} + \E\left\{\left. \left\Vert
                                         \frac{\partial^2}{\partial\bth^2}
                                         S_2(\bth^*)\right\Vert^2\1_{A}\right\vert
  X\right\}\right]^{1/2} \\
& \qquad \cdot \E\left(\left.\Vert \hat{\bth} -
  \bth_0\Vert^4\1_{A}\right\vert X\right)^{1/2}.
\end{align*}
Now we find the derivatives
\begin{align*}
\frac{\partial^2}{\partial \bth^2}S_1(\bth) & =
                                              \left[\begin{array}{cc}\frac{3}{\s^8n}\sum_{i
                                                      = 1}^n
                                                      \frac{\widecheck{y}_i^2}{\eta^2\l_i+1}-\frac{1}{\s^6}
                                                      &
                                                        \frac{1}{\s^6n}\sum_{i
                                                        = 1}^n
                                                        \frac{\l_i\widecheck{y}_i^2}{(\eta^2\l_i+1)^2}
                                                      \\
 \frac{1}{\s^6n}\sum_{i = 1}^n
  \frac{\l_i\widecheck{y}_i^2}{(\eta^2\l_i+1)^2} & \frac{1}{\s^4n}\sum_{i=
                                               1}^n
                                               \frac{\l_i^2\widecheck{y}_i^2}{(\eta^2\l_i
                                               + 1)^3}
                                                      \end{array}\right],\\
\frac{\partial^2}{\partial \bth^2}S_2(\bth) & =
                                              \left[\begin{array}{cc}
                                                      \frac{1}{\s^6n}\sum_{i
                                                      = 1}^n
                                                      \frac{\l_i\widecheck{y}_i^2}{(\eta^2\l_i+1)^2}
                                                      &
                                                        \frac{1}{\s^4n}\sum_{i=1}^n
                                                        \frac{\l_i^2\widecheck{y}_i^2}{(\eta^2\l_i+1)^3}
                                                      \\
                                                      \frac{1}{\s^4n}\sum_{i
                                                      =
                                                      1}^n\frac{\l_i^2\widecheck{y}_i^2}{(\eta^2\l_i+1)^3}
                                                      &
                                                        \frac{3}{\s^2n}\sum_{i
                                                        = 1}^n
                                                        \frac{\l_i^3\widecheck{y}_i^2}{(\eta^2\l_i+1)^4}
                                                        -\frac{1}{n}\sum_{i=1}^n \frac{\l_i^3}{(\eta^2\l_i+1)^3} \end{array}\right]
\end{align*}
and observe that on $A$,
\[
\left\Vert
                                         \frac{\partial^2}{\partial\bth^2}
                                         S_1(\bth^*)\right\Vert^2,
  \ \left\Vert
                                         \frac{\partial^2}{\partial\bth^2}
                                         S_2(\bth^*)\right\Vert^2
                                       \leq \frac{5888(\s_0^2+1)^8(\l_1+1)^6}{\s_0^{16}}\left(1+
                                         \frac{\Vert\y\Vert^4}{n^2}\right).
\]
Thus, 
\begin{align}\nonumber
\E\left(\Vert\r\Vert\1_{A}\vert X\right) & \leq \frac{77(\s_0^2 +
                                                1)^4(\l_1+1)^3}{\s_0^8}\left\{1
                                                +
     \frac{1}{n^2}                                           \E(\Vert\y\Vert^4\vert X
                                               )\right\}^{1/2}\E\left(\left.\Vert
                                                \hat{\bth} -
                                                \bth_0\Vert^4\1_{A}\right\vert
                                                X\right)^{1/2} \\\nonumber
& \leq \frac{77(\s_0^2 +
                                                1)^4(\l_1+1)^3}{\s_0^8}\left[1
                                                +
     \frac{4}{n^2}  \left\{\l_1^2\E(\Vert p^{1/2}\bb\Vert^4) +
  \E(\Vert\be\Vert^4)  \right\}
  \right]^{1/2}\\ \nonumber
& \qquad \cdot \E\left(\left.\Vert
                                                \hat{\bth} -
                                                \bth_0\Vert^4\1_{A}\right\vert
                                                X\right)^{1/2}\\ \nonumber
& \leq
  \frac{621(\g+1)^2(\s_0^2+1)^5(\eta_0^2+1)(\l_1+1)^4}{\s_0^8}\cdot
  \frac{p+n}{n} \\ \nonumber
& \qquad \cdot \E\left(\left.\Vert
                                                \hat{\bth} -
                                                \bth_0\Vert^4\1_{A}\right\vert
                                                X\right)^{1/2} \\ \label{rbd}
& \leq \frac{156(\g+1)^2(\s_0^2+1)^5(\eta_0^2+1)(\l_1+1)^4}{\s_0^4}\cdot
  \frac{p+n}{n^2} \log(n)^2.
\end{align}
Combining this bound with the definition of $\Delta_4$ yields
\begin{equation}\label{D4bd}
\Delta_4 \leq \frac{883\vert f\vert_1(\g+1)^2(\s_0^2+1)^7(\eta_0^2+1)^5(\l_1+1)^{10}}{\s_0^4 \vfrak(\L)}\cdot
  \frac{p+n}{n^{3/2}} \log(n)^2.
\end{equation}
Finally, the theorem follows by combining \eqref{normal_decomp}, \eqref{D1bd}--\eqref{D2bd},
\eqref{D3bd} and \eqref{D4bd}.

\bibliography{QF_ref}

\end{appendices}

\newpage
\vspace*{0.5in}

\begin{center}
{\Large \textbf{Supplementary material: \\ Flexible results for
    quadratic forms with \\ applications to
  variance components estimation}}
\end{center}

\vspace*{0.5in}

\setcounter{section}{0}
    \renewcommand{\thesection}{S\arabic{section}}

\setcounter{equation}{0}
    \renewcommand{\theequation}{S\arabic{equation}}

\setcounter{lemma}{0}
    \renewcommand{\thelemma}{S\arabic{lemma}}

\section{Proof of Proposition
  \ref{prop:RE_dconc}} \label{supp:RE_dconc}  To prove Proposition \ref{prop:RE_dconc}, we begin by retracing the steps
of the proof of Proposition \ref{prop:RE_conc}.  Following the proof
in of Proposition \ref{prop:RE_conc} in Appendix \ref{app:RE_conc}, we have
\begin{align*}
\left\{\Vert \tilde{\bth} - \bth_0 \Vert > r \right\} & \subseteq
                                                        \left\{\sup_{0
                                                        \leq \eta^2 <
                                                        \infty} \vert
                                                        \tilde{\s}_*^2(\eta^2)
                                                        -
                                                        \s_0^2(\eta^2)\vert
                                                        > \frac{r}{2\sqrt{2}}
                                                        \right\}  \\
  &
                                                                       \qquad 
                                                                       \cup  \left\{\sup_{0
                                                        \leq \eta^2 <
                                                        \infty} \vert
                                                        \tilde{H}_*(\eta^2)
                                                        -
                                                        H_0(\eta^2)\vert
                                                        >
                                                        \frac{\s_0^2r \vfrak(\L)}{32
              \sqrt{2}(\s_0^2+1)(\l_1+1)^5(\eta_0^2+1)^4}
                                                        \right\}
  \\  &
                                                                      \qquad \cup  \left\{\sup_{0
                                                        \leq \eta^2 <
                                                        \infty} \vert
                                                        \tilde{\ell}_*(\eta^2)
                                                        -
                                                        \ell_0(\eta^2)\vert
                                                        >
                                                      \frac{\eta_0^4
                                                        \chi(\eta_0^2,\L)\vfrak(\L)}{16
                                                       (\s_0^2 + 1)^2(\l_1+1)^2(\eta_0^2+1)^2}
                                                        \right\},
\end{align*}
where we have adopted the notation from Appendix \ref{app:RE_conc}, except
that a tilde indicates all of the $\y$'s in the corresponding quantity
are replaced by $\tilde{\y} = X\tilde{\bb} + \be$.  We further decompose the event $\{\Vert
\tilde{\bth} - \bth_0\Vert > r\}$ and obtain
\begin{equation}\label{incltilde}
\left\{\Vert \tilde{\bth} - \bth_0 \Vert > r \right\} \subseteq
E \cup (E_1 \cap E_{\circ}) \cup (E_2 \cap E_{\circ}) \cup (E_3 \cap
E_{\circ}\cap E_{\star})  \cup
E_{\circ}^c \cup E_{\star}^c
\end{equation}
where 
\begin{align*}
E & = \left\{\sup_{0
                                                        \leq \eta^2 <
                                                        \infty} \vert
                                                        \s_*^2(\eta^2)
                                                        -
                                                        \s_0^2(\eta^2)\vert
                                                        > \frac{r}{4\sqrt{2}}
                                                        \right\}  \\
  &
                                                                       \qquad 
                                                                       \cup  \left\{\sup_{0
                                                        \leq \eta^2 <
                                                        \infty} \vert
                                                        H_*(\eta^2)
                                                        -
                                                        H_0(\eta^2)\vert
                                                        >
                                                        \frac{\s_0^2r \vfrak(\L)}{64
              \sqrt{2}(\s_0^2+1)(\l_1+1)^5(\eta_0^2+1)^4}
                                                        \right\}
  \\  &
                                                                      \qquad \cup  \left\{\sup_{0
                                                        \leq \eta^2 <
                                                        \infty} \vert
                                                        \ell_*(\eta^2)
                                                        -
                                                        \ell_0(\eta^2)\vert
                                                        >
                                                      \frac{\eta_0^4
                                                        \chi(\eta_0^2,\L)\vfrak(\L)}{32
                                                       (\s_0^2 + 1)^2(\l_1+1)^2(\eta_0^2+1)^2}
                                                        \right\}, \\
E_1 & = \left\{\sup_{0
                                                        \leq \eta^2 <
                                                        \infty} \vert
                                                        \tilde{\s}_*^2(\eta^2)
                                                        -
                                                        \s_*^2(\eta^2)\vert
                                                        > \frac{r}{4\sqrt{2}}
                                                        \right\},  \\
  E_2&
                                                                       =
                                                                       \left\{\sup_{0
                                                        \leq \eta^2 <
                                                        \infty} \vert
                                                        \tilde{H}_*(\eta^2)
                                                        -
                                                        H_*(\eta^2)\vert
                                                        >
                                                        \frac{\s_0^2r \vfrak(\L)}{64
              \sqrt{2}(\s_0^2+1)(\l_1+1)^5(\eta_0^2+1)^4}
                                                        \right\},
  \\ E_3 &
                                                                     = \left\{\sup_{0
                                                        \leq \eta^2 <
                                                        \infty} \vert
                                                        \tilde{\ell}_*(\eta^2)
                                                        -
                                                        \ell_*(\eta^2)\vert
                                                        >
                                                      \frac{\eta_0^4
                                                        \chi(\eta_0^2,\L)\vfrak(\L)}{32
                                                       (\s_0^2 + 1)^2(\l_1+1)^2(\eta_0^2+1)^2}
                                                        \right\}, \\
E_{\circ} & = \left\{\frac{\s_0^2}{4} \leq \frac{1}{n}\Vert \y \Vert^2
            \leq 4\s_0^2(\eta_0^2\l_1 + 1)\right\}, \\
E_{\star} & = \left\{\sup_{0 \leq \eta^2 < \infty} \vert
            \s_*^2(\eta^2) - \tilde{\s}_*^2(\eta^2)\vert \leq \frac{\s_0^2}{8}\right\}.
\end{align*}
To prove the proposition, we bound the probability of the various events on the right-hand side
of \eqref{incltilde}.

First, from the proof of Proposition \ref{prop:RE_conc}, it follows that there
is a constant $0 < C < \infty$ such that
\begin{align}\label{Ebd1}
\P(E\vert X) & \leq C\exp\left[-\frac{n}{C}\cdot
               \frac{\kappa(\s_0^2,\eta_0^2,\L)}{\g^2(\g^2+1)}\cdot
               \frac{\vfrak(\L)^2}{\{\vfrak(\L)+1\}^2}\cdot
               \frac{r^2}{(r+1)^2}\right], \mbox{ if }  n_0=n, \\ \label{Ebd2}
\P(E\vert X) & \leq C\exp\left[-\frac{n}{C}\cdot
               \frac{\kappa(\s_0^2,\eta_0^2,\L)}{\g^2(\g^2+1)}\cdot
               \left(1 - \frac{n_0}{n}\right)^4\left(\frac{n_0}{n}\right)^2\cdot
               \frac{r^2}{(r+1)^2}\right], \mbox{ if }  n_0<n.
\end{align}
To bound $\P(E_1\cap E_{\circ}\vert X)$, note the inequalities
\begin{align*}
\vert\tilde{\s}_*^2(\eta^2) - \s_*^2(\eta^2)\vert & = \left\vert
                                                    \frac{1}{n}\tilde{\y}^{\T}\left(\frac{\eta^2}{p}XX^{\T}
                                                    +
                                                    I\right)^{-1}\tilde{\y}
                                                    - \frac{1}{n}\y^{\T}\left(\frac{\eta^2}{p}XX^{\T}
                                                    +
                                                    I\right)^{-1}\y\right\vert
  \\
& \leq  \left\vert \frac{1}{n}(\bb - \tilde{\bb})^{\T}X^{\T}\left(\frac{\eta^2}{p}XX^{\T}
                                                    +
                                                    I\right)^{-1}X(\bb
  - \tilde{\bb})\right\vert  \\
& \qquad +\left\vert \frac{2}{n}(\bb - \tilde{\bb})^{\T}X^{\T}\left(\frac{\eta^2}{p}XX^{\T}
                                                    +
                                                    I\right)^{-1}\y
  \right\vert \\ 
& \leq \frac{\l_1 p}{n}\Vert \bb - \tilde{\bb}\Vert^2 +
  \frac{2 \l_1^{1/2} p^{1/2}}{n}\Vert \y\Vert \Vert \bb - \tilde{\bb}\Vert.
\end{align*}
Thus, on $E_{\circ}$, 
\[
\vert\tilde{\s}_*^2(\eta^2) - \s_*^2(\eta^2)\vert \leq 4\left(\frac{p
    + n}{n}\right)(\l_1+1)(\s_0^2 +
1)^{1/2}(\eta_0^2 + 1)^{1/2}\Vert \bb - \tilde{\bb}\Vert\left(1 + \Vert \bb - \tilde{\bb}\Vert\right)
\]
and it follows that
\begin{align}\nonumber
E_1 \cap E_{\circ}  
& \subseteq \left\{\Vert \bb - \tilde{\bb} \Vert \left(1  + \Vert \bb
  - \tilde{\bb}\Vert\right) >
  \frac{r}{16\sqrt{2}(\l_1+1)(\s_0^2+1)^{1/2}(\eta_0^2+1)^{1/2}}\left(\frac{n}{n+p}\right)\right\}
  \\ \label{E1incl1}
& \subseteq \left\{\Vert \bb - \tilde{\bb}\Vert^2 >
  \frac{1}{1024(\l_1+1)^2(\s_0^2+1)(\eta_0^2+1)}\left(\frac{n}{n+p}\right)^2\frac{r^2}{1 + r}\right\}.
\end{align}
Hence, there is a constant $0 < C_1 < \infty$ such that
\begin{equation} \label{E1bd1}
 \P(E_1 \cap E_{\circ} \vert X) \leq \P \left\{\left.\Vert \bb -
   \tilde{\bb}\Vert >
   \frac{1}{C_1}\cdot\kappa(\s_0^2,\eta_0^2,\L)\cdot\frac{n}{p+n}\cdot\frac{r}{1
     + r^{1/2}}\right\vert X\right\}
\end{equation}

Next, to bound $\P(E_2 \cap E_{\circ}\vert X)$, 
\begin{align*}
\vert \tilde{H}_*(\eta^2) & - H_*(\eta^2)\vert \\ 
& \leq \left\vert
                                               \frac{1}{n}(\tilde{\y}
  -\y)^{\T}\left(\frac{1}{p}XX^{\T}\right)
                                               \left(\frac{\eta^2}{p}XX^{\T}
                                               +
  I\right)^{-2}(\tilde{\y} - \y) \right\vert \\
& \qquad + \left\vert
                                               \frac{2}{n}(\tilde{\y}
  -\y)^{\T}\left(\frac{1}{p}XX^{\T}\right)
                                               \left(\frac{\eta^2}{p}XX^{\T}
                                               +
  I\right)^{-2}\y \right\vert + \l_1\vert \s_*^2(\eta^2) -
  \tilde{\s}^2_*(\eta^2)\vert \\
& \leq \frac{\l_1^2 p}{n}\Vert \bb - \tilde{\bb}\Vert^2 + \frac{2 \l_1^{3/2} p^{1/2}}{n}
  \Vert \bb - \tilde{\bb} \Vert\Vert\y\Vert +  \l_1\vert \s_*^2(\eta^2) -
  \tilde{\s}^2_*(\eta^2)\vert \\
& \leq \frac{2 \l_1^2 p }{n}\Vert \bb - \tilde{\bb}\Vert^2 +
  \frac{4 \l_1^{3/2} p^{1/2}}{n}\Vert\y\Vert\Vert \bb -
  \tilde{\bb}\Vert.  
\end{align*}
Hence, on $E_{\circ}$, 
\[
\vert\tilde{H}_*(\eta^2) - H_*(\eta^2)\vert  \leq 4\left(\frac{p
    + n}{n}\right)(\l_1+1)^2(\s_0^2 +
1)^{1/2}(\eta_0^2 + 1)^{1/2}\Vert \bb - \tilde{\bb}\Vert\left(1 + \Vert \bb - \tilde{\bb}\Vert\right)
\]
and
\begin{align*}
\bigg\{\sup_{0
                                                        \leq \eta^2 <
                                                        \infty} &  \vert
                                                        \tilde{H}_*(\eta^2)
                                                        -
                                                        H_*(\eta^2)\vert
                                                        >
                                                        \frac{\s_0^2r \vfrak(\L)}{64
              \sqrt{2}(\s_0^2+1)(\l_1+1)^5(\eta_0^2+1)^4}
                                                        \bigg\} \cap
                                                                  E_{\circ}\\ 
& \subseteq \left\{\Vert \bb - \tilde{\bb}\Vert\left(1 + \Vert \bb -
  \tilde{\bb}\Vert\right) > \frac{\s_0^2\vfrak(\L)
  r}{256\sqrt{2}(\s_0^2+1)^{3/2}(\eta_0^2+1)^{9/2}(\l_1+1)^7}\left(\frac{n}{p+n}\right)\right\}
  \\ \label{Htildebd}
& \subseteq \left\{\Vert \bb - \tilde{\bb}\Vert^2 >
  \frac{\s_0^4}{262144(\s_0^2+1)^3(\eta_0^2+1)^9(\l_1+1)^{14}}\left(\frac{n}{p+n}\right)^2\frac{\vfrak(\L)^2}{1+\vfrak(\L)}
  \cdot \frac{r^2}{1+r}\right\}.
\end{align*}
Consequently, there is a constant $0 < C_2 < \infty$ such that
\begin{equation}\label{E2bd1}
\P(E_2 \cap E_{\circ}\vert X) \leq \P\left\{\Vert \tilde{\bb} -
  \bb\Vert >
  \frac{1}{C_2}\cdot \kappa(\s_0^2,\eta_0^2,\L)\cdot \frac{n}{n+p} \cdot \frac{\vfrak(\L)}{1
  + \vfrak(\L)^{1/2}}\cdot\frac{r}{1+r^{1/2}}\right\}.
\end{equation}

Now we bound $\P(E_3 \cap E_{\circ} \cap E_{\star}\vert X)$.  Note
that on the event $E_{\circ}$, 
\[
\s_*^2(\eta^2) = \frac{1}{n}\y^{\T}\left(\frac{\eta^2}{p}XX^{\T} +
  I\right)^{-1}\y \geq \frac{1}{n(\eta^2\l_1 + 1)}\Vert\y\Vert^2 \geq
\frac{\s_0^2}{4(\eta^2\l_1+1)}.
\]
It follows that if we are on the event $E_{\circ} \cap E_{\star}$,
then 
\begin{align*}
\vert \tilde{\ell}_*(\eta^2) - \ell_*(\eta^2)\vert & \leq
                                                     \frac{1}{2}\left\vert
                                                     \log\left\{
                                                     \frac{\tilde{\s}_*^2(\eta^2)}{\s_*^2(\eta^2)}\right\}\right\vert
  \\
& \leq \frac{1}{2}\left\vert \frac{\s_*^2(\eta^2) -
  \tilde{\s}_*^2(\eta^2)}{\s_*^2(\eta)}\right\vert + \frac{1}{2}\left\vert \frac{\s_*^2(\eta^2) -
  \tilde{\s}_*^2(\eta^2)}{\tilde{\s}_*^2(\eta)}\right\vert \\
& \leq \frac{6}{\s_0^2}(\eta^2 \l_1+1)\vert \s_*^2(\eta^2) -
  \tilde{\s}_*^2(\eta^2)\vert \\
& \leq \frac{6p}{\s_0^2n}\Vert \bb - \tilde{\bb}\Vert^2
  + \frac{12 \l_1^{1/2} p^{1/2}}{\s_0^2\l_{n_0}^{1/2}n}\Vert \y \Vert\Vert
  \bb - \tilde{\bb}\Vert \\
& \leq \frac{6p}{\s_0^2n}\Vert \bb - \tilde{\bb}\Vert^2
  + \frac{48\l_1^{1/2}p^{1/2}}{\l_{n_0}^{1/2}n}(\eta_0^2\l_1+1)\Vert
  \bb - \tilde{\bb}\Vert \\
& \leq  \frac{48(\s_0^2+1) (\eta_0^2+1)(\l_1+1)^{3/2}(\l_{n_0}^{-1}+1)^{1/2}}{\s_0^2}\cdot\frac{p}{n}\Vert
  \bb - \tilde{\bb}\Vert\left(1 + \Vert
  \bb - \tilde{\bb}\Vert\right).
\end{align*}
Thus, 
\begin{align*}
E_3 \cap E_{\circ} \cap E_{\star} 
& \subseteq \left\{\Vert
  \bb - \tilde{\bb}\Vert\left(1 + \Vert
  \bb - \tilde{\bb}\Vert\right) >
  \frac{\eta_0^4\s_0^2\chi(\eta_0^2,\L)\vfrak(\L)n/(p+n)}{1536(\s_0^2+1)^3(\l_1+1)^{7/2}(\eta_0^2+1)^3}\right\}
  \\
& \subset \left\{\Vert \bb - \tilde{\bb} \Vert >
  \frac{\eta_0^4\s_0^2\chi(\eta_0^2,\L)n/(p+n)}{1536(\s_0^2+1)^3(\l_1+1)^{7/2}(\eta_0^2+1)^3}\cdot\frac{\vfrak(\L)}{1
  + \vfrak(\L)^{1/2}}\right\} 
\end{align*}
and there exists a constant $0 < C_3 < \infty$ such that
\begin{equation}\label{E3bd1}
\P(E_3 \cap E_{\circ} \cap E_{\star}\vert X) \leq
\P\left\{\left.\Vert \tilde{\bb} - \bb\Vert >
    \frac{1}{C_3}\cdot \kappa(\s_0^2,\eta_0^2,\L) \cdot \frac{n}{p+n}\cdot
    \frac{\vfrak(\L)}{1 + \vfrak(\L)^{1/2}}\right\vert X\right\}.
\end{equation}

Finally, it remains to bound $\P(E_{\circ}^c\vert X)$ and
$\P(E_{\star}^c\vert X)$.  A bound on the former follows directly from
Theorem \ref{thm:QF_conc}, which implies that  there is a constant $0 <
C_{\circ} < \infty$ such that 
\begin{equation}\label{Ecircbd1}
\P(E_{\circ}^c\vert X) \leq
C_{\circ}\exp\left\{-\frac{n}{C_{\circ}}\cdot \frac{\kappa(\s_0^2,\eta_0^2,\L)}{\g^2(\g^2+1)}\right\}.
\end{equation}  
Bounding $\P(E_{\star}^c\vert X)$ is also easy: Just replace
$r/(4\sqrt{2})$ with $\s_0^2/8$ in the definition of $E_1$ and use
\eqref{E1incl1} to obtain
\begin{equation}\label{Estarbd1}
\P(E_{\star}^c\vert X) \leq \P\left\{\Vert \bb - \tilde{\bb} \Vert > \frac{1}{C_{\star}}\cdot
  \kappa(\s_0^2,\eta_0^2,\L)\cdot \frac{n}{p+n}\right\}.
\end{equation}
The proposition follows by combining \eqref{incltilde}--\eqref{Ebd2}
and \eqref{E1bd1}--\eqref{Estarbd1}.

\section{Supporting lemmas}\label{supp:lemmas}

\begin{lemma}\label{lemma:Tid}
Let $t_{ij}^{kl}$ be as in \eqref{tentry11}--\eqref{tentry22}, from the proof of Theorem
\ref{thm:QF_napprox}.   Additionally, let $\mu_i^{(m)} = \E(\zeta_i^m)$. Then
\begin{align}\nonumber
t_{11}^{kl} & =\frac{4}{d} \sum_{m =1}^d \sum_{1 \leq i \neq j \leq d}
  \zeta_i\zeta_j(1 + \zeta_m^2)q_{mi}^{(k)}q_{mj}^{(l)} + \frac{2}{d} \sum_{1 \leq i \neq j \leq d}
  \zeta_j(\mu_i^{(3)} - 3\zeta_i - \zeta_i^3)(q_{ij}^{(k)}q_{ii}^{(l)} +
  q_{ii}^{(k)}q_{ij}^{(l)}) \\ \label{Tid11}
& \quad + \frac{4}{d}
  \sum_{1 \leq i \neq j \leq d} \{(\zeta_i^2\zeta_j^2 - 1) +
  (\zeta_j^2-1)\}q_{ij}^{(k)}q_{ij}^{(l)} \\ \nonumber
& \qquad + \frac{1}{d}\sum_{i=
  1}^d\{(\zeta_i^4-\mu_i^{(4)}) - 2(\zeta_i^2 - 1)\}q_{ii}^{(k)}q_{ii}^{(l)} , \\ \label{Tid12}
t_{12}^{kl} & = \frac{2}{d}\sum_{1
  \leq i \neq j \leq d} \zeta_j(\mu_i^{(3)} - \zeta_i +
  \zeta_i^3)q_{ij}^{(k)}q_{ii}^{(l)} + \frac{1}{d}\sum_{i = 1}^d\{(\zeta_i^4 - \mu_i^{(4)}) -2(\zeta_i^2 - 1)\}q_{ii}^{(k)}q_{ii}^{(l)},\\ \label{Tid21}
t_{21}^{kl} & = \frac{2}{d}\sum_{1 \leq i \neq j \leq d}
              \zeta_j(\mu_i^{(3)} - \zeta_i + \zeta_i^3)q_{ii}^{(k)}q_{ij}^{(l)}+ \frac{1}{d} \sum_{i = 1}^d \{(\zeta_i^4 - \mu_i^{(4)}) -
  2(\zeta_i^2 - 1)\}q_{ii}^{(k)}q_{ii}^{(l)}, \\ \label{Tid22}
t_{22}^{kl} & = \frac{1}{d}\sum_{i = 1}^d \{(\zeta_i^4 - \mu_i^{(4)})
              - 2(\zeta_i^2 - 1)\} q_{ii}^{(k)}q_{ii}^{(l)}.
\end{align}
\end{lemma}

\begin{proof}
Let $\w_k = (w_k,\check{w}_k)^{\T}
\in \R^2$ and
write 
\[
V = \left[\begin{array}{cccc} V_{11} & V_{12} & \cdots & V_{1K} \\
            V_{21}& V_{22} & \cdots & V_{2K} \\
\vdots & \vdots & \ddots & \vdots \\ V_{K1} & V_{K2} & \cdots& V_{KK} \end{array}\right],
\]
where
\[
V_{kl} = 
\left[\begin{array}{cc}\E(w_kw_l) & \E(w_k\check{w}_l) \\ \E(\check{w}_kw_l) &
                                                                   \E(\check{w}_k\check{w}_l) \end{array}\right]
                                                               = \E(\w_k\w_l^{\T}).  
\]
Then \eqref{exvar} implies that
\[
\E\{(\w' - \w)(\w' - \w)^{\T}\} = 2\left[\begin{array}{cccc} V_{11}\L_1^{\T} & V_{12}\L_1^{\T} & \cdots & V_{1K}\L_1^{\T} \\
            V_{21}\L_1^{\T}& V_{22}\L_1^{\T} & \cdots & V_{2K}\L_1^{\T} \\
\vdots & \vdots & \ddots & \vdots \\ V_{K1}\L_1^{\T} & V_{K2}\L_1^{\T} & \cdots& V_{KK}\L_1^{\T} \end{array}\right].
\]
Then 
\begin{align} \label{t11simp}
t_{11}^{kl} & = \E\{(w_k' - w_k)(w_l' - w_l)|\bz\} -
              \frac{2}{d}\{2\E(w_kw_l) - \E(w_k\check{w}_l)\}, \\ \label{t12simp}
t_{12}^{kl} & = \E\{(w_k' - w_k)(\check{w}_k' - \check{w}_k)|\bz\} -
              \frac{2}{d}\E(w_k\check{w}_l),  \\\label{t21simp}
t_{21}^{kl} & = \E\{(\check{w}_k' - \check{w}_k)(w_l' - w_l)|\bz\} -
              \frac{2}{d}\{2\E(\check{w}_kw_l) - \E(\check{w}_k\check{w}_l)\}, \\ \label{t22simp}
t_{22}^{kl} & = \E\{(\check{w}_k' - \check{w}_k)(\check{w}_l' - \check{w}_l)|\bz\} - \frac{2}{d}\E(\check{w}_k\check{w}_l).
\end{align}
Proving each of \eqref{Tid11}--\eqref{Tid22} is now a straightforward,
yet tedious calculation, though \eqref{Tid12}--\eqref{Tid22} are equivalent.  We begin with \eqref{Tid22} (the simplest
identity) and work our way up to \eqref{Tid11}. 

We have:
\begin{align*}
t_{22}^{kl} & = \E\{(\check{w}_k' - \check{w}_k)(\check{w}_l' - \check{w}_l)|\bz\} - \frac{2}{d}\E(\check{w}_k \check{w}_l) \\ 
& = \E\left[\{(\zeta_{\i}')^2 -
  \zeta_{\i}^2\}^2(\mathbf{e}_{\i}^{\T}Q_k\mathbf{e}_{\i})(\mathbf{e}_{\i}^{\T}Q_l\mathbf{e}_{\i})|\bz\right]
  \\ & \qquad -
\frac{2}{d}\E\left\{(\bz^{\T}Q_k^{\mathrm{D}}\bz)(\bz^{\T}Q_l^{\mathrm{D}}\bz)\right\}  + 
              \frac{2}{d}\E(\bz^{\T}Q_k^{\mathrm{D}}\bz)\E (\bz^{\T}Q_l^{\mathrm{D}}\bz) \\
& = \E\left[\{(\zeta_{\i}')^2 -
  \zeta_{\i}^2\}^2(\mathbf{e}_{\i}^{\T}Q_k\mathbf{e}_{\i})(\mathbf{e}_{\i}^{\T}Q_l\mathbf{e}_{\i})|\bz\right]
 - \frac{2}{d}\E\left[\sum_{i,j=1}^{d}
  \zeta_i^2\zeta_j^2q_{ii}^{(k)}q_{jj}^{(l)}\right]+\frac{2}{d}\tr(Q_k)\tr(Q_l) \\
& = \frac{1}{d}\sum_{i = 1}^{d}\E\left[\{(\zeta_i')^2 -
  \zeta_i^2\}^2|\bz\right]q_{ii}^{(k)}q_{ii}^{(l)} -
  \frac{2}{d}\sum_{i = 1}^d (\mu_i^{(4)}-1)q_{ii}^{(k)}q_{ii}^{(l)}  \\
& =  \frac{1}{d} \sum_{i = 1}^d \mu_i^{(4)}q_{ii}^{(k)}q_{ii}^{(l)} - \frac{2}{d}\sum_{i= 1}^{d}
  \zeta_i^2q_{ii}^{(k)}q_{ii}^{(l)} + \frac{1}{d}\sum_{i = 1}^{d}
  \zeta_i^4q_{ii}^{(k)}q_{ii}^{(l)} - 
  \frac{2}{d}\sum_{i = 1}^d (\mu_i^{(4)}-1)q_{ii}^{(k)}q_{ii}^{(l)} \\
& = \frac{1}{d}\sum_{i = 1}^d \{(\zeta_i^4 - \mu_i^{(4)}) - 2(\zeta_i^2
  - 1)\}q_{ii}^{(k)}q_{ii}^{(l)}.
\end{align*}
Thus, \eqref{Tid22}.  Next, we prove \eqref{Tid12}:
\begin{align*}
t_{12}^{kl} & = \E\{(w_k'-w_k)(\check{w}_k' -\check{w}_k)|\bz\} - \frac{2}{d}
              \E(w_k\check{w}_l) \\
& = \E\left[2(\zeta_{\i}' - \zeta_{\i})\{(\zeta_{\i}')^2 -
  \zeta_{\i}^2\}(\mathbf{e}_{\i}^{\T}Q_k\bz)(\mathbf{e}_{\i}^{\T}Q_l\mathbf{e}_{\i})|\bz\right]
  \\
& \qquad 
  + \E\left[(\zeta_{\i}' - \zeta_{\i})^2\{(\zeta_{\i}')^2 -
  \zeta_{\i}^2\}(\mathbf{e}_{\i}^{\T}Q_k\mathbf{e}_{\i})(\mathbf{e}_{\i}^{\T}Q_l\mathbf{e}_{\i})|\bz\right]
  \\
& \qquad 
  - \frac{2}{d} \E\{(\bz^{\T}Q_k\bz)(\bz^{\T}Q_l^{\mathrm{D}}\bz)\} +
  \frac{2}{d}\E(\bz^{\T}Q_k\bz)\E(\bz^{\T}Q_l^{\mathrm{D}}\bz) \\
& = \frac{2}{d}\sum_{i = 1}^d  \E\left[(\zeta_i' - \zeta_i)\{(\zeta_i')^2 -
  \zeta_i^2\}(\mathbf{e}_i^{\T}Q_k\bz)(\mathbf{e}_i^{\T}Q_l\mathbf{e}_i)|\bz\right]
  \\
& \qquad + \frac{1}{d}\sum_{i = 1}^d \E\left[(\zeta_i' - \zeta_i)^2\{(\zeta_i')^2 -
  \zeta_i^2\}(\mathbf{e}_i^{\T}Q_k\mathbf{e}_i)(\mathbf{e}_i^{\T}Q_i\mathbf{e}_i)|\bz\right]
  \\
& \qquad -\frac{2}{d}\sum_{i,j,m=1}^d
  \E(\zeta_i\zeta_j\zeta_m^2)q_{ij}^{(k)}q_{mm}^{(l)} +
  \frac{2}{d}\sum_{i,j=1}^d q_{ii}^{(k)}q_{jj}^{(l)} \\
& = \frac{2}{d}\sum_{i = 1}^d \E\left[\{(\zeta_i')^3 -
  (\zeta_i')^2\zeta_i - \zeta_i'\zeta_i^2 +
  \zeta_i^3\}(\mathbf{e}_i^{\T}Q_k\bz)(\mathbf{e}_i^{\T}Q_l\mathbf{e}_i)|\bz\right]
  \\ 
& \qquad + \frac{1}{d}\sum_{i = 1}^d \E\left[\{(\zeta_i')^4 - 2(\zeta_i')^3\zeta_i +
    2\zeta_i'\zeta_i^3 - \zeta_i^4\}(\mathbf{e}_i^{\T}Q_k\mathbf{e}_i)(\mathbf{e}_i^{\T}Q_l\mathbf{e}_i)|\bz\right]\\
  & \qquad - 
    \frac{2}{d}\sum_{i = 1}^d (\mu_i^{(4)} -
    1)q_{ii}^{(k)}q_{jj}^{(l)} \\
& = \frac{2}{d}\sum_{i,j = 1}^d \zeta_j(\mu_i^{(3)} - \zeta_i +
  \zeta_i^3)q_{ij}^{(k)}q_{ii}^{(l)} + \frac{1}{d}\sum_{i = 1}^d (\mu_i^{(4)} -
  2\mu_i^{(3)}\zeta_i - \zeta_i^4)q_{ii}^{(k)}q_{ii}^{(l)} \\& \qquad - 
    \frac{2}{d}\sum_{i = 1}^d (\mu_i^{(4)} -
    1)q_{ii}^{(k)}q_{jj}^{(l)}  \\
& = \frac{2}{d}\sum_{i = 1}^d (\zeta_i\mu_i^{(3)} - \zeta_i^2 +
  \zeta_i^4)q_{ii}^{(k)}q_{ii}^{(l)} + \frac{2}{d}\sum_{1 \leq i \neq
  j \leq d} \zeta_j(\mu_i^{(3)} - \zeta_i +
  \zeta_i^3)q_{ij}^{(k)}q_{ii}^{(l)}  \\
& \qquad + \frac{1}{d}\sum_{i = 1}^d (\mu_i^{(4)} -
  2\mu_i^{(3)}\zeta_i - \zeta_i^4)q_{ii}^{(k)}q_{ll}^{(l)} \\& \qquad - 
    \frac{2}{d}\sum_{i = 1}^d (\mu_i^{(4)} -
    1)q_{ii}^{(k)}q_{jj}^{(l)}  \\
& = \frac{2}{d}\sum_{1
  \leq i \neq j \leq d} \zeta_j(\mu_i^{(3)} - \zeta_i +
  \zeta_i^3)q_{ij}^{(k)}q_{ii}^{(l)} + \frac{1}{d}\sum_{i = 1}^d\{(\zeta_i^4 - \mu_i^{(4)}) -2(\zeta_i^2 - 1)\}q_{ii}^{(k)}q_{ii}^{(l)}.
\end{align*}
The identity \eqref{Tid21} follows immediately from \eqref{Tid12} by
symmetry.  Finally, we prove \eqref{Tid11}:
\begin{align*}
t_{11}^{kl} & = \E\{(w_k' - w_k)(w_l' - w_l)|\bz\} -
              \frac{4}{d}\E(w_kw_l) + \frac{2}{d}\E(w_k\check{w}_l) \\
& = \E\left[4(\zeta_{\i}' -
  \zeta_{\i})^2(\mathbf{e}_{\i}^{\T}Q_k\bz)(\mathbf{e}_{\i}^{\T}Q_l\bz)|\bz\right]
  + \E\left[2(\zeta_{\i}' -
  \zeta_{\i})^3(\mathbf{e}_{\i}^{\T}Q_k\bz)(\mathbf{e}_{\i}^{\T}Q_l\mathbf{e}_{\i})|\bz\right]
  \\
& \qquad + \E\left[2(\zeta_{\i}' -
  \zeta_{\i})^3(\mathbf{e}_{\i}^{\T}Q_k\mathbf{e}_{\i})(\mathbf{e}_{\i}^{\T}Q_l\bz)|\bz\right]
  + \E\left[(\zeta_{\i}' - \zeta_{\i})^4
  (\mathbf{e}_{\i}^{\T}Q_k\mathbf{e}_{\i})(\mathbf{e}_{\i}^{\T}Q_l\mathbf{e}_{\i})|\bz\right]
  \\
& \qquad - \frac{4}{d}\E\{(\bz^{\T}Q_k\bz)(\bz^{\T}Q_l\bz)\} +
  \frac{4}{d}\E(\bz^{\T}Q_k\bz)\E(\bz^{\T}Q_l\bz) \\
& \qquad + \frac{2}{d}\E\{(\bz^{\T}Q_k\bz)(\bz^{\T}Q_l^{\mathrm{D}}\bz)\} -
  \frac{2}{d}\E(\bz^{\T}Q_k\bz)\E(\bz^{\T}Q_l^{\mathrm{D}}\bz) \\
& = \frac{4}{d}\sum_{i  =1}^{d}\E\left[(\zeta_i' -
  \zeta_i)^2(\mathbf{e}_i^{\T}Q_k\bz)(\mathbf{e}_i^{\T}Q_l\bz)|\bz\right]
  + \frac{2}{d}\sum_{i =1}^{d}\E\left[(\zeta_i' -
  \zeta_i)^3(\mathbf{e}_i^{\T}Q_k\bz)(\mathbf{e}_i^{\T}Q_l\mathbf{e}_i)|\bz\right]
  \\
& \qquad + \frac{2}{d}\sum_{i = 1}^{d} \E\left[(\zeta_i' -
  \zeta_i)^3(\mathbf{e}_i^{\T}Q_k\mathbf{e}_i)(\mathbf{e}_i^{\T}Q_l\bz)|\bz\right]
  + \frac{1}{d}\sum_{i= 1}^{d} \E\left[(\zeta_i' - \zeta_i)^4
  (\mathbf{e}_i^{\T}Q_k\mathbf{e}_i)(\mathbf{e}_i^{\T}Q_l\mathbf{e}_i)|\bz\right]
  \\ & \qquad -\frac{4}{d}\sum_{i,j,m,n=1}^d
       \E(\zeta_i\zeta_j\zeta_m\zeta_n)q_{ij}^{(k)}q_{mn}^{(l)} +
       \frac{2}{d}\sum_{i,j,m=1}^d
       \E(\zeta_i\zeta_j\zeta_m^2)q_{ij}^{(k)}q_{mm}^{(l)} +
       \frac{2}{d}\sum_{i,j=1}^d q_{ii}^{(k)}q_{jj}^{(l)} \\
& = \frac{4}{d}\sum_{i  =1}^{d}\E\left[\{(\zeta_i')^2 -
  2\zeta_i'\zeta_i +
  \zeta_i^2\}(\mathbf{e}_i^{\T}Q_k\bz)(\mathbf{e}_i^{\T}Q_l\bz)|\bz\right]
  \\
& \qquad 
  + \frac{2}{d}\sum_{i =1}^{d}\E\left[\{(\zeta_i')^3 -
  3(\zeta_i')^2\zeta_i + 3\zeta_i'\zeta_i^2 -
  \zeta_i^3\}\{(\mathbf{e}_i^{\T}Q_k\bz)(\mathbf{e}_i^{\T}Q_l\mathbf{e}_i)
  + (\mathbf{e}_i^{\T}Q_k\mathbf{e}_i)(\mathbf{e}_i^{\T}Q_l\bz)\}|\bz\right]
  \\
& \qquad 
  + \frac{1}{d}\sum_{i= 1}^{d} \E\left[\{(\zeta_i')^4 -
  4(\zeta_i')^3\zeta_i + 6(\zeta_i')^2\zeta_i^2 -4\zeta_i'\zeta_i^3 + \zeta_i^4\}
  (\mathbf{e}_i^{\T}Q_k\mathbf{e}_i)(\mathbf{e}_i^{\T}Q_l\mathbf{e}_i)|\bz\right]
  \\
& \qquad - \frac{4}{d}\sum_{1 \leq m\neq n \leq d}\sum_{i,j=1}^d
  \E(\zeta_i\zeta_j\zeta_m\zeta_n)q_{ij}^{(k)}q_{mn}^{(l)} -\frac{2}{d}
  \sum_{i,j,m = 1}^d
  \E(\zeta_i\zeta_j\zeta_m^2)q_{ij}^{(k)}q_{mm}^{(l)} +
  \frac{2}{d}\sum_{i,j = 1}^d q_{ii}^{(k)}q_{jj}^{(l)} \\ 
& = \frac{4}{d}\sum_{i  =1}^{d} (1+\zeta_i^2)(\mathbf{e}_i^{\T}Q_k\bz)(\mathbf{e}_i^{\T}Q_l\bz)
  \\
& \qquad 
  + \frac{2}{d}\sum_{i =1}^{d} (\mu_i^{(3)} -
  3\zeta_i -
  \zeta_i^3)\{(\mathbf{e}_i^{\T}Q_k\bz)(\mathbf{e}_i^{\T}Q_l\mathbf{e}_i)
  + (\mathbf{e}_i^{\T}Q_k\mathbf{e}_i)(\mathbf{e}_i^{\T}Q_l\bz)\}
  \\
& \qquad 
  + \frac{1}{d}\sum_{i= 1}^{d} (\mu_i^{(4)} -
  4\mu_i^{(3)}\zeta_i + 6\zeta_i^2 + \zeta_i^4)
  (\mathbf{e}_i^{\T}Q_k\mathbf{e}_i)(\mathbf{e}_i^{\T}Q_l\mathbf{e}_i)
  \\
& \qquad - \frac{8}{d}\sum_{1 \leq i \neq j \leq d} q_{ij}^{(k)}q_{ij}^{(l)} -
  \frac{2}{d}\sum_{i = 1}^d (\mu_i^{(4)} - 1)q_{ii}^{(k)}q_{ii}^{(l)}
  \\
& = \frac{4}{d} \sum_{i,j,m=1}^d \zeta_i\zeta_j(1 +
  \zeta_m^2)q_{mi}^{(k)}q_{mj}^{(l)} + \frac{2}{d}\sum_{i,j=1}^d \zeta_j(\mu_i^{(3)} - 3\zeta_i -
  \zeta_i^3)(q_{ij}^{(k)}q_{ii}^{(l)} + q_{ii}^{(k)}q_{ij}^{(l)}) \\
& \qquad + \frac{1}{d}\sum_{i = 1}^d (\mu_i^{(4)} -
  4\mu_i^{(3)}\zeta_i + 6\zeta_i^2 +
  \zeta_i^4)q_{ii}^{(k)}q_{ii}^{(l)} - \frac{8}{d}\sum_{i,j = 1}^d q_{ij}^{(k)}q_{ij}^{(l)} -
  \frac{2}{d}\sum_{i = 1}^d (\mu_i^{(4)} - 5)q_{ii}^{(k)}q_{ii}^{(l)}
  \\
& = \frac{4}{d} \sum_{m =1}^d \sum_{1 \leq i \neq j \leq d}
  \zeta_i\zeta_j(1 + \zeta_m^2)q_{mi}^{(k)}q_{mj}^{(l)} + \frac{2}{d} \sum_{1 \leq i \neq j \leq d}
  \zeta_j(\mu_i^{(3)} - 3\zeta_i - \zeta_i^3)(q_{ij}^{(k)}q_{ii}^{(l)} +
  q_{ii}^{(k)}q_{ij}^{(l)}) \\
& \qquad + \frac{4}{d}
  \sum_{1 \leq i \neq j \leq d} \{(\zeta_i^2\zeta_j^2 - 1) + (\zeta_j^2-1)\}q_{ij}^{(k)}q_{ij}^{(l)} + \frac{1}{d}\sum_{i=
  1}^d\{(\zeta_i^4-1) - 2(\zeta_i^2 - 1)\}q_{ii}^{(k)}q_{ii}^{(l)}. 
\end{align*}
The identity \eqref{Tid11} and the lemma follow.
\end{proof}

\begin{lemma}\label{lemma:Tbd}
Assume that the conditions of Theorem
\ref{thm:QF_napprox}  and Lemma \ref{lemma:Tid} hold and let $c(\g) = 4096
(\g + 1)^8$.  Then
\begin{align}\label{Tbd11}
\E\{(t_{11}^{kl})^2\} & \leq \frac{8\{108c(\g)^2 + 763c(\g) + 930\}}{d} \Vert
Q_k\Vert^2 \Vert Q_l\Vert^2 \\ \label{Tbd12}
\E\{(t_{12}^{kl})^2\} & \leq \frac{2\{24c(\g)^2 + 69c(\g) + 1\} }{d}\Vert Q_k\Vert^2\Vert Q_l\Vert^2,\\ \label{Tbd21}
\E\{(t_{21}^{kl})^2\} & \leq \frac{2\{24c(\g)^2 + 69c(\g) + 1\} }{d}\Vert Q_k\Vert^2\Vert Q_l\Vert^2,\\ \label{Tbd22}
\E\{(t_{22}^{kl})^2\} & \leq \frac{c(\g)+4}{d}\Vert Q_k\Vert^2 \Vert Q_l\Vert^2.
\end{align}
\end{lemma}

\begin{proof}
First note that \eqref{subgauss} implies
$\E(|\zeta_i|^m) \leq c(\g)$,  $m = 1,\ldots,8$.
This moment
bound for the $\zeta_i$ will be used repeatedly below.  Each bound in the lemma follows from a direct calculations.  However,
\eqref{Tbd11} is substantially more involved than the others; we save
this bound until the end.  First, we derive \eqref{Tbd22}.
We have
\begin{align*}
\E\{(t_{22}^{kl})^2\} & = \frac{1}{d^2}\sum_{i = 1}^d \E\left[\{(\zeta_i^4
                        - \mu_i^{(4)}) - 2(\zeta_i^2 -
                        1)\}^2\right](q_{ii}^{(k)}q_{ii}^{(l)})^2 \\
& \leq \frac{1}{d^2}\sum_{i = 1}^d
  \E\left\{(\zeta_i^4 -
  2\zeta_i^2)^2\right\}(q_{ii}^{(k)}q_{ii}^{(l)})^2\\
& = \frac{1}{d^2}\sum_{i = 1}^d
  (\mu_i^{(8)} - 4\mu_i^{(6)} + 4)(q_{ii}^{(k)}q_{ii}^{(l)})^2\\ 
& \leq \frac{c(\g)+4}{d}\Vert Q_k \Vert^2 \Vert Q_l\Vert^2.
\end{align*}
Next, we prove \eqref{Tbd12}.  Let $\bm_3 = (\mu_1^{(3)},\ldots,\mu_d^{(3)})^{\T} \in \R^d$ and, for
matrices $A,B$, let $A \circ B$ denote their Hadamard product.  Then
\begin{align*}
\E\{(t_{12}^{kl})^2\} & \leq \frac{8}{d^2}\sum_{1 \leq i \neq j \leq
                        d} \sum_{1 \leq m \neq n \leq d}
                        \E\left\{\zeta_j\zeta_n(\mu_i^{(3)} - \zeta_i +
                        \zeta_i^3)(\mu_m^{(3)} - \zeta_m +
                        \zeta_m^3)\right\}q_{ij}^{(k)}q_{ij}^{(l)}q_{mn}^{(k)}q_{mn}^{(l)}
  \\
& \qquad + \frac{2}{d^2}\sum_{i = 1}^d\E\left[\{(\zeta_i^4 -
  \mu_i^{(4)}) - 2(\zeta_i^2 -
  1)\}^2\right](q_{ii}^{(k)}q_{ii}^{(l)})^2 \\
& \leq \frac{8}{d^2} \sum_{1\leq i \neq j \leq d} \E\left\{\zeta_j (\mu_j^{(3)} - \zeta_j +
                        \zeta_j^3)\right\}\E\left\{\zeta_i (\mu_i^{(3)} - \zeta_i +
                        \zeta_i^3)\right\}(q_{ij}^{(k)}q_{ij}^{(l)})^2\\
& \qquad +
  \frac{8}{d^2} \sum_{1 \leq i \neq j\leq d} 
  \E\left\{\zeta_j^2(\mu_i^{(3)} - \zeta_i +
                        \zeta_i^3)^2\right\}(q_{ij}^{(k)}q_{ij}^{(l)})^2\\
& \qquad +
  \frac{8}{d^2} \sum_{1 \leq i \neq j \neq m \leq d} 
  \E\left\{\zeta_j^2(\mu_i^{(3)} - \zeta_i +
                        \zeta_i^3)(\mu_m^{(3)} - \zeta_m +
                        \zeta_m^3)\right\}q_{ij}^{(k)}q_{ij}^{(l)}q_{mj}^{(k)}q_{mj}^{(l)}\\
& \qquad + \frac{2}{d^2}\sum_{i = 1}^d
  \E\left\{(\zeta_i^4 -
  2\zeta_i^2)^2\right\}(q_{ii}^{(k)}q_{ii}^{(l)})^2\\
& =\frac{32}{d^2} \sum_{1 \leq i \neq j \neq m \leq d}
  \mu_i^{(3)}
  \mu_j^{(3)}q_{im}^{(k)}q_{mi}^{(l)}q_{jm}^{(k)}q_{mj}^{(l)}\\ & \qquad
                                                                +
  \frac{8}{d^2}\sum_{1 \leq i \neq j \leq d} \{\mu_i^{(6)}  +
                                                                \mu_i^{(4)}\mu_j^{(4)}
                                                                -
                                                                3\mu_i^{(4)}
                                                                -
                                                                \mu_j^{(4)}
                                                                +
                                                                3(\mu_i^{(3)})^2
                                                                + 2\}(q_{ij}^{(k)}q_{ij}^{(l)})^2 \\
& \qquad +  \frac{2}{d^2}\sum_{i = 1}^d
  (\mu_i^{(8)} - 4\mu_i^{(6)} + 4)(q_{ii}^{(k)}q_{ii}^{(l)})^2 \\ 
& = \frac{32}{d^2} \sum_{i,j,m=1}^d
  \mu_i^{(3)}
  \mu_j^{(3)}q_{im}^{(k)}q_{mi}^{(l)}q_{jm}^{(k)}q_{mj}^{(l)} \\
& \qquad - \frac{64}{d^2}\sum_{i,j=1}^d   \mu_i^{(3)}
  \mu_j^{(3)}q_{ii}^{(k)}q_{ii}^{(l)}q_{ji}^{(k)}q_{ij}^{(l)}
\\ & \qquad + \frac{8}{d^2}\sum_{i,j=1}^d \{\mu_i^{(6)}  +
                                                                \mu_i^{(4)}\mu_j^{(4)}
                                                                -
                                                                3\mu_i^{(4)}
                                                                -
                                                                \mu_j^{(4)}
                                                                -
                                                                (\mu_i^{(3)})^2
                                                                + 2\}(q_{ij}^{(k)}q_{ij}^{(l)})^2 \\
& \qquad +  \frac{2}{d^2}\sum_{i = 1}^d
  \{\mu_i^{(8)} - 8\mu_i^{(6)} - 4(\mu_i^{(4)})^2 + 16\mu_i^{(4)} +
  20(\mu_i^{(3)})^2 -  4\}(q_{ii}^{(k)}q_{ii}^{(l)})^2\\ 
& \leq \frac{32}{d^2} \bm_3^{\T}(Q_k\circ Q_l)^2\bm_3  -
  \frac{64}{d^2}\bm_3^{\T} Q_k^{\mathrm{D}}Q_l^{\mathrm{D}}(Q_k\circ
  Q_l)\bm_3 \\
& \qquad + \frac{8\{c(\g) + c(\g)^2 + 2\}}{d^2}\tr\{(Q_k\circ Q_l)^2\} +
  \frac{2\{17c(\g)+20c(\g)^2\}}{d} \Vert Q_k\Vert^2\Vert Q_l\Vert^2\\
& \leq \frac{32c(\g)}{d}\Vert Q_k\circ Q_l\Vert^2 + \frac{64c(\g)}{d} \Vert
  Q_k\Vert \Vert Q_l\Vert \Vert Q_k \circ Q_l\Vert \\
& \qquad + \frac{8\{c(\g) + c(\g)^2 + 2\}}{d}\Vert Q_k\circ Q_l\Vert^2 +
  \frac{2\{17c(\g)+20c(\g)^2\}}{d} \Vert Q_k\Vert^2\Vert Q_l\Vert^2\\
& \leq \frac{2\{24c(\g)^2 + 69c(\g) + 1\} }{d}\Vert Q_k\Vert^2\Vert Q_l\Vert^2,
\end{align*}
where we have used the fact that $\Vert Q_k \circ Q_l \Vert \leq \Vert
Q_k\Vert \Vert Q_l\Vert$ (Theorem 3.1 of \citep{horn1990analog}).
Thus, we have proved \eqref{Tbd12}; \eqref{Tbd21} follows immediately
by symmetry.   Finally, we bound the second moment of
$t_{11}^{kl}$.  Observe that
\begin{equation}\label{t11terms}
\E\{(t_{11}^{kl})^2\}  \leq D_1 + D_2 + D_3 + D_4,
\end{equation}
where
\begin{align*}
D_1 & = \frac{64}{d^2} \sum_{m,n=1}^d \sum_{\substack{1 \leq i \neq j
      \leq d \\ 1 \leq u \neq v \leq d}}
  \E\left\{\zeta_i\zeta_j\zeta_u\zeta_v(1 + \zeta_m^2)(1+\zeta_n^2)\right\}q_{im}^{(k)}q_{mj}^{(l)}q_{un}^{(k)}q_{nv}^{(l)},
                       \\
D_2 & =  \frac{16}{d^2}\sum_{1 \leq i \neq j \leq d}
                        \sum_{1 \leq m \neq n\leq d}
                        \E\left\{\zeta_j\zeta_n(\mu_i^{(3)} - 3\zeta_i
                        - \zeta_i^3)(\mu_m^{(3)} - 3\zeta_m
                        - \zeta_m^3) \right\}\\
& \qquad \qquad \cdot (q_{ij}^{(k)}q_{ii}^{(l)} +
  q_{ii}^{(k)}q_{ij}^{(l)}) (q_{mn}^{(k)}q_{mm}^{(l)} +
  q_{mm}^{(k)}q_{mn}^{(l)}), \\
D_3 & = \frac{64}{d^2}\sum_{1 \leq i \neq j \leq d} \sum_{1 \leq m
  \neq n \leq d} \E\left[\{(\zeta_i^2\zeta_j^2 - 1) + (\zeta_j^2 -
  1)\}\{(\zeta_m^2\zeta_n^2 - 1)+(\zeta_n^2 - 1)\}\right] q_{ij}^{(k)}q_{ij}^{(l)}q_{mn}^{(k)}q_{mn}^{(l)},
  \\
D_4 & =  \frac{4}{d^2} \sum_{i=1}^d \E\left[\{(\zeta_i^4 -
  \mu_i^{(4)}) - 2(\zeta_i^2 -
  1)\}^2\right](q_{ii}^{(k)}q_{ii}^{(l)})^2. 
\end{align*}
We bound $D_1,D_2,D_3,D_4$ separately.  First we consider $D_1$, the
most complicated term.  Define the diagonal matrix $M_k =
\diag(\mu_1^{(k)},\ldots,\mu_d^{(k)})$ and observe that
\begin{align*}
D_1 & = \frac{64}{d^2}\sum_{m,n=1}^d \sum_{1 \leq i \neq j \leq d}
      \E\left\{\zeta_i^2\zeta_j^2(1 +
      \zeta_m^2)(1+\zeta_n^2) \right\} (q_{im}^{(k)}q_{mj}^{(l)}q_{in}^{(k)}q_{nj}^{(l)}
      + q_{im}^{(k)}q_{mj}^{(l)}q_{jn}^{(k)}q_{ni}^{(l)}) \\
& = \frac{64}{d^2} \sum_{1 \leq i\neq j\neq m \neq n \leq d} \E\left\{\zeta_i^2\zeta_j^2(1 +
      \zeta_m^2)(1+\zeta_n^2) \right\} (q_{im}^{(k)}q_{mj}^{(l)}q_{in}^{(k)}q_{nj}^{(l)}
      + q_{im}^{(k)}q_{mj}^{(l)}q_{jn}^{(k)}q_{ni}^{(l)}) \\
& \qquad + \frac{64}{d^2} \sum_{1 \leq i \neq j \neq m \leq d} \E\left\{\zeta_i^2\zeta_j^2(1 +
      \zeta_i^2)(1+\zeta_m^2) \right\} (q_{ii}^{(k)}q_{ij}^{(l)}q_{im}^{(k)}q_{mj}^{(l)}
      + q_{ii}^{(k)}q_{ij}^{(l)}q_{jm}^{(k)}q_{mi}^{(l)}) \\
& \qquad + \frac{64}{d^2} \sum_{1 \leq i \neq j \neq m \leq d}\E\left\{\zeta_i^2\zeta_j^2(1 +
      \zeta_j^2)(1+\zeta_m^2) \right\} (q_{ij}^{(k)}q_{jj}^{(l)}q_{im}^{(k)}q_{mj}^{(l)}
      + q_{ij}^{(k)}q_{jj}^{(l)}q_{jm}^{(k)}q_{mi}^{(l)}) \\
& \qquad + \frac{64}{d^2} \sum_{1 \leq i \neq j \neq m \leq d}\E\left\{\zeta_i^2\zeta_j^2(1 +
      \zeta_m^2)(1+\zeta_i^2) \right\} (q_{im}^{(k)}q_{mj}^{(l)}q_{ii}^{(k)}q_{ij}^{(l)}
      + q_{im}^{(k)}q_{mj}^{(l)}q_{ji}^{(k)}q_{ii}^{(l)}) \\
& \qquad + \frac{64}{d^2}\sum_{1 \leq i \neq j \neq m \leq d} \E\left\{\zeta_i^2\zeta_j^2(1 +
      \zeta_m^2)(1+\zeta_j^2) \right\} (q_{im}^{(k)}q_{mj}^{(l)}q_{ij}^{(k)}q_{jj}^{(l)}
      + q_{im}^{(k)}q_{mj}^{(l)}q_{jj}^{(k)}q_{ji}^{(l)}) \\
& \qquad + \frac{64}{d^2}\sum_{1 \leq i \neq j \leq d} \E\left\{\zeta_i^2\zeta_j^2(1+\zeta_i^2)^2 \right\} (q_{ii}^{(k)}q_{ij}^{(l)}q_{ii}^{(k)}q_{ij}^{(l)}
      + q_{ii}^{(k)}q_{ij}^{(l)}q_{ji}^{(k)}q_{ii}^{(l)}) \\
& \qquad + \frac{64}{d^2} \sum_{1 \leq i \neq j \leq d} \E\left\{\zeta_i^2\zeta_j^2(1 +
      \zeta_i^2)(1+\zeta_j^2) \right\} (q_{ii}^{(k)}q_{ij}^{(l)}q_{ij}^{(k)}q_{jj}^{(l)}
      + q_{ii}^{(k)}q_{ij}^{(l)}q_{jj}^{(k)}q_{ji}^{(l)}) \\
& \qquad + \frac{64}{d^2} \sum_{1 \leq i\neq j \leq d} \E\left\{\zeta_i^2\zeta_j^2(1 +
      \zeta_j^2)(1+\zeta_i^2) \right\} (q_{ij}^{(k)}q_{jj}^{(l)}q_{ii}^{(k)}q_{ij}^{(l)}
      + q_{ij}^{(k)}q_{jj}^{(l)}q_{ji}^{(k)}q_{ii}^{(l)}) \\
& \qquad + \frac{64}{d^2} \sum_{1 \leq i \neq j \leq d} \E\left\{\zeta_i^2\zeta_j^2(1+\zeta_j^2)^2 \right\} (q_{ij}^{(k)}q_{jj}^{(l)}q_{ij}^{(k)}q_{jj}^{(l)}
      + q_{ij}^{(k)}q_{jj}^{(l)}q_{jj}^{(k)}q_{ji}^{(l)}) \\
& = \frac{256}{d^2} \sum_{1 \leq i \neq j \neq m \neq n \leq d}  (q_{im}^{(k)}q_{mj}^{(l)}q_{in}^{(k)}q_{nj}^{(l)}
      + q_{im}^{(k)}q_{mj}^{(l)}q_{jn}^{(k)}q_{ni}^{(l)}) \\
& \qquad + \frac{128}{d^2}\sum_{1 \leq i \neq j \neq m \leq d} (1 + \mu_i^{(4)}) (q_{ii}^{(k)}q_{ij}^{(l)}q_{im}^{(k)}q_{mj}^{(l)}
      + q_{ii}^{(k)}q_{ij}^{(l)}q_{jm}^{(k)}q_{mi}^{(l)}) \\
& \qquad + \frac{128}{d^2} \sum_{1 \leq i \neq j \neq m \leq d} (1 + \mu_j^{(4)}) (q_{ij}^{(k)}q_{jj}^{(l)}q_{im}^{(k)}q_{mj}^{(l)}
      + q_{ij}^{(k)}q_{jj}^{(l)}q_{jm}^{(k)}q_{mi}^{(l)}) \\
& \qquad + \frac{128}{d^2} \sum_{1 \leq i \neq j \neq m \leq d} (1 + \mu_i^{(4)}) (q_{im}^{(k)}q_{mj}^{(l)}q_{ii}^{(k)}q_{ij}^{(l)}
      + q_{im}^{(k)}q_{mj}^{(l)}q_{ji}^{(k)}q_{ii}^{(l)}) \\
& \qquad + \frac{128}{d^2}\sum_{1 \leq i \neq j \neq m \leq d} (1 + \mu_j^{(4)})(q_{im}^{(k)}q_{mj}^{(l)}q_{ij}^{(k)}q_{jj}^{(l)}
      + q_{im}^{(k)}q_{mj}^{(l)}q_{jj}^{(k)}q_{ji}^{(l)}) \\
& \qquad + \frac{64}{d^2}\sum_{1 \leq i \neq j \leq d} (\mu_i^{(6)} +
  2\mu_i^{(4)} + 1)(q_{ii}^{(k)}q_{ij}^{(l)}q_{ii}^{(k)}q_{ij}^{(l)}
      + q_{ii}^{(k)}q_{ij}^{(l)}q_{ji}^{(k)}q_{ii}^{(l)}) \\
& \qquad + \frac{64}{d^2} \sum_{1 \leq i \neq j \leq d} (\mu_i^{(4)} +
  1)(\mu_j^{(4)} + 1)(q_{ii}^{(k)}q_{ij}^{(l)}q_{ij}^{(k)}q_{jj}^{(l)}
      + q_{ii}^{(k)}q_{ij}^{(l)}q_{jj}^{(k)}q_{ji}^{(l)}) \\
& \qquad + \frac{64}{d^2} \sum_{1 \leq i\neq j \leq d} (\mu_i^{(4)} +
  1)(\mu_j^{(4)} + 1)(q_{ij}^{(k)}q_{jj}^{(l)}q_{ii}^{(k)}q_{ij}^{(l)}
      + q_{ij}^{(k)}q_{jj}^{(l)}q_{ji}^{(k)}q_{ii}^{(l)}) \\
& \qquad + \frac{64}{d^2} \sum_{1 \leq i \neq j \leq d}  (\mu_j^{(6)} +
  2\mu_j^{(4)} + 1)(q_{ij}^{(k)}q_{jj}^{(l)}q_{ij}^{(k)}q_{jj}^{(l)}
      + q_{ij}^{(k)}q_{jj}^{(l)}q_{jj}^{(k)}q_{ji}^{(l)}) \\
& = \frac{256}{d^2} \sum_{m,n=1}^d \sum_{1 \leq i \neq j \leq d}   (q_{im}^{(k)}q_{mj}^{(l)}q_{in}^{(k)}q_{nj}^{(l)}
      + q_{im}^{(k)}q_{mj}^{(l)}q_{jn}^{(k)}q_{ni}^{(l)})  \\
& \qquad + \frac{128}{d^2}\sum_{1 \leq i \neq j \neq m \leq d} (\mu_i^{(4)}-1) (q_{ii}^{(k)}q_{ij}^{(l)}q_{im}^{(k)}q_{mj}^{(l)}
      + q_{ii}^{(k)}q_{ij}^{(l)}q_{jm}^{(k)}q_{mi}^{(l)}) \\
& \qquad + \frac{128}{d^2} \sum_{1 \leq i \neq j \neq m \leq d} (\mu_j^{(4)}-1) (q_{ij}^{(k)}q_{jj}^{(l)}q_{im}^{(k)}q_{mj}^{(l)}
      + q_{ij}^{(k)}q_{jj}^{(l)}q_{jm}^{(k)}q_{mi}^{(l)}) \\
& \qquad + \frac{128}{d^2} \sum_{1 \leq i \neq j \neq m \leq d} (\mu_i^{(4)}-1) (q_{im}^{(k)}q_{mj}^{(l)}q_{ii}^{(k)}q_{ij}^{(l)}
      + q_{im}^{(k)}q_{mj}^{(l)}q_{ji}^{(k)}q_{ii}^{(l)}) \\
& \qquad + \frac{128}{d^2}\sum_{1 \leq i \neq j \neq m \leq d} (\mu_j^{(4)}-1)(q_{im}^{(k)}q_{mj}^{(l)}q_{ij}^{(k)}q_{jj}^{(l)}
      + q_{im}^{(k)}q_{mj}^{(l)}q_{jj}^{(k)}q_{ji}^{(l)}) \\
& \qquad + \frac{64}{d^2}\sum_{1 \leq i \neq j \leq d} (\mu_i^{(6)} +
  2\mu_i^{(4)} -3)(q_{ii}^{(k)}q_{ij}^{(l)}q_{ii}^{(k)}q_{ij}^{(l)}
      + q_{ii}^{(k)}q_{ij}^{(l)}q_{ji}^{(k)}q_{ii}^{(l)}) \\
& \qquad + \frac{64}{d^2} \sum_{1 \leq i \neq j \leq d}
  (\mu_i^{(4)}\mu_j^{(4)} + \mu_i^{(4)} + \mu_j^{(4)} - 3)(q_{ii}^{(k)}q_{ij}^{(l)}q_{ij}^{(k)}q_{jj}^{(l)}
      + q_{ii}^{(k)}q_{ij}^{(l)}q_{jj}^{(k)}q_{ji}^{(l)}) \\
& \qquad + \frac{64}{d^2} \sum_{1 \leq i\neq j \leq d}  (\mu_i^{(4)}\mu_j^{(4)} + \mu_i^{(4)} + \mu_j^{(4)} - 3)(q_{ij}^{(k)}q_{jj}^{(l)}q_{ii}^{(k)}q_{ij}^{(l)}
      + q_{ij}^{(k)}q_{jj}^{(l)}q_{ji}^{(k)}q_{ii}^{(l)}) \\
& \qquad + \frac{64}{d^2} \sum_{1 \leq i \neq j \leq d}  (\mu_j^{(6)} +
  2\mu_j^{(4)} -3)(q_{ij}^{(k)}q_{jj}^{(l)}q_{ij}^{(k)}q_{jj}^{(l)}
      + q_{ij}^{(k)}q_{jj}^{(l)}q_{jj}^{(k)}q_{ji}^{(l)}) \\
& =  \frac{256}{d^2} \sum_{m,n=1}^d \sum_{1 \leq i \neq j \leq d}   (q_{im}^{(k)}q_{mj}^{(l)}q_{in}^{(k)}q_{nj}^{(l)}
      + q_{im}^{(k)}q_{mj}^{(l)}q_{jn}^{(k)}q_{ni}^{(l)}) \\
& \qquad + \frac{256}{d^2}\sum_{1\leq i \neq j \neq m \leq d}^d
  (\mu_i^{(4)} - 1)(q_{ii}^{(k)}q_{ij}^{(l)}q_{im}^{(k)}q_{mj}^{(l)} +
  q_{ii}^{(k)}q_{ij}^{(l)}q_{jm}^{(k)}q_{mi}^{(l)} \\ & \qquad \qquad +
  q_{ji}^{(k)}q_{ii}^{(l)}q_{jm}^{(k)}q_{mi}^{(l)} +
                                                        q_{ji}^{(k)}q_{ii}^{(l)}q_{im}^{(k)}q_{mj}^{(l)})
  \\
& \qquad + \frac{64}{d^2}\sum_{1 \leq i \neq j \leq d} (\mu_i^{(6)} +
  2\mu_i^{(4)} - 3)(q_{ii}^{(k)}q_{ij}^{(l)}q_{ii}^{(k)}q_{ij}^{(l)} +
  q_{ii}^{(k)}q_{ij}^{(l)}q_{ji}^{(k)}q_{ii}^{(l)} \\
& \qquad \qquad +
  q_{ji}^{(k)}q_{ii}^{(l)}q_{ji}^{(k)}q_{ii}^{(l)} +
  q_{ji}^{(k)}q_{ii}^{(l)}q_{ii}^{(k)}q_{ij}^{(l)}) \\
& \qquad + \frac{64}{d^2} \sum_{1 \leq i \neq j \leq d}
  (\mu_i^{(4)}\mu_j^{(4)} + \mu_i^{(4)} + \mu_j^{(4)} -
  3)(q_{ii}^{(k)}q_{ij}^{(l)}q_{jj}^{(k)}q_{ji}^{(l)}
  +2q_{ii}^{(k)}q_{ij}^{(l)}q_{ij}^{(k)}q_{jj}^{(l)} \\
& \qquad \qquad + q_{ij}^{(k)}q_{jj}^{(l)}q_{ji}^{(k)}q_{ii}^{(l)}) \\
& = \frac{256}{d^2} \sum_{i,j,m,n=1}^d (q_{im}^{(k)}q_{mj}^{(l)}q_{in}^{(k)}q_{nj}^{(l)}
      + q_{im}^{(k)}q_{mj}^{(l)}q_{jn}^{(k)}q_{ni}^{(l)})  - \frac{512}{d^2} \sum_{i,j,m=1}^d
  q_{ij}^{(k)}q_{ji}^{(l)}q_{im}^{(k)}q_{mi}^{(l)} \\
& \qquad + \frac{256}{d^2}\sum_{i,j,m=1}^d
  (\mu_i^{(4)} - 1)(q_{ii}^{(k)}q_{ij}^{(l)}q_{im}^{(k)}q_{mj}^{(l)} +
  q_{ii}^{(k)}q_{ij}^{(l)}q_{jm}^{(k)}q_{mi}^{(l)} \\ & \qquad \qquad +
  q_{ji}^{(k)}q_{ii}^{(l)}q_{jm}^{(k)}q_{mi}^{(l)} +
                                                        q_{ji}^{(k)}q_{ii}^{(l)}q_{im}^{(k)}q_{mj}^{(l)})
  \\
& \qquad - \frac{256}{d^2}\sum_{i,j=1}^d (\mu_i^{(4)} -
  1) \left\{(q_{ii}^{(k)}q_{ij}^{(l)})^2 +
  (q_{ji}^{(k)}q_{ii}^{(l)})^2\right\}
  \\
& \qquad - \frac{256}{d^2}\sum_{i,j=1}^d (\mu_i^{(4)} - 1)(q_{ii}^{(k)}q_{ij}^{(l)}q_{ij}^{(k)}q_{jj}^{(l)} +
  q_{ii}^{(k)}q_{ij}^{(l)}q_{jj}^{(k)}q_{ji}^{(l)} \\ & \qquad \qquad +
  q_{ji}^{(k)}q_{ii}^{(l)}q_{jj}^{(k)}q_{ji}^{(l)} +
                                                        q_{ji}^{(k)}q_{ii}^{(l)}q_{ij}^{(k)}q_{jj}^{(l)})
  \\
& \qquad -\frac{1536}{d^2}\sum_{i,j=1}^d (\mu_i^{(4)} - 1)q_{ii}^{(k)}q_{ii}^{(l)}q_{ij}^{(k)}q_{ji}^{(l)} 
  \\
& \qquad + \frac{2048}{d^2}\sum_{i = 1}^d (\mu_i^{(4)} -
  1)(q_{ii}^{(k)}q_{ii}^{(l)})^2 \\
& \qquad + \frac{64}{d^2}\sum_{i,j=1}^d (\mu_i^{(6)} +
  2\mu_i^{(4)} - 3)(q_{ii}^{(k)}q_{ij}^{(l)}q_{ii}^{(k)}q_{ij}^{(l)} +
  q_{ii}^{(k)}q_{ij}^{(l)}q_{ji}^{(k)}q_{ii}^{(l)} \\
& \qquad \qquad +
  q_{ji}^{(k)}q_{ii}^{(l)}q_{ji}^{(k)}q_{ii}^{(l)} +
  q_{ji}^{(k)}q_{ii}^{(l)}q_{ii}^{(k)}q_{ij}^{(l)}) \\
& \qquad + \frac{64}{d^2} \sum_{i,j=1}^d
  (\mu_i^{(4)}\mu_j^{(4)} + \mu_i^{(4)} + \mu_j^{(4)} -
  3)(q_{ii}^{(k)}q_{ij}^{(l)}q_{jj}^{(k)}q_{ji}^{(l)}
  +2q_{ii}^{(k)}q_{ij}^{(l)}q_{ij}^{(k)}q_{jj}^{(l)} \\
& \qquad \qquad + q_{ij}^{(k)}q_{jj}^{(l)}q_{ji}^{(k)}q_{ii}^{(l)}) \\
& \qquad - \frac{256}{d^2}\sum_{i = 1}^d \left\{\mu_i^{(6)} +
  (\mu_i^{(4)})^2 + 4\mu_i^{(4)} -
  6\right\}(q_{ii}^{(k)}q_{ii}^{(l)})^2 \\ 
&= \frac{256}{d^2} \sum_{i,j,m,n=1}^d (q_{im}^{(k)}q_{mj}^{(l)}q_{in}^{(k)}q_{nj}^{(l)}
      + q_{im}^{(k)}q_{mj}^{(l)}q_{jn}^{(k)}q_{ni}^{(l)})  - \frac{512}{d^2} \sum_{i,j,m=1}^d
  q_{ij}^{(k)}q_{ji}^{(l)}q_{im}^{(k)}q_{mi}^{(l)} \\
& \qquad + \frac{256}{d^2}\sum_{i,j,m=1}^d
  (\mu_i^{(4)} - 1)(q_{ii}^{(k)}q_{ij}^{(l)}q_{im}^{(k)}q_{mj}^{(l)} +
  q_{ii}^{(k)}q_{ij}^{(l)}q_{jm}^{(k)}q_{mi}^{(l)} \\ & \qquad \qquad +
  q_{ji}^{(k)}q_{ii}^{(l)}q_{jm}^{(k)}q_{mi}^{(l)} +
                                                        q_{ji}^{(k)}q_{ii}^{(l)}q_{im}^{(k)}q_{mj}^{(l)})
  \\
& \qquad + \frac{64}{d^2}\sum_{i,j=1}^d (\mu_i^{(6)}- 2\mu_i^{(4)} +
  1) \left\{(q_{ii}^{(k)}q_{ij}^{(l)})^2 +
  (q_{ji}^{(k)}q_{ii}^{(l)})^2\right\}
  \\
& \qquad + \frac{128}{d^2}\sum_{i,j=1}^d (\mu_i^{(6)} -
  10\mu_i^{(4)} +9)
  q_{ii}^{(k)}q_{ii}^{(l)}q_{ij}^{(k)} q_{ji}^{(l)} 
  \\
& \qquad + \frac{64}{d^2} \sum_{i,j=1}^d
  (\mu_i^{(4)}\mu_j^{(4)} - 3\mu_i^{(4)} + \mu_j^{(4)} +1)(q_{ii}^{(k)}q_{ij}^{(l)}q_{ij}^{(k)}q_{jj}^{(l)} +
  q_{ii}^{(k)}q_{ij}^{(l)}q_{jj}^{(k)}q_{ji}^{(l)} \\ & \qquad \qquad +
  q_{ji}^{(k)}q_{ii}^{(l)}q_{jj}^{(k)}q_{ji}^{(l)} +
                                                        q_{ji}^{(k)}q_{ii}^{(l)}q_{ij}^{(k)}q_{jj}^{(l)})\\
& \qquad - \frac{256}{d^2}\sum_{i = 1}^d \left\{\mu_i^{(6)} +
  (\mu_i^{(4)})^2 - 4\mu_i^{(4)} +
  2\right\}(q_{ii}^{(k)}q_{ii}^{(l)})^2 \\
& = \frac{256}{d^2} \left\{\tr(Q_kQ_lQ_lQ_k) +
  \tr(Q_kQ_lQ_kQ_l)\right\} - \frac{512}{d^2} \tr\left\{(Q_kQ_l)\circ
  (Q_kQ_l)\right\} \\
& \qquad + \frac{256}{d^2}\left[\tr\left\{(M_4 -
  I)Q_k^{\mathrm{D}}Q_lQ_lQ_k\right\} + \tr\left\{(M_4 -
  I)Q_k^{\mathrm{D}}Q_lQ_kQ_l\right\}\right] \\
& \qquad  + \frac{256}{d^2}\left[\tr\left\{(M_4 -
  I)Q_l^{\mathrm{D}}Q_kQ_kQ_l\right\} + \tr\left\{(M_4 -
  I)Q_l^{\mathrm{D}}Q_kQ_lQ_k\right\}\right] \\ 
& \qquad + \frac{64}{d^2}\left[ \tr\left\{(M_6 - 2M_4 + 1)
  Q_k^{\mathrm{D}}Q_lQ_lQ_l^{\mathrm{D}}\right\} + \tr\left\{(M_6 - 2M_4 + 1)
  Q_l^{\mathrm{D}}Q_kQ_kQ_l^{\mathrm{D}} \right\}\right] \\ 
& \qquad + \frac{128}{d^2} \tr\left\{(M_6 - 10M_4 +
  9)Q_k^{\mathrm{D}}Q_kQ_l^{\mathrm{D}}Q_l\right\} \\ 
& \qquad + \frac{64}{d^2}\tr\left\{M_4Q_k^{\mathrm{D}}Q_l M_4Q_l^{\mathrm{D}} Q_k
  -3M_4Q_k^{\mathrm{D}}Q_lQ_l^{\mathrm{D}}Q_k +
  Q_k^{\mathrm{D}}Q_l M_4Q_l^{\mathrm{D}} Q_k
  +Q_k^{\mathrm{D}}Q_lQ_l^{\mathrm{D}}Q_k  \right\}  \\
& \qquad +  
\frac{64}{d^2}\tr\left\{M_4Q_k^{\mathrm{D}}Q_lM_4Q_k^{\mathrm{D}}Q_l
  -3M_4Q_k^{\mathrm{D}}Q_lQ_k^{\mathrm{D}}Q_l+
 Q_k^{\mathrm{D}}Q_lM_4Q_k^{\mathrm{D}}Q_l+
  Q_k^{\mathrm{D}}Q_lQ_k^{\mathrm{D}}Q_l\right\} \\
& \qquad + \frac{64}{d^2}\tr\left\{M_4Q_l^{\mathrm{D}}Q_kM_4Q_k^{\mathrm{D}}Q_l
  -3M_4Q_l^{\mathrm{D}}Q_kQ_k^{\mathrm{D}}Q_l+
 Q_l^{\mathrm{D}}Q_kM_4Q_k^{\mathrm{D}}Q_l+
  Q_l^{\mathrm{D}}Q_kQ_k^{\mathrm{D}}Q_l\right\} \\ 
& \qquad + \frac{64}{d^2}\tr\left\{M_4Q_l^{\mathrm{D}}Q_kM_4Q_l^{\mathrm{D}}Q_k
  -3M_4Q_l^{\mathrm{D}}Q_kQ_l^{\mathrm{D}}Q_k+
 Q_l^{\mathrm{D}}Q_kM_4Q_l^{\mathrm{D}}Q_k+
  Q_l^{\mathrm{D}}Q_kQ_l^{\mathrm{D}}Q_k\right\} \\ 
& \qquad - \frac{256}{d^2}\tr\left\{(M_6 + M_4^2 - 4M_4 +
  2)(Q_k^{\mathrm{D}}Q_l^{\mathrm{D}})^2\right\} \\
& \leq \frac{512}{d}\Vert Q_k\Vert^2\Vert Q_l\Vert^2 +
  \frac{512}{d} \Vert Q_k\Vert^2\Vert Q_l\Vert^2 +
  \frac{1024\{c(\g)+1\}}{d}\Vert Q_k\Vert^2\Vert Q_l\Vert^2 \\
& \qquad + \frac{128\{3c(\g)+1\}}{d}\Vert Q_k\Vert^2\Vert Q_l\Vert^2 +
  \frac{128\{11c(\g) + 9\}}{d}\Vert Q_k\Vert^2\Vert Q_l\Vert^2 \\
& \qquad  +
  \frac{256\{c(\g)^2 +4c(\g) + 1\}}{d}\Vert Q_k\Vert^2\Vert Q_l\Vert^2 +
  \frac{256\{c(\g)^2 + 5c(\g) + 2\}}{d}\Vert Q_k\Vert^2\Vert Q_l\Vert^2 \\ \numberthis\label{I1}
& = \frac{512\{c(\g)^2 + 10c(\g) + 8\}}{d}\Vert Q_k\Vert^2\Vert Q_l\Vert^2.
\end{align*}
To bound $D_2$, 
\begin{align*}
D_2 & = \frac{16}{d^2} \sum_{1 \leq i \neq j \leq d}
  \E\left\{\zeta_j^2(\mu_i^{(3)} - 3\zeta_i -
  \zeta_i^3)^2\right\}(q_{ij}^{(k)}q_{ii}^{(l)} +
  q_{ii}^{(k)}q_{ij}^{(l)})^2 \\
& \qquad + \frac{16}{d^2} \sum_{1 \leq i \neq j \leq d}
  \E\left\{\zeta_j(\mu_j^{(3)} - 3\zeta_j -
  \zeta_j^3)\right\}\E\left\{\zeta_i(\mu_i^{(3)} - 3\zeta_i -
  \zeta_i^3)\right\} \\
& \qquad \qquad \cdot (q_{ij}^{(k)}q_{ii}^{(l)} +
  q_{ii}^{(k)}q_{ij}^{(l)})(q_{ji}^{(k)}q_{jj}^{(l)} +
  q_{jj}^{(k)}q_{ji}^{(l)}) \\
& = \frac{16}{d^2} \sum_{1 \leq i \neq j \leq d}
  \left\{\mu_i^{(6)} + 6\mu_i^{(4)} - (\mu_i^{(3)})^2 + 9\right\}(q_{ij}^{(k)}q_{ii}^{(l)} +
  q_{ii}^{(k)}q_{ij}^{(l)})^2  \\
& \qquad + \frac{16}{d^2} \sum_{1 \leq i \neq j \leq d} (\mu_i^{(4)}
  +3)(\mu_j^{(4)} + 3) (q_{ij}^{(k)}q_{ii}^{(l)} +
  q_{ii}^{(k)}q_{ij}^{(l)})(q_{ji}^{(k)}q_{jj}^{(l)} +
  q_{jj}^{(k)}q_{ji}^{(l)}) \\
& = \frac{16}{d^2} \sum_{i,j=1}^d
  \left\{\mu_i^{(6)} + 6\mu_i^{(4)} - (\mu_i^{(3)})^2 + 9\right\}(q_{ij}^{(k)}q_{ii}^{(l)} +
  q_{ii}^{(k)}q_{ij}^{(l)})^2  \\
& \qquad + \frac{16}{d^2} \sum_{i,j=1}^d (\mu_i^{(4)}
  +3)(\mu_j^{(4)} + 3) (q_{ij}^{(k)}q_{ii}^{(l)} +
  q_{ii}^{(k)}q_{ij}^{(l)})(q_{ji}^{(k)}q_{jj}^{(l)} +
  q_{jj}^{(k)}q_{ji}^{(l)}) \\
& \qquad - \frac{16}{d^2}\sum_{i = 1}^d \left\{\mu_i^{(6)} -
  (\mu_i^{(4)})^2 - (\mu_i^{(3)})^2 \right\}
  (q_{ii}^{(k)}q_{ii}^{(l)})^2 \\
& \leq \frac{16\{7c(\g) + 9\}}{d^2} \sum_{i,j=1}^d (q_{ij}^{(k)}q_{ii}^{(l)})^2 +
  \frac{32\{7c(\g) + 9\}}{d^2} \sum_{i,j=1}^d q_{ij}^{(k)}q_{ii}^{(l)}  q_{ii}^{(k)}q_{ij}^{(l)} + \frac{16\{7c(\g) + 9\}}{d^2} \sum_{i,j=1}^d ( q_{ii}^{(k)}q_{ij}^{(l)})^2 \\
& \qquad + \frac{16}{d^2} \sum_{i,j=1}^d  (\mu_i^{(4)}
  +3)(\mu_j^{(4)} +
  3)q_{ji}^{(k)}q_{ii}^{(l)}q_{ij}^{(k)}q_{jj}^{(l)} 
  + \frac{16}{d^2} \sum_{i,j=1}^d  (\mu_i^{(4)}
  +3)(\mu_j^{(4)} +
  3)q_{ji}^{(k)}q_{ii}^{(l)}q_{jj}^{(k)}q_{ji}^{(l)} \\
& \qquad + \frac{16}{d^2} \sum_{i,j=1}^d  (\mu_i^{(4)}
  +3)(\mu_j^{(4)} +
  3)q_{ii}^{(k)}q_{ij}^{(l)}q_{ij}^{(k)}q_{jj}^{(l)} + \frac{16}{d^2} \sum_{i,j=1}^d  (\mu_i^{(4)}
  +3)(\mu_j^{(4)} +
  3) q_{ii}^{(k)}q_{ij}^{(l)}q_{jj}^{(k)}q_{ji}^{(l)} \\
& \qquad +  \frac{32c(\g)^2}{d}\Vert Q_k \Vert^2 \Vert Q_l\Vert^2 \\
& = \frac{16\{7c(\g)+9\}}{d^2} \left[\tr
  \{Q_l^{\mathrm{D}}Q_kQ_kQ_l^{\mathrm{D}}\} +
  2\tr(Q_k^{\mathrm{D}}Q_kQ_lQ_l^{\mathrm{D}}) + \tr
  \{Q_k^{\mathrm{D}}Q_lQ_lQ_k^{\mathrm{D}}\} \right] \\
& \qquad + \frac{16}{d^2}\left[\tr\left\{(M_4 + 3I)Q_l^{\mathrm{D}}Q_k (M_4 + 3I)
 Q_l^{\mathrm{D}} Q_k\right\} + \tr\left\{(M_4 + 3I)
  Q_l^{\mathrm{D}} Q_k   (M_4+3I)Q_k^{\mathrm{D}}Q_l\right\}\right] \\
& \qquad + \frac{16}{d^2}\left[\tr\left\{(M_4 +
  3I)Q_k^{\mathrm{D}}Q_l(M_4 + 3I) Q_l^{\mathrm{D}}Q_k\right\} +
  \tr\left\{(M_4 + 3I)Q_k^{\mathrm{D}}Q_l(M_4 +
  3I)Q_k^{\mathrm{D}}Q_l\right\}\right] \\
& \qquad +  \frac{32c(\g)^2}{d}\Vert Q_k \Vert^2 \Vert Q_l\Vert^2 \\
& \leq \frac{64\{7c(\g)+9\}}{d} \Vert Q_k\Vert^2\Vert Q_l\Vert^2 +
  \frac{64\{c(\g) + 3\}^2}{d}\Vert Q_k\Vert^2\Vert Q_l\Vert^2 +
  \frac{32c(\g)^2}{d}\Vert Q_k \Vert^2 \Vert Q_l\Vert^2 \\ \numberthis\label{I2}
& = \frac{32\{3c(\g)^2 + 26c(\g) + 36\}}{d}\Vert Q_k \Vert^2 \Vert Q_l\Vert^2.
\end{align*}
Next, we bound $D_3$: 
\begin{align*}
D_3 & = \frac{64}{d^2}\sum_{1 \leq i \neq j \neq m \leq
      d}\E\left[\{(\zeta_i^2\zeta_j^2 - 1) + (\zeta_j^2 -
      1)\}\{(\zeta_i^2\zeta_m^2 - 1) + (\zeta_m^2 -
      1)\}\right]q_{ij}^{(k)}q_{ij}^{(l)}q_{im}^{(k)}q_{im}^{(l)} \\
& \qquad + \frac{64}{d^2} \sum_{1 \leq i \neq j \neq m \leq d} \E\left[\{(\zeta_i^2\zeta_j^2 - 1) + (\zeta_j^2 -
      1)\}\{(\zeta_j^2\zeta_m^2 - 1) + (\zeta_m^2 -
  1)\}\right]q_{ij}^{(k)}q_{ij}^{(l)}q_{jm}^{(k)}q_{jm}^{(l)} \\
& \qquad + \frac{64}{d^2} \sum_{1 \leq i \neq j \neq m \leq d} \E\left[\{(\zeta_i^2\zeta_j^2 - 1) + (\zeta_j^2 -
      1)\}\{(\zeta_m^2\zeta_i^2 - 1) + (\zeta_i^2 -
  1)\}\right]q_{ij}^{(k)}q_{ij}^{(l)}q_{mi}^{(k)}q_{mi}^{(l)} \\
& \qquad + \frac{64}{d^2} \sum_{1 \leq i \neq j \neq m \leq d} \E\left[\{(\zeta_i^2\zeta_j^2 - 1) + (\zeta_j^2 -
      1)\}\{(\zeta_m^2\zeta_j^2 - 1) + (\zeta_j^2 -
  1)\}\right]q_{ij}^{(k)}q_{ij}^{(l)}q_{mj}^{(k)}q_{mj}^{(l)} \\
& \qquad + \frac{64}{d^2} \sum_{1 \leq i \neq j \leq d} \E\left[\{(\zeta_i^2\zeta_j^2 - 1) + (\zeta_j^2 -
      1)\}\{(\zeta_i^2\zeta_j^2 - 1) + (\zeta_j^2 -
  1)\}\right](q_{ij}^{(k)}q_{ij}^{(l)})^2 \\
& \qquad + \frac{64}{d^2} \sum_{1 \leq i \neq j \leq d} \E\left[\{(\zeta_i^2\zeta_j^2 - 1) + (\zeta_j^2 -
      1)\}\{(\zeta_j^2\zeta_i^2 - 1) + (\zeta_i^2 -
  1)\}\right](q_{ij}^{(k)}q_{ij}^{(l)})^2 \\
& = \frac{64}{d^2} \sum_{1 \leq i \neq j \neq m \leq d}
  (\mu_i^{(4)} - 1) q_{ij}^{(k)}q_{ij}^{(l)}q_{im}^{(k)}q_{im}^{(l)}
  \\
& \qquad + \frac{64}{d^2} \sum_{1 \leq i \neq j \neq m \leq d} (2\mu_j^{(4)} -
  2) q_{ij}^{(k)}q_{ij}^{(l)}q_{jm}^{(k)}q_{jm}^{(l)} \\
&  \qquad + \frac{64}{d^2} \sum_{1 \leq i \neq j \neq m \leq d}
  (\mu_i^{(4)} - 1) q_{ij}^{(k)}q_{ij}^{(l)}q_{mi}^{(k)}q_{mi}^{(l)}
  \\
& \qquad + \frac{64}{d^2} \sum_{1 \leq i \neq j \neq m \leq d}
  (4\mu_j^{(4)} - 4) q_{ij}^{(k)}q_{ij}^{(l)}q_{mj}^{(k)}q_{mj}^{(l)}
  \\
& \qquad + \frac{64}{d^2} \sum_{1 \leq i \neq j \leq d} \{(\mu_i^{(4)}
  + 3)\mu_j^{(4)} -4\}(q_{ij}^{(k)}q_{ij}^{(l)})^2 \\
& \qquad + \frac{64}{d^2} \sum_{1 \leq i \neq j \leq d}
  (\mu_i^{(4)}\mu_j^{(4)} + \mu_i^{(4)} + \mu_j^{(4)} - 3)
  (q_{ij}^{(k)}q_{ij}^{(l)})^2 \\
& = \frac{512}{d^2} \sum_{1 \leq i \neq j \neq m \leq d} (\mu_i^{(4)}
  - 1) q_{ij}^{(k)}q_{ij}^{(l)}q_{im}^{(k)}q_{im}^{(l)} \\
& \qquad +
  \frac{64}{d^2}\sum_{1 \leq i \neq j \leq d} (2\mu_i^{(4)}\mu_j^{(4)}
  + 4\mu_j^{(4)} + \mu_i^{(4)} - 7)(q_{ij}^{(k)}q_{ij}^{(l)})^2 \\
& = \frac{512}{d^2} \sum_{i,j,m=1}^d (\mu_i^{(4)} -
  1)q_{ij}^{(k)}q_{ij}^{(l)}q_{im}^{(k)}q_{im}^{(l)} \\
& \qquad - \frac{512}{d^2} \sum_{i,j=1}^d (\mu_i^{(4)} -
  1)q_{ij}^{(k)}q_{ij}^{(l)}q_{ii}^{(k)}q_{ii}^{(l)} \\
& \qquad - \frac{512}{d^2} \sum_{i,j=1}^d (\mu_i^{(4)} -
  1)(q_{ij}^{(k)}q_{ij}^{(l)})^2 \\
& \qquad - \frac{512}{d^2} \sum_{i,j=1}^d  (\mu_i^{(4)} -
  1)q_{ii}^{(k)}q_{ii}^{(l)}q_{ij}^{(k)}q_{ij}^{(l)} \\
& \qquad + \frac{1024}{d^2}\sum_{i = 1}^d (\mu_i^{(4)} -
  1)(q_{ii}^{(k)}q_{ii}^{(l)})^2 \\
& \qquad + \frac{64}{d^2} \sum_{i,j=1}^d  (2\mu_i^{(4)}\mu_j^{(4)}
  + 4\mu_j^{(4)} + \mu_i^{(4)} - 7)(q_{ij}^{(k)}q_{ij}^{(l)})^2 \\
& \qquad - \frac{64}{d^2}\sum_{i = 1}^d \left\{2(\mu_i^{(4)})^2 +
  5\mu_i^{(4)} - 7\right\}(q_{ii}^{(k)}q_{ii}^{(l)})^2 \\
& = \frac{512}{d^2}\sum_{i,j,m=1}^d (\mu_i^{(4)} -
  1)q_{ij}^{(k)}q_{ij}^{(l)}q_{im}^{(k)}q_{im}^{(l)} -
  \frac{1024}{d^2} \sum_{i,j = 1}^d (\mu_i^{(4)} -
  1)q_{ii}^{(k)}q_{ii}^{(l)}q_{ij}^{(k)}q_{ij}^{(l)} \\
& \qquad + \frac{64}{d^2}\sum_{i,j=1}^d (2\mu_i^{(4)}\mu_j^{(4)} -
  4\mu_j^{(4)} + \mu_i^{(4)} +1)
  (q_{ij}^{(k)}q_{ij}^{(l)})^2 \\
& \qquad - \frac{64}{d^2} \sum_{i = 1}^d \left\{2(\mu_i^{(4)})^2 -
  11\mu_i^{(4)} +9\right\}(q_{ii}^{(k)}q_{ii}^{(l)})^2 \\
& = \frac{512}{d^2} \tr\left\{(M_4 - I)(Q_kQ_l)\circ(Q_kQ_l)\right\} -
  \frac{1024}{d^2} \tr\left\{(M_4 -
  I)Q_k^{\mathrm{D}}Q_l^{\mathrm{D}}Q_kQ_l\right\} \\
& \qquad + \frac{128}{d^2} \tr\left\{M_4(Q_k\circ Q_l)M_4  (Q_k\circ
  Q_l) - 2(Q_k\circ Q_l)M_4(Q_k\circ Q_l)\right\} \\
& \qquad + \frac{64}{d^2}\tr\left\{M_4(Q_k \circ Q_l)^2 +(Q_k \circ
  Q_l)^2 \right\} - \frac{64}{d^2} \tr\left\{(2M_4^2 - 11M_4 +
  9)(Q_k^{\mathrm{D}}Q_l^{\mathrm{D}})^2\right\} \\
& \leq \frac{128\{2c(\g)^2+20c(\g) + 17\}}{d}\Vert Q_k\Vert^2 \Vert Q_l\Vert^2 \numberthis\label{I3}.
\end{align*}
It remains to bound $D_4$, but this is easy.  By \eqref{Tid22},
\begin{equation} \label{I4}
D_4 \leq \frac{4\{c(\g)+4\}}{d}\Vert Q_k\Vert^2\Vert Q_l\Vert^2.  
\end{equation}
Combining \eqref{t11terms} and \eqref{I1}--\eqref{I4} yields
\[
\E\{(t_{11}^{kl})^2\} \leq \frac{8\{108c(\g)^2 + 763c(\g) + 930\}}{d} \Vert
Q_k\Vert^2 \Vert Q_l\Vert^2.
\]
The lemma follows. 
\end{proof}

\begin{lemma}\label{lemma:s2_unif}
 Assume that the random variables $\b_1,\ldots,\b_p,\e_1,\ldots,\e_n$
 satisfy \eqref{subgbd}.  
Additionally, define
\begin{equation}\label{omega}
\omega(\L) = \frac{1}{(\l_1+1)^2(\l_{n_0}^{-1}+1)^2}.
\end{equation}
\begin{itemize}
\item[(a)] There is an absolute constant $0 < C < \infty$ such that
\[
\P\left\{\sup_{0 \leq \eta^2 < \infty}  \left\vert \s_*^2(\eta^2) -
     \s_0^2(\eta^2)\right\vert > 
  r\Bigg\vert X\right\} \leq C\exp\left[-\frac{n}{C}\min\left\{\frac{r^2}{\g^4}\omega(\L)^2,\frac{r}{\g^2}\omega(\L)\right\}\right]
\]
for all $r \geq 0$.
\item[(b)] Assume that $n_0 = n$.  There is an absolute constant $0 <
  C < \infty$ such that 
\[
\P\left\{\sup_{0 \leq \eta^2 < \infty}  \eta^2\left\vert \s_*^2(\eta^2) -
     \s_0^2(\eta^2)\right\vert > 
  r\Bigg\vert X\right\}\leq
C\exp\left[-\frac{n}{C}\min\left\{\frac{r^2}{\g^4} \omega(\L)^2,\frac{r}{\g^2}\omega(\L)\right\}\right]
\]
for all $r \geq 0$.  
\end{itemize}
\end{lemma}

\begin{proof}
Define $\bz = (p^{1/2}\bb^{\T},\be^{\T})^{\T} \in
\R^{p+n}$ and $t_i(\eta^2) = (\eta^2\l_i+1)^{-1}$.
 Then
\[
\s_*^2(\eta^2)  = \frac{1}{n}\bz^{\T}Q(\eta^2)\bz,
\] 
where $Q(\eta^2) = VT(\eta^2)V^{\T}$, $T(\eta^2) =
\diag\{t_1(\eta^2),\ldots,t_n(\eta^2)\}$ and
\[
V = \left[\begin{array}{c} p^{-1/2}X^{\T} \\ I \end{array}\right]U.  
\]
Thus, $\s_*^2(\eta^2)$ can be expressed as a quadratic form and we can
apply Theorem \ref{thm:QF_conc}with $K = 1$.  

We prove part (b) of the lemma first.  Assume that $n_0 = n$ and notice
\begin{align}\nonumber
\P\Bigg\{\sup_{0 \leq \eta^2 < \infty} \eta^2\big\vert \s_*^2(\eta^2) &-
    \s_0^2(\eta^2)\big\vert  > 
  r\Bigg\vert X\Bigg\}  \\ \nonumber & = \P\left[\left.\sup_{0 \leq \eta^2 < \infty} \eta^2\vert\bz^{\T}Q(\eta^2)\bz -
  \E\{\bz^{\T}Q(\eta^2)\bz\vert X\}\vert > ns\right\vert X\right]  \\ \label{pbbd1}
                         & \leq P_1^- + P_1^+,
\end{align}
where
\begin{align*}
P_1^- & = \P\left[\left.\sup_{0 \leq \eta^2 \leq 1} \eta^2\vert\bz^{\T}Q(\eta^2)\bz -
  \E\{\bz^{\T}Q(\eta^2)\bz\vert X\}\vert > nr\right\vert X\right], \\
P_1^+ & = \P\left[\left.\sup_{1 \leq \eta^2 < \infty} \eta^2\vert\bz^{\T}Q(\eta^2)\bz -
  \E\{\bz^{\T}Q(\eta^2)\bz\vert X\}\vert > nr\right\vert X\right].
\end{align*}
We'll apply Theorem \ref{thm:QF_conc} twice,
to $P_1^-$ and $P_1^+$ separately.  In order to apply Theorem
\ref{thm:QF_conc} to $P_1^-$, we need to derive a Lipschitz bound,
as in \eqref{lipschitz}.  For $0 \leq u^2,v^2 \leq 1$, 
\[
\vert u^2t_i(u^2) - v^2t_i(v^2) \vert = \left\vert\frac{u^2}{u^2\l_i + 1} -
                                  \frac{v^2}{v^2\l_i + 1}\right\vert 
                                  \leq  \vert u^2 - v^2\vert.
\]
Additionally, for $\eta^2 \geq 0$, we have the bounds 
\begin{equation}\label{other_bds}
\Vert VV^{\T} \Vert \leq \l_1 + 1, \ \ \ \Vert \eta^2T(\eta^2) \Vert
\leq \frac{\eta^2}{\eta^2\l_{n_0}+1} , \
\ \ \Vert \eta^2T(\eta^2)\Vert_{\HS}^2 \leq
\left(\frac{\eta^2}{\eta^2\l_{n_0}+1}\right)^2 n,
\end{equation}
where we have used the fact that $\l_n = \l_{n_0} > 0$.  
Thus, Theorem \ref{thm:QF_conc} implies that there is a constant
$0 < C_1^- < \infty$ such that
\begin{equation}\label{pbminus}
P_1^- \leq C_1^- \exp\left[-\frac{n}{C_1^-}\min\left\{\frac{r^2}{\g^4(\l_1+1)^2},\frac{r}{\g^2(\l_1+1)}\right\}\right],
\end{equation}
whenever $r^2 \geq C_1^-\g^4(\l_1+1)^2n^{-1}$.

We can't immediately apply Theorem \ref{thm:QF_conc} to bound $P_1^+$,
because the supremum inside the probability is over a non-compact
interval.  However, observe that $P_1^+$ can be rewritten as
\[
P_1^+ = \P\left[\left.\sup_{0 < \eta^2 \leq 1} \eta^{-2}\vert
  \bz^{\T}Q(\eta^{-2})\bz - \E\{\bz^{\T}Q(\eta^{-2})\bz\vert X\}\vert
  > nr \right\vert X\right].
\]
Now we can apply Theorem \ref{thm:QF_conc} as soon as we derive the
required Lipschitz bound.  For $0 < u^2,v^2 \leq 1$,
\[
\vert u^{-2}t_i(u^{-2}) - v^{-2}t_i(v^{-2})\vert = \left\vert
  \frac{1}{\l_i + u^2} - \frac{1}{\l_i + v^2}\right\vert \leq
\l_{n_0}^{-2}\vert u^2 - v^2\vert,
\]
where we have again used the fact that $\l_n = \l_{n_0} > 0$.  
Combining this with \eqref{other_bds} and Theorem \ref{thm:QF_conc}
implies that these exists a constant $0 < C_1^+ < \infty$ such
that
\begin{equation}\label{pbplus}
P_1^+ \leq C_1^+
\exp\left[-\frac{n}{C_1^+}\min\left\{\frac{r^2}{\g^4(\l_1+1)^2(\l_{n_0}^{-1}+1)^4},\frac{r}{\g^2(\l_1+1)(\l_{n_0}^{-1}+1)^2}\right\}\right],
\end{equation}
whenever $r^2 \geq C_1^+\g^4(\l_1+1)^2(\l_{n_0}+1)^4n^{-1}$.  Part (b) of
the lemma follows by combining \eqref{pbbd1} and
\eqref{pbminus}--\eqref{pbplus}.

To prove part (a), drop the assumption that $n_0 = n$.  Our proof
strategy is the same as in part (b), but the
proof is easier because  we
don't need to worry about whether or not $\l_n = 0$.  Briefly, observe that for all
$\eta^2 \geq 0$, $\Vert T(\eta^2) \Vert \leq 1, \ \ \Vert
T(\eta^2)\Vert_{\HS}^2 \leq n$.  Additionally, 
if $0 \leq
u^2,v^2 \leq 1$, then
$\vert t_i(u^2) - t_i(v^2)\vert  \leq \l_1\vert u^2 - v^2 \vert$ and
$\vert t_i(u^{-2}) - t_i(v^{-2}) \vert  \leq \l_{n_0}^{-1}\vert u^2 - v^2\vert$.
Proceeding just as in the proof of part (b), it follows that there are constants
$0 < C_2^-,C_2^+ < \infty$ such that 
\begin{align*}
\P\Bigg\{\sup_{0 \leq \eta^2 \leq1} \big\vert \s_*^2(\eta^2) &-
    \s_0^2(\eta^2)\big\vert  > 
  r\Bigg\vert X\Bigg\}  \\
& \leq C_2^-\exp\left[-\frac{n}{C_2^-}\min\left\{\frac{r^2}{\g^4(\l_1+1)^4},\frac{r}{\g^2(\l_1+1)^2}\right\}\right],
  \\
\P\Bigg\{\sup_{1 \leq \eta^2 < \infty} \big\vert \s_*^2(\eta^2) &-
    \s_0^2(\eta^2)\big\vert  > 
  r\Bigg\vert X\Bigg\} \\
& \leq C_2^+\exp\left[-\frac{n}{C_2^+}\min\left\{\frac{r^2}{\g^4(\l_1+1)^2(\l_{n_0}^{-1}+1)^2},\frac{r}{\g^2(\l_1+1)(\l_{n_0}^{-1}+1)}\right\}\right], 
\end{align*}
whenever $r^2 \geq C_2^-\g^4(\l_1+1)^4n^{-1}$ and $r \geq
C_2^+\g^4(\l_1+1)^2(\l_{n_0}^{-1}+1)^2n^{-1}$, respectively.  This implies part (a) of the
lemma.  
\end{proof}

\begin{lemma}\label{lemma:H_bd}
Let $H_0(\eta^2) = \E\{H_*(\eta^2)\vert X\}$, where $H_*(\eta^2)$ is
given in \eqref{score}.  For $\eta^2 \geq 0$,
\begin{equation}\label{hid}
H_0(\eta^2) = \frac{\s_0^2}{2n^2} \sum_{i,j=1}^n
  \frac{(\eta_0^2 - \eta^2)(\l_i -
  \l_j)^2}{(\eta^2\l_i+1)^2(\eta^2\l_j+1)^2} ,
\end{equation}
where $H_0(\eta^2) = \E\{H_*(\eta^2)\vert X\}$.
Hence, 
\begin{equation}\label{hbd}
\vert H_0(\eta^2)\vert \geq \frac{\s^2_0\vert\eta_0^2 -
  \eta^2\vert}{(\eta^2 \l_1 + 1)^4}\vfrak(\L).  
\end{equation}
\end{lemma}

\begin{proof}
The inequality \eqref{hbd} follows from \eqref{hid} and the identity
\[
\frac{1}{2n^2}\sum_{i,j=1}^n (\l_j - \l_i)^2 = \frac{1}{n}\sum_{i =
  1}^n \l_i^2 - \frac{1}{n^2}\sum_{i,j=1}^n \l_i\l_j  = \vfrak(\L).  
\]  
To prove \eqref{hid}, rewrite $H_0(\eta^2)$ as
\begin{align*}
H_0(\eta^2) & = \frac{\s_0^2}{n}\sum_{i = 1}^n \frac{\l_i(\eta_0^2\l_i +
  1)}{(\eta^2\l_i + 1)^2} - \s_0^2\left(\frac{1}{n}\sum_{i = 1}^n
  \frac{\l_i}{\eta^2\l_i + 1}\right) \left(\frac{1}{n}\sum_{i = 1}^n
  \frac{\eta_0^2\l_i + 1}{\eta^2\l_i + 1}\right) \\
& = \frac{\s_0^2}{n^2}\sum_{i,j = 1}^n \left(\frac{\eta_0^2\l_i +
  1}{\eta^2\l_i + 1}\right)\left(\frac{\l_i}{\eta^2\l_i + 1}  -
  \frac{\l_j}{\eta^2\l_j + 1}\right) \\
& = -\frac{\s_0^2}{n^2}\sum_{i,j = 1}^n \left(\frac{\eta_0^2\l_j +
  1}{\eta^2\l_j + 1}\right)\left( 
  \frac{\l_i}{\eta^2\l_i + 1}- \frac{\l_j}{\eta^2\l_j + 1}\right),
\end{align*} 
where the last expression is obtained by interchanging the indices
$i,j$ in the previous expression. Adding the last two expressions yields
\begin{align*}
2H_0(\eta^2) & = \frac{\s_0^2}{n^2}\sum_{i,j = 1}^n \left(\frac{\eta_0^2\l_i +
  1}{\eta^2\l_i + 1} - \frac{\eta_0^2\l_j +
  1}{\eta^2\l_j + 1}\right)\left(\frac{\l_i}{\eta^2\l_i + 1}  -
  \frac{\l_j}{\eta^2\l_j + 1}\right) \\
& = \frac{\s_0^2}{n^2} \sum_{i,j=1}^n
  \frac{(\eta_0^2 - \eta^2)(\l_i -
  \l_j)^2}{(\eta^2\l_i+1)^2(\eta^2\l_j+1)^2} .
\end{align*}
Equation \eqref{hid} follows.  
\end{proof}

\begin{lemma}\label{lemma:sld}
Let $H_0(\eta^2) = \E\{H_*(\eta^2)\vert X\}$, where $H_*(\eta^2)$ is
given in \eqref{score}.  Additionally, assume that the random variables
$\b_1,\ldots,\b_p,\e_1,\ldots,\e_n$ satisfy \eqref{subgbd}. 
There is an absolute constant $0 < C < \infty$ such that 
\[
\P\left\{\left.\sup_{0 \leq \eta^2 < \infty} \vert H_*(\eta^2) - H_0(\eta^2)\vert
  > r\right\vert X\right\} \leq C \exp\left[-\frac{n}{C}\min\left\{\frac{r^2}{\g^4(\l_1+1)^{6}},\frac{r}{\g^2(\l_1+1)^3}\right\}\right],
\]
for all $r \geq 0$.  
\end{lemma}

\begin{proof}  The proof is similar to that of Lemma \ref{lemma:s2_unif}.
 Let $\bz =
(p^{1/2}\bb^{\T},\be^{\T})^{\T} \in \R^{p + n}$ and rewrite 
\[
H_*(\eta^2)  = \frac{1}{n}\bz^{\T}Q(\eta^2) \bz,
\]
where $Q(\eta^2) = VT(\eta^2)V^{\T}$, $T(\eta^2) = \diag\{t_1(\eta^2),\ldots,t_n(\eta^2)\}$, 
\begin{align*}
t_i(\eta^2) & = \frac{1}{n}\sum_{j =
  1}^n\frac{\l_i - \l_j}{(\eta^2\l_i + 1)^2(\eta^2\l_j+1)}, \ \ i =
              1,\ldots,n, \\
V & = \left[\begin{array}{c} p^{-1/2}X^{\T} \\ I \end{array}\right]U.
\end{align*}
Then
\begin{align}\nonumber
\P\Bigg\{\sup_{0 \leq \eta^2 < \infty} \vert H_*(\eta^2) - H_0(\eta^2)\vert
  > r\Bigg\vert X\Bigg\} & = \P\left[\left.\sup_{0 \leq \eta^2 < \infty} \vert \bz^{\T}Q(\eta^2)\bz - \E\{\bz^{\T}Q(\eta^2)\bz\}\vert
  > n r\right\vert X\right] \\ \label{p1ad1}
& = P_1 + P_2,
\end{align}
where
\begin{align*}
P_1 & = \P\left[\left.\sup_{0 \leq \eta^2 \leq 1} \vert \bz^{\T}Q(\eta^2)\bz - \E\{\bz^{\T}Q(\eta^2)\bz\}\vert
  > nr\right\vert X\right], \\
P_2 & = \P\left[\left.\sup_{1 \leq \eta^2 <\infty} \vert \bz^{\T}Q(\eta^2)\bz - \E\{\bz^{\T}Q(\eta^2)\bz\}\vert
  > nr\right\vert X\right].
\end{align*}
We bound $P_1$ and $P_2$ separately, using Theorem \ref{thm:QF_conc}.  

In order to apply Theorem
\ref{thm:QF_conc}, again we need to check the Lipschitz condition
\eqref{lipschitz} and get bounds for $\Vert V^{\T}V\Vert$, $\Vert
T(0) \Vert$, and $\Vert T(0)\Vert_{\HS}^2$.   If $0 \leq u^2,v^2 \leq 1$, then we have the Lipschitz bound
\begin{align}\nonumber
\vert t_{i}(u^2) - t_{i}(v^2) \vert & = \left\vert\frac{1}{n}\sum_{j =
  1}^n\frac{\l_i - \l_j}{(u^2\l_i + 1)^2(u^2\l_j+1)}-\frac{1}{n}\sum_{j =
  1}^n\frac{\l_i - \l_j}{(v^2\l_i + 1)^2(v^2\l_j+1)} \right\vert \\ \nonumber
& \leq \frac{1}{n}\sum_{j = 1}^n \frac{\l_i^2\l_j\vert \l_i -\l_j\vert
  \vert v^6 - u^6\vert }{(u^2\l_i +
  1)^2(u^2\l_j+1) (v^2\l_i + 1)^2(v^2\l_j+1)}\\ \nonumber 
& \qquad + \frac{1}{n}\sum_{j = 1}^n \frac{ (\l_i^2 + 2\l_i \l_j)\vert \l_i -\l_j\vert
 \vert v^4
  - u^4 \vert }{(u^2\l_i +
  1)^2(u^2\l_j+1) (v^2\l_i + 1)^2(v^2\l_j+1)}\\ \nonumber
& \qquad + \frac{1}{n}\sum_{j = 1}^n \frac{ (2\l_i + \l_j)\vert \l_i -\l_j\vert
 \vert v^2 - u^2\vert}{(u^2\l_i +
  1)^2(u^2\l_j+1) (v^2\l_i + 1)^2(v^2\l_j+1)}\\ \label{lip_bd1}
& \leq 24\l_1^2\vert u^2 - v^2 \vert,
\end{align}
for $i = 1,\ldots,n$.  
Additionally, for $\eta^2 \geq 0$, 
\begin{equation}\label{other_bds_H}
\Vert V^{\T}V \Vert  \leq \l_1 + 1, \ \ \Vert T(\eta^2)\Vert  \leq 2\l_1, \ \ \Vert
T(\eta^2)\Vert_{\HS}^2 
\leq 4\l_1^2n.
\end{equation}
Thus, Theorem \ref{thm:QF_conc} and \eqref{lip_bd1}--\eqref{other_bds_H}
imply that there is a constant $0 < C_1 < \infty$ such that 
\begin{equation}\label{p1ab1}
P_1 \leq C_1 \exp\left[-\frac{n}{C_1}\min\left\{\frac{r^2}{\g^4(\l_1+1)^{6}},\frac{r}{\g^2(\l_1+1)^3}\right\}\right],
\end{equation}
whenever $r^2 \geq C_1\g^4(\l_1+1)^{6}n^{-1}$.  

Turning our attention to $P_2$, we have
\begin{equation}\label{p1ap2}
P_2 = \P\left[\left.\sup_{0 < \eta^2 \leq 1} \vert
  \bz^{\T}Q(\eta^{-2})\bz - \E\{\bz^{\T}Q(\eta^{-2})\bz\}\vert > r\right\vert X\right].
\end{equation}
We aim to apply Theorem \ref{thm:QF_conc} again, but first need the following
Lipschitz bound:  
\begin{align}\nonumber
\vert t_i(u^{-2}) - t_i(v^{-2})\vert & = \left\vert\frac{1}{n}\sum_{j
                                       = 1}^n \frac{u^6(\l_i -
                                       \l_j)}{(\l_i + u^2)^2(\l_j + u^2)} - \frac{1}{n}\sum_{j
                                       = 1}^n \frac{v^6(\l_i -
                                       \l_j)}{(\l_i + v^2)^2(\l_j +
                                       v^2)} \right\vert \\ \nonumber
& \leq \frac{1}{n}\sum_{j = 1}^n \frac{\l_i^2\l_j\vert \l_i - \l_j\vert\vert u^6
  - v^6\vert}{(\l_i+u^2)^2(\l_j+u^2) (\l_i+v^2)^2(\l_j+v^2)} \\ \nonumber
& \qquad + \frac{1}{n}\sum_{j = 1}^n \frac{u^2v^2(\l_i^2 + 2\l_i\l_j)\vert \l_i - \l_j\vert\vert u^4
  - v^4\vert}{(\l_i+u^2)^2(\l_j+u^2) (\l_i+v^2)^2(\l_j+v^2)} \\ \nonumber
& \qquad + \frac{1}{n}\sum_{j = 1}^n \frac{u^4v^4(2\l_i + \l_j)\vert \l_i - \l_j\vert\vert u^2
  - v^2\vert}{(\l_i+u^2)^2(\l_j+u^2) (\l_i+v^2)^2(\l_j+v^2)} \\ \label{lip_bd2}
& \leq 24\vert u^2 - v^2\vert,
\end{align}
for $0 \leq u^2,v^2 \leq 1$, $i = 1,\ldots,n$.  Now apply Theorem \ref{thm:QF_conc}, using
\eqref{other_bds_H} and \eqref{lip_bd2}, to conclude that there is a constant $0 <
C_2 < \infty$ such that
\begin{align} \nonumber
\P\Bigg[\sup_{0 < \eta^2 \leq 1} &\vert
  \bz^{\T}Q(\eta^{-2})\bz - \E\{\bz^{\T}Q(\eta^{-2})\bz\}\vert >
  r\Bigg\vert X\Bigg] \\ \label{p1ab2} & \leq C_2 \exp\left[-\frac{n}{C_2}\min\left\{\frac{r^2}{\g^4(\l_1+1)^{6}},\frac{r}{\g^2(\l_1+1)^3}\right\}\right],
\end{align}
whenever $r^2 \geq C_2\g^4(\l_1+1)^{6}n^{-1}$.  The lemma follows from
\eqref{p1ad1}, \eqref{p1ab1}--\eqref{p1ap2}, and \eqref{p1ab2}.
\end{proof}

\begin{lemma}\label{lemma:llik_cty}
Let
\[
\chi(\eta_0^2,\L) = \frac{1}{2(\eta_0^2+1)^4(\l_1+1)^4(\l_{n_0}^{-1}+1)^2}.
\]
\begin{itemize}
\item[(a)] Suppose $n_0 = n$.  Then
\[
\ell_0(\eta^2_0) - \ell_0(\eta^2) \geq \frac{(\eta^2 -
  \eta_0^2)^2 \chi(\eta_0^2,\L)}{(\vert
  \eta^2-\eta_0^2\vert+1)^2}\cdot \vfrak(\L), \ \ \ 
\eta^2,\eta_0^2 \geq 0.  
\]
\item[(b)] If $n_0 < n$, then
\[
\ell_0(\eta_0^2) - \ell_0(\eta^2) \geq \frac{(\eta^2 -
  \eta_0^2)^2 \chi(\eta_0^2,\L)}{(\vert
  \eta^2-\eta_0^2\vert+1)^2} \cdot\left(1 - \frac{n_0}{n}\right) \frac{n_0}{n}, \ \ \ 
\eta^2,\eta_0^2 \geq 0.  
\]
\end{itemize}
\end{lemma}

\begin{proof}
Let $T(\eta^2) = \eta^2a + b$, where $a,b \geq 0$ are nonnegative numbers to be specified further
below, and note that
\begin{equation}\label{lemma_diff0}
\ell_0(\eta_0^2) - \ell_0(\eta^2)  =
\frac{1}{2}\log\left\{\frac{T(\eta^2)}{n}\sum_{i=1}^n \frac{\eta_0^2 \l_i +
                                    1}{\eta^2\l_i+1}\right\} -\frac{1}{2n}\sum_{i = 1}^n \log
                                    \left\{T(\eta^2)\left(\frac{\eta_0^2 \l_i +
                                    1}{\eta^2\l_i+1}\right) \right\}.
\end{equation}
By Taylor's theorem, 
\begin{align} \nonumber
\log
                                    \left\{T(\eta^2)\left(\frac{\eta_0^2 \l_i +
                                    1}{\eta^2\l_i+1}\right) \right\}& \leq \log\left\{\frac{T(\eta^2)}{n}\sum_{i=1}^n \frac{\eta_0^2 \l_i +
                                    1}{\eta^2\l_i+1}\right\} \\ \nonumber
& \qquad + \left(\frac{1}{n}\sum_{i=1}^n \frac{\eta_0^2 \l_i +
                                    1}{\eta^2\l_i+1}\right)^{-1}\left(\frac{\eta_0^2 \l_j +
                                    1}{\eta^2\l_j+1} - \frac{1}{n}\sum_{i=1}^n \frac{\eta_0^2 \l_i +
                                    1}{\eta^2\l_i+1}\right) \\ \label{taylor}
& \qquad - \frac{1}{2h_j^2}\left(\frac{\eta_0^2 \l_j +
                                    1}{\eta^2\l_j+1} - \frac{1}{n}\sum_{i=1}^n \frac{\eta_0^2 \l_i +
                                    1}{\eta^2\l_i+1}\right)^2T(\eta^2)^2,
\end{align}
where $h_j \geq 0$ is between 
\[
T(\eta^2)\left(\frac{\eta_0^2 \l_i +
                                    1}{\eta^2\l_i+1}\right) \mbox{ and
                                } \frac{T(\eta^2)}{n}\sum_{i=1}^n \frac{\eta_0^2 \l_i +
                                    1}{\eta^2\l_i+1}.
\]
Now let $h\geq 0$ be any number satisfying $\max_{j = 1,\ldots,n} h_j
\leq h$.  Summing from $j =1,\ldots,n$ in \eqref{taylor} and plugging this in to
\eqref{lemma_diff0} yields
\begin{equation}\label{lemma_diff1}
\ell_0(\eta_0^2) - \ell_0(\eta^2)  \geq
                                    \frac{T(\eta^2)^2}{2h^2}\Delta(\eta^2,\eta_0^2),         
\end{equation}
where
\begin{align*}
\Delta(\eta^2,\eta_0^2) & = \frac{1}{n}\sum_{j =1}^n \left(\frac{\eta_0^2 \l_j +
                                    1}{\eta^2\l_j+1} - \frac{1}{n}\sum_{i=1}^n \frac{\eta_0^2 \l_i +
                                    1}{\eta^2\l_i+1}\right)^2 \\
& = \frac{1}{n}\sum_{i =1}^n \left(\frac{\eta_0^2 \l_i +
                                    1}{\eta^2\l_i+1}\right)^2 - \left(\frac{1}{n}\sum_{i=1}^n \frac{\eta_0^2 \l_i +
                                    1}{\eta^2\l_i+1}\right)^2 \\
& = \frac{\eta_0^2-\eta^2}{n^2}\sum_{i,j=1}^n
  \frac{(\eta_0^2\l_i+1)(\l_i - \l_j)}{(\eta^2\l_i+1)^2(\eta^2\l_j+1)}
  \\
& = -\frac{\eta_0^2-\eta^2}{n^2}\sum_{i,j=1}^n
  \frac{(\eta_0^2\l_j+1)(\l_i - \l_j)}{(\eta^2\l_i+1)(\eta^2\l_j+1)^2}.
\end{align*}
Adding the last two expressions above and dividing by two, we obtain
\[
\Delta(\eta^2,\eta_0^2)  = \frac{(\eta^2 - \eta_0^2)^2}{2n^2}\sum_{i,j=1}^n
    \frac{(\l_i - \l_j)^2}{(\eta^2\l_i+1)^2(\eta^2\l_j+1)^2}.  
\]
Thus, combining this with \eqref{lemma_diff1}, it follows that
\begin{equation}\label{master_bd}
\ell_0(\eta_0^2) - \ell_0(\eta^2) \geq \frac{(\eta^2a+b)^2(\eta^2 -
  \eta_0^2)^2}{4h^2n^2}\sum_{i,j=1}^n \frac{(\l_i - \l_j)^2}{(\eta^2\l_i+1)^2(\eta^2\l_j+1)^2}.
\end{equation}
Now we consider the cases where $n_0=n$ and $n_0<n$ separately.  

Suppose that $n_0=n$ and let $a = \l_1$, $b =1$ in $T(\eta^2)$.
Then we can take $h =
(\eta_0^2+1)(\l_1+1)^2(\l_n^{-1}+1)$, and \eqref{master_bd} implies
\begin{align*}
\ell_0(\eta_0^2) - \ell_0(\eta^2) & \geq \frac{(\eta^2 -
  \eta_0^2)^2}{2(\eta^2\l_1+1)^2(\eta_0^2+1)^2(\l_1+1)^4(\l_n^{-1}+1)^2}\vfrak(\L)
  \\
&  \geq \frac{(\eta^2 - \eta_0^2)^2}{2(\vert \eta^2 - \eta_0^2\vert+1)^2(\eta_0^2+1)^4(\l_1+1)^6(\l_n^{-1}+1)^2}\vfrak(\L).
\end{align*}
Part (a) follows.  

Now assume that $n_0 < n$.  Let $a=0$, $b=1$ and $h=(\eta_0^2+1)(\l_1 + 1)$. Then, by \eqref{master_bd},
\begin{align*}
\ell_0(\eta^2_0) - \ell_0(\eta^2) &
                                  \geq \frac{(\eta^2 -
                                    \eta_0^2)^2}{4(\eta_0^2+1)^2(\l_1+1)^2n^2
                                    } \sum_{i,j=1}^{n_0}
                                    \frac{(\l_i-\l_j)^2}{(\eta^2\l_i+1)^2(\eta^2\l_j+1)^2}
  \\
&\qquad 
                                    + \frac{(\eta^2 -
                                    \eta_0^2)^2(n-n_0)}{2(\eta_0^2+1)^2(\l_1+1)^2n^2
                                    } \sum_{i = 1}^{n_0}
  \frac{\l_i^2}{(\eta^2\l_i+1)^2} \\
& \geq \frac{(\eta^2 -
                                    \eta_0^2)^2}{2 (\vert \eta^2 - \eta_0^2\vert+1)^2 (\eta_0^2+1)^4(\l_1+1)^4(\l_{n_0}^{-1}+1)^2
                                    } \cdot \left( 1- \frac{n_0}{n}\right)\frac{n_0}{n}.
\end{align*}
This implies part (b).  
\end{proof}

\begin{lemma}\label{lemma:llik_dev}
Assume that the random variables
$\b_1,\ldots,\b_p,\e_1,\ldots,\e_n$ satisfy \eqref{subgbd} and let $\omega(\L)$ be as defined in \eqref{omega}. 
\begin{itemize}
\item[(a)] Suppose that $n_0 = n$.  There exists an absolute constant
  $0 < C<  \infty$ such that
\begin{align*}
\P\Bigg\{\sup_{0 \leq \eta^2 < \infty} & \vert \ell_*(\eta^2) -
  \ell_0(\eta^2)\vert > r\Bigg\vert X\Bigg\} \\ & \leq  C\exp\left[-\frac{n}{C}\cdot \frac{\s_0^4\omega(\L)^2}{\g^2(\g^2+1)(\s_0^2+1)(\l_1+1)^2}\min\left\{r^2,r,1\right\}\right]
\end{align*}
for all $r \geq 0$.  
\item[(b)] Suppose that $n_0 < n$. There exists an absolute constant
  $0 < C<  \infty$ such that
\begin{align*}
\P\Bigg\{\sup_{0 \leq \eta^2 < \infty}  & \vert \ell_*(\eta^2) -
  \ell_0(\eta^2)\vert > r\Bigg\vert X\Bigg\} \\ 
& \leq C\exp\left[-\frac{n}{C}\cdot
  \frac{\s_0^4\omega(\L)^2}{\g^2(\g^2+1)(\s_0^2+1)}\left(1 - \frac{n_0}{n}\right)^2\min\left\{r^2,r,1\right\}\right]
\end{align*}
for all $r \geq 0$.  
\end{itemize}
\end{lemma}

\begin{proof}
To prove this lemma, we use Lemma \ref{lemma:s2_unif}.  First notice that
\[
\ell_*(\eta^2) - \ell_0(\eta^2) =
\frac{1}{2}\log\{\s_0^2(\eta^2)\} -
\frac{1}{2}\log\{\s_*^2(\eta^2)\}.  
\]
Next, assume that $n = n_0$.  Then
\[
(\eta^2+1)\s_0^2(\eta^2) = \frac{\s_0^2(\eta^2+1)}{n}\sum_{i = 1}^n \frac{\eta_0^2\l_i +
  1}{\eta^2\l_i + 1} \geq \frac{\s_0^2}{n}\sum_{i = 1}^n \frac{\eta_0^2\l_i +
  1}{\l_i + 1} \geq \frac{\s_0^2}{\l_1 + 1}.
\]
It follows that
\begin{align*}
\Bigg\{\sup_{0 \leq \eta^2 < \infty} & \vert \ell_*(\eta^2) -
  \ell_0(\eta^2)\vert > r\Bigg\} \\ & = \left\{\sup_{0 \leq \eta^2 <
                                    \infty} \left\vert \log\{\s_0^2(\eta^2)\}
                                    -\log\{\s_*^2(\eta^2)\}\right\vert >
                                    2r\right\}\\
& =\left\{\sup_{0 \leq \eta^2 <
                                    \infty} \log\{\s_0^2(\eta^2)\}
                                    -\log\{\s_*^2(\eta^2)\}>
                                    2r\right\} \\
& \qquad \cup \left\{\sup_{0 \leq \eta^2 <
                                    \infty} \log\{\s_*^2(\eta^2)\}-\log\{\s_0^2(\eta^2)\}
                                    >
                                    2r\right\} \\
& \subseteq \left\{\sup_{0 \leq \eta^2 <
                                    \infty}\frac{\s_0^2(\eta^2) - \s_*^2(\eta^2)}{\s_*^2(\eta^2)} >
                                    2r\right\} \cup \left\{\sup_{0 \leq \eta^2 <
                                    \infty} \frac{\s_*^2(\eta^2) - \s_0^2(\eta^2)}{\s_0^2(\eta^2)}>
                                    2r\right\} \\
& \subseteq  \left\{\sup_{0 \leq \eta^2 <
                                    \infty} (\eta^2 + 1)\left\vert \s_0^2(\eta^2)- \s_*^2(\eta^2)\right\vert >
                                    \frac{4\s_0^2 r}{\l_1+1}\right\}
  \\
& \qquad \cup \left\{\sup_{0 \leq \eta^2 < \infty}
  (\eta^2 + 1)\s_*^2(\eta^2) < \frac{\s_0^2}{2(\l_1+1)} \right\} \\
& \subseteq \left\{\sup_{0 \leq \eta^2 <
                                    \infty} (\eta^2 + 1)\left\vert \s_0^2(\eta^2)- \s_*^2(\eta^2)\right\vert>
                                    \frac{4\s_0^2 r}{\l_1+1}\right\}
  \\
& \qquad \cup \left\{\sup_{0 \leq \eta^2 < \infty}
   (\eta^2 + 1)\left\vert \s_0^2(\eta^2)-
  \s_*^2(\eta^2)\right\vert > \frac{\s_0^2}{2(\l_1+1)} \right\}.
\end{align*}
Thus, by Lemma \ref{lemma:s2_unif} (a)--(b), there is a constant $0 < C
<\infty$ such that
\begin{align*}
\P\Bigg\{\sup_{0 \leq \eta^2 < \infty}  \vert \ell_*(\eta^2) -
  \ell_0(\eta^2)\vert > r\Bigg\vert X\Bigg\}  & \leq
C\exp\left[-\frac{n}{C}\min\left\{\frac{\s_0^4 \omega(\L)^2 }{\g^4(\l_1+1)^2}r^2,\frac{\s_0^2\omega(\L)}{\g^2(\l_1+1)}r\right\}\right]
                                 \\
& \qquad +
  C\exp\left[-\frac{n}{C}\min\left\{\frac{\s_0^4 \omega(\L)^2}{\g^4(\l_1+1)^2},\frac{\s_0^2 \omega(\L)}{\g^2(\l_1+1)}\right\}\right],
\end{align*}
whenever $r^2 \geq C \g^4 \s_0^{-4}(\l_1+1)^2\omega(\L)^{-2}n^{-1}$.
Part (a) of the lemma
follows.  

To prove part (b) of the lemma, assume that $n_0 < n$.  Then
$\s_0^2(\eta^2)\geq \s_0^2(\eta_0^2\l_{n_0}+1)(1 - n_0/n)$.
Similar to the proof of part (a), it follows that
\begin{align*}
\left\{\sup_{0 \leq \eta^2 < \infty} \vert \ell_*(\eta^2) -
  \ell_0(\eta^2)\vert > r\right\} & \subseteq \left\{\sup_{0 \leq \eta^2 <
                                    \infty} \left\vert \s_0^2(\eta^2)- \s_*^2(\eta^2)\right\vert>
                                    4\s_0^2\left(1 - \frac{n_0}{n}\right)r\right\}
  \\
& \qquad \cup \left\{\sup_{0 \leq \eta^2 < \infty}
 \left\vert \s_0^2(\eta^2)- \s_*^2(\eta^2)\right\vert>
  \frac{\s_0^2}{2}\left(1 - \frac{n_0}{n}\right) \right\}.
\end{align*}
Thus, by Lemma  \ref{lemma:s2_unif} (a), there is a constant $0 < C
<\infty$ such that
\begin{align*}
\P\Bigg\{\sup_{0 \leq \eta^2 < \infty}  \vert \ell_*(\eta^2) -&
  \ell_0(\eta^2)\vert > r\Bigg\vert X\Bigg\} \\ & \leq
C\exp\left[-\frac{n}{C}\min\left\{\frac{\s_0^4\omega(\L)^2}{\g^4}\left(1 - \frac{n_0}{n}\right)^2r^2,\frac{\s_0^2\omega(\L)}{\g^2}\left(1 - \frac{n_0}{n}\right)r\right\}\right]
                                 \\
& \qquad +
  C\exp\left[-\frac{n}{C}\min\left\{\frac{\s_0^4\omega(\L)^2}{\g^4}\left(1
  - \frac{n_0}{n}\right)^2,\frac{\s_0^2\omega(\L)}{\g^2}\left(1 -
  \frac{n_0}{n}\right)\right\}\right], 
\end{align*}
whenever $r^2 \geq C\g^4 \s_0^{-4}\omega(\L)^{-2}n(n-n_0)^{-2}$.
This implies part (b). 
\end{proof}

\begin{lemma}\label{lemma:vqf}
Let $Q = (q_{ij})$ be a $d \times d$ positive semidefinite matric and let $\bz =
(\zeta_1,\ldots,\zeta_d)^{\T} \in \R^d$ be a random vector with
independent components that have mean zero and variance 1.  Let $\mu_j^{(k)}
= \E(\vert\zeta_j\vert^k)$ and assume that $\mu_j^{(4)} < \infty$.
Finally, define $\bm_4 = (\mu_1^{(4)},\ldots,\mu_d^{(4)})^{\T}$, $\q_k =
(q_{11}^k,\ldots,q_{dd}^k)^{\T}$.  Then
\begin{equation}\label{vqf}
\Var(\bz^{\T}Q\bz)  = \bm_4^{\T}\q_2 - 3\Vert\q\Vert^2 + 2\tr(Q^2)
\leq \left[2 + \max\{\mu_j^{(4)}\}\right]d\Vert Q\Vert^2.
\end{equation}
\end{lemma}

\begin{proof}
The inequality in \eqref{vqf} is obvious.  To
prove the equality, we have
\begin{align*}
\Var(\bz^{\T}Q\bz) & = \E\left\{(\bz^{\T}Q\bz)^2\right\} -
                     \E\left\{\bz^{\T}Q\bz\right\}^2 \\
& = \E\left\{\left(\sum_{i,j=1}^d
  q_{ij}\zeta_i\zeta_j\right)^2\right\}- \tr(Q)^2 \\
& = \E\left(\sum_{i,j,k,l=1}^d
  q_{ij}q_{kl}\zeta_i\zeta_j\zeta_k\zeta_l\right) - \tr(Q)^2 \\
& = \E\left\{\sum_{i = 1}^n q_{ii}^2\zeta_i^4\right\}
  +2\E\left\{\sum_{1 \leq i < j \leq d} (2q_{ij}^2 +
  q_{ii}q_{jj})\zeta_i^2\zeta_j^2\right\}  - \tr(Q)^2 \\
& = \bm_4^{\T}\q_2 + 2\sum_{i\neq j} q_{ij}^2 +
  \sum_{i \neq j} q_{ii}q_{jj} - \tr(Q)^2 \\
& = \bm_4^{\T}\q_2 - 3\Vert\q\Vert^2 + 2\tr(Q^2),
\end{align*}
as was to be shown.  
\end{proof}

\end{document}